\documentclass[english, 11pt]{amsart}
\usepackage{amssymb}
\usepackage{tikz}
\usepackage{tikz-cd}
\usepackage[colorlinks, linkcolor=red]{hyperref}
\usepackage{xcolor}
\usepackage{geometry}
\definecolor{immi}{rgb}{0,.6,.1}

\newbox\removebox
\newcommand\remove[2]{%
\setbox\removebox=\ifmmode\hbox{$#2$}\else\hbox{#2}\fi%
\leavevmode
\rlap{\textcolor{#1}{\vrule height0.8ex depth-0.6ex width\wd\removebox}}%
\box\removebox
}
\long\def\bigremove#1{%
\par\setbox\removebox=\vbox{#1}%
\vbox{%
\vbox to0pt{\hbox{\tikz\draw[color=blue,thick] (0,0) -- (\wd\removebox,-\ht\removebox)  (\wd\removebox,0) -- (0,-\ht\removebox);}}
\box\removebox
}
}

\usepackage{mathrsfs} 

\newcommand{\loc}{{\mathrm{loc}}}

\makeatletter
\def\RFss@@#1{\RF^*_{\!*#1}}
\def\RFss@_#1{\RFss@@{,#1}}
\def\RFss{\@ifnextchar_{\RFss@}{\RFss@@{}}}
\makeatother

\newcommand{\RF}{{\rm RF}}

\def\Supp{{\operatorname{Supp}}}

\def\moi{\operatorname{moi}}

\def\Div{\operatorname{Div}}

\def\Sing{\operatorname{Sing}}

\def\deg{\operatorname{deg}}

\def\acc{{{\rm ac}}}

\def\Proj{\operatorname{Proj}}

\def\11{{\mathbf 1}}
\def\AA{{\mathbb A}}

\def\CC{{\mathbb C}}

\def\FF{{\mathbb F}}
\def\GG{{\mathbb G}}

\def\NN{{\mathbb N}}

\def\PP{{\mathbb P}}
\def\QQ{{\mathbb Q}}
\def\RR{{\mathbb R}}

\def\ZZ{{\mathbb Z}}

\def\cA{{\mathcal A}}
\def\cB{{\mathcal B}}

\def\cD{{\mathcal D}}
\def\cE{{\mathcal E}}

\def\cG{{\mathcal G}}
\def\cH{{\mathcal H}}
\def\cI{{\mathcal I}}
\def\cJ{{\mathcal J}}

\def\cN{{\mathcal N}}
\def\cO{{\mathcal O}}

\def\cS{{\mathcal S}}

\def\cZ{{\mathcal Z}}

\def\sfF{\mathsf{F}}

\newcommand{\spec}{\operatorname{Spec}}

\newtheorem{thm}{Theorem}[section]
\newtheorem{lem}[thm]
{Lemma}
\newtheorem{lem-defn}[thm]
{Lemma-Definition}
\newtheorem{cor}[thm]
{Corollary}
\newtheorem{prop}[thm]
{Proposition}
\newtheorem{conj}[thm]
{Conjecture}
{Problem}
\newtheorem{maintheorem}
{Main Theorem}
\newtheorem{example}[thm]
{Example}
\theoremstyle{plain}
\newtheorem*{namedthm}{\namedthmname}
\newcounter{namedthm}


\theoremstyle{definition}
\newtheorem{defn}[thm]
{Definition}
\newtheorem{notn}[thm]
{Notation}

\theoremstyle{remark}
\newtheorem{remark}[thm]
{Remark}

\theoremstyle{plain}

  {\par\medskip\noindent #1\par\begingroup%
    \advance\leftskip by 1em\advance\rightskip by 1em}%
  {\par\endgroup}

\DeclareMathOperator*{\lcm}{lcm}
\DeclareMathOperator*{\Spec}{Spec}

\DeclareMathOperator*{\rank}{rank}

\newcommand{\ord}{\operatorname{ord}}

\def\cI{\mathcal{I}}
\def\cJ{\mathcal{J}}
\def\cD{\mathcal{D}}
\def\cA{\mathcal{A}}

\def\cG{\mathcal{G}}

\def\cO{\mathcal{O}}

\renewcommand{\and}{ \quad \text{and} \quad }

\begin{document}

\setcounter{tocdepth}{2} 

\author[Kien~H.~Nguyen]
{Kien Huu Nguyen}
\address{KU Leuven, Department of Mathematics,
Celestijnenlaan 200B, B-3001 Leu\-ven, Bel\-gium}
\email{kien.nguyenhuu@kuleuven.be}
\address{Institute of Mathematics, Vietnam Academy of Science and Technology, 18 Hoang Quoc Viet, Nghia Do, Hanoi, Vietnam}
\email{nhkien@math.ac.vn}
\urladdr{https://sites.google.com/site/nguyenkienmath/home}
\author[Takehiko Yasuda]{Takehiko Yasuda}
\address{Kavli Institute for the Physics and Mathematics of the Universe, The
University of Tokyo, 5-1-5 Kashiwanoha, Kashiwa, Chiba, 277-8583,
Japan}
\address{Department of Mathematics, Graduate School of Science, the University of Osaka, Toyonaka, Osaka 560-0043, JAPAN} \email{yasuda.takehiko.sci@osaka-u.ac.jp}
\title[]{On Igusa's conjectures for weighted homogeneous polynomials with isolated singularities}
\begin{abstract}
In 1991, Denef introduced an explicit formula for Igusa's local zeta functions in terms of  embedded resolutions of singularities over $p$-adic fields with good reduction. 
In this paper, by slightly generalizing the $p$-adic McKay correspondence given by the second author in 2017, we extend Denef's result to the setting of $\QQ$-embedded resolutions of singularities.  Consequently, the strong monodromy conjecture with respect to Igusa's local zeta functions  holds for weighted homogeneous polynomials with isolated singularities. Moreover, we use our formula to prove Igusa's conjecture for exponential sums in the case of non-degenerate weighted homogeneous polynomials with isolated singularities.
\end{abstract}
\maketitle

\tableofcontents

\section{Introduction}

\subsection{Background}

Let $f\in\ZZ[x_1,...,x_n]$ be a non constant polynomial. For each positive integer $N$, an important arithmetic problem is to estimate the number of solutions of the congruence equation 
$$
f(x)\equiv 0 \mod N$$
in terms of $N$. By the Chinese remainder theorem, the problem reduces to estimating the number of solutions of the congruence equation
\begin{equation}\label{pm}
f(x)\equiv 0 \mod p^m
\end{equation}
for each prime $p$ and each positive integer $m$. Let $N_{p,m}$ be the number of solutions of Equation (\ref{pm}). Borevich and Shafarevich conjectured in \cite{Bosha} that the Poincaré series
$$P_p(T)=\sum_{m\geq 1}N_{p,m}p^{-mn}T^m$$
is a rational function of $T$. This conjecture was proved by Igusa in \cite{Igusa1,Igusa2} by using Hironaka's embedded resolution of singularities of the pair $(\AA_{\QQ_p}^n,\{f=0\})$. Denef also gave another proof of this conjecture by using model theoretic tools. The proof of Igusa is related to the theory of Igusa's local zeta functions 
$$\cZ_{f,\Phi,\chi}(s)=\int_{\QQ_p^n}\Phi(x)\chi(\acc(f(x)))|f(x)|^s|dx|,$$
where $\Re(s)>0$, $\acc(z)= zp^{-\ord_p(z)}$ if $z\in\QQ_p\setminus\{0\}$ and $\acc(0)=0$, $\ord_p:\QQ_p\to\ZZ\cup\{+\infty\}$ is the valuation map, $\Phi$ is a locally constant function with compact support on $\QQ_p^n$, $|z|=p^{-\ord_p(z)}$, $\chi:(\ZZ_p^{\times},.)\to (\CC^{\times},.)$ is a continuous  homomorphism of topological groups and $|dx|$ is the Haar measure on $\QQ_p^n$ normalized such that the volume of $\ZZ_p^n$ is $1$. In fact, Igusa proved that $\cZ_{f,\Phi,\chi}(s)$ is a rational function of $p^{-s}$. Moreover, we have 
$$P_p(p^{-s})= \frac{1-p^{-s}\cZ_{f,\textbf{1}_{\ZZ_p^n},\chi_{\textnormal{triv}}}(s))}{1-p^{-s}},$$
where $\chi_{\textnormal{triv}}:(\ZZ_p^{\times},.)\to (\CC^{\times},.), u\mapsto 1$ and $\textbf{1}_{\ZZ_p^n}$ is the characteristic function of $\ZZ_p^n$. To know about the size of $N_{p,m}$, it is sufficient to know  all poles of $\cZ_{f,\textbf{1}_{\ZZ_p^n},\chi_{\textnormal{triv}}}(s)$ and their multiplicity (see, e.g.,  \cite[Corollary 2.4]{Loeasym}, \cite[Corollary 2.9]{VeysZ} and \cite{NguyenLW}).  
The strong monodromy conjecture is stated in \cite{IGUSA88} that the real part of every pole of $\cZ_{f,\textbf{1}_{\ZZ_p^n},\chi}(s)$ is a root of the Bernstein-Sato polynomial of $f$ (see Conjecture \ref{conjmono} in Subsection \ref{Igex}). This conjecture is still widely open in the general case.  Until now, we only know that this conjecture holds for a few special cases such as the case of two variables polynomials (see \cite{Loemono}). Motivic version of Igusa's local zeta functions (called motivic zeta functions) was introduced by Denef and Loeser in \cite{DLJAG} and \cite{DL02}. We can also state the strong monodromy conjecture for  these motivic zeta functions. Moreover, for each prime $p$,  there is a $p$-specialization map $s_p$ that sends motivic zeta functions associated with $f$ over $\QQ$ to  Igusa's local zeta functions associated with $f$ and $p$ provided that $p$ is large enough. In particular, $s_p$ sends a motivic zeta function (called the naive motivic zeta function) associated with $f$ over $\QQ$ to $\cZ_{f,\textbf{1}_{\ZZ_p^n},\chi_{\textnormal{triv}}}(s)$ if $p$ is large enough. 

The fact is that we can study poles of Igusa's local zeta functions by using exponential sums modulo $p^m$. 
Let $\Phi$ be a locally constant function with compact support on $\QQ_p^n$ and $z\in\QQ_p$. We associated to $\Phi, z,f$ the exponential sum
$$E_{f,\Phi,\psi_z}=\int_{\QQ_p^n}\Phi(x)\psi_z(f(x))|dx|,$$
where $\psi:\QQ_p\to\CC^{\times}$ is the standard additive character of $\QQ_p$ defined by $\psi(u)=\exp(2\pi i\tilde{u})$ if $\tilde{u}$ is any rational number in $(\QQ_p\setminus\ZZ_p)\cup\{0\}$ 
such that $u-\tilde{u}\in\ZZ_p$ and $\psi_z:\QQ_p\to\CC^{\times}$ is the additive character of $\QQ_p$ defined by $\psi_z(u)=\psi(zu)$. When $z=p^{-m}$ then 
$$E_{f,\textbf{1}_{\ZZ_p^n},\psi_z}=p^{-mn}\sum_{x\in(\ZZ/p^m\ZZ)^n}\exp\left(\frac{2\pi if(x)}{p^m}\right)$$
is the usual exponential sum modulo $p^m$ of $f$.  Let $C_f$ be the set of critical points of $f$. Suppose that the support $\Supp(\Phi)$ of $\Phi$ has an empty intersection with $C_f\setminus f^{-1}(0)$. Let $z\in \QQ_p$ such that $|z|$ is large enough, then $E_{f,\Phi,\psi_z}$ is a finite $\CC$-linear combination of functions of the form $\chi(\acc(z))|z|^{\lambda}|\ord_p(z)|^{\beta}$ with coefficients independent of $z$, with $\lambda\in\CC$ a pole of   $(p^{s+1}-1) \cZ_{f,\Phi,\chi_{\textnormal{triv}}}(s)$ if $\chi=\chi_{\textnormal{triv}}$ or $\cZ_{f,\Phi,\chi}(s)$ if $\chi\neq\chi_{\textnormal{triv}}$, and with $\beta+1$ a positive integer less than the multiplicity of the pole $\lambda$. Moreover, all poles $\lambda$ as above appear with non-zero coefficients 
in this linear combination (see \cite{Igusa3} or \cite[Corollary 1.4.5]{DenefBour}). Therefore,  understanding $E_{f,\Phi,\psi_z}$ also helps to better understand the poles of Igusa's local zeta functions associated with $f$ and $\Phi$. For instance, for each prime $p$, let $\sigma_{f,p}$ be the minimal real number $\sigma$ such that $-\sigma$ is the real part of a pole of $(p^{s+1}-1) \cZ_{f,\11_{\ZZ_p^n},\chi_{\textnormal{triv}}}(s)$ or $\cZ_{f,\11_{\ZZ_p^n},\chi}(s)$ for $\chi\neq\chi_{\textnormal{triv}}$. Let $\sigma_f=\liminf_{p \to+\infty}\sigma_{f,p}$. Then $\sigma_f$ is the maximal real number such that if $\sigma<\sigma_f$ then for every  large enough prime $p$, $(p^{s+1}-1) \cZ_{f,\11_{\ZZ_p^n},\chi_{\textnormal{triv}}}(s)$ and $\cZ_{f,\11_{\ZZ_p^n},\chi}(s)$ for $\chi\neq\chi_{\textnormal{triv}}$ have no pole in the right half plane $\{\Re(s)>-\sigma\}$.   More especially, $\sigma_f$ is characterized by Igusa's conjecture for exponential sums. Roughly speaking, this conjecture is stated that $\sigma_f$ is the maximal real number such that for all $\sigma<\sigma_f$, there is a constant $c(\sigma,f)$ depending only on $f,\sigma$ and an integer $M$ depending only on $f$  such that for all $p>M$ we have 
$$|E_{f,\textbf{1}_{\ZZ_p^n},\psi_z}|\leq c(\sigma,f)|z|^{-\sigma}$$
for all $z\in \QQ_p$ such that $\ord_p(z)\leq -2$ (see \cite{CMN,NguyenVeys} and Conjecture \ref{conjex} in Subsection \ref{Igex}). This conjecture relates to the existence of integer solutions of the Diophantine equation $f=0$ when $f$ is homogeneous as mentioned by Igusa in \cite{Igusa3} (see also \cite{Haris,Nguyennsd}). More precisely, Igusa expected that this conjecture can help to give some simple sufficient conditions for the validity of the principle (so called the local-global principle or Hasse's principle) that a (resp. weighted homogeneous) polynomial $f$ with integer coefficients has a (resp. non-zero) integer solution if and only if $f$ has a (resp. non-zero) solution over the real field $\RR$ and a (resp. non-zero) solution over $\QQ_p$ for all primes $p$. Igusa's conjecture for exponential sums was proved partially in \cite{CDenSper,CDenSperlocal,JWright,Lichtin2,CMN,CAN,NguyenVeys,CluckerNguyen,Nguyennsd} but has been opened in the general case.  On the other hand, exponential sums modulo $p^m$ and Igusa's conjecture for exponential sums also has further applications on singularity theory such as characterizing special properties of singularities (see \cite{CMN,MoiIde}). 

\subsection{Main results}

Fix $w=(w_1,...,w_n)\in\NN_{>0}^n$ with $\gcd(w_1,...,w_n)=1$  and regard the polynomial ring $\ZZ[x_1,...,x_n]$ as an \(\NN\)-graded ring by putting \(\deg x_i = w_i\). Let $f\in\ZZ[x_1,...,x_n]$ be a non-constant polynomial which is homogeneous relative to this grading.
We assume that  $f\neq 0$, the $w$-degree of $f$ is $d>1$ and that  $f$ defines a hypersurface with an isolated singularity in $\AA_\CC^n$. In this case, the origin $(0,...,0)$ of $\AA_\CC^n$ is the only singularity of $\{f=0\}$. Let $p$ be a prime, $\chi$ be a multiplicative character  of $\ZZ_p^{\times}$ of order $d(\chi)$ and $Y$ be a subscheme of $\AA_{\ZZ}^n$. One of the main goals of this paper is to prove the following two theorems.

\begin{maintheorem}\label{mthm1}
For a large enough prime $p$ and $\Phi_{p,Y}=\11_{\{x\in \ZZ_p^n\mid (x \mod p)\in Y(\FF_p)\}}$,  
the set \(A\) of  the real parts of poles of 
\[
\cZ_{f,\Phi_{p,Y},\chi}(s)
\]
 satisfies the following properties:
\begin{enumerate}
	\item  If $d(\chi)\neq 1$ and $d(\chi)\nmid d$ then $A=\emptyset$.
	\item If $d(\chi)\neq 1$ and $d(\chi)\mid d$ then $A\subset \{\frac{-(w_1+...+w_n)}{d}\}$. 
	\item If $d(\chi)=1$ then $A\subset \{-1,\frac{-(w_1+...+w_n)}{d}\}$.	 
\end{enumerate}
Moreover, all poles of  
$\cE_{p,\chi}(s)\cZ_{f,\Phi_{p,Y},\chi}(s)$ have order at most $1$, where $\cE_{p,\chi}=1$ if $d(\chi)>1$ and $\cE_{p,\chi}=(p^{s+1}-1)$ if $d(\chi)=1$. 
\end{maintheorem}

We set 
$$\kappa(w)=\max_{\emptyset\neq J\subsetneq\{1,...,n\}, j\notin J} \lcm(\gcd(w_i)_{i\in J},w_j).$$

\begin{maintheorem}\label{mthm2}
Suppose $d\geq \kappa(w)$. 
There is a natural number $M$ depending only on $f$ and a constant $c$ depending only on $n,w,d$ such that 
$$|E_{f,\Phi_{p,Y},\psi_z}|\leq c|z|^{-\frac{\sum_{i=1}^nw_i}{d}}$$
 for every subscheme $Y\subset \AA_\ZZ^n$, every prime $p>M$ and every $z\in\QQ_p$ with $\ord_p(z)\leq -2$. 
Moreover, when $Y=\AA_\ZZ^n$, this inequality holds for every prime $p>M$ and every $z\in\QQ_p^{\times}$. 

\end{maintheorem}

In the situation of Main Theorem \ref{mthm1}, $\frac{-(w_1+...+w_n)}{d}$ is indeed a root of the Bernstein-Sato polynomial of $f$ as shown in \cite{BBV}. Thus, the strong monodromy conjecture holds for $f$. Moreover, Main Theorem \ref{mthm2} confirms Igusa's conjecture for exponential sums in the case of weighted homogeneous polynomials with isolated singularities under a slight condition on their degrees and the weight $w$. In particular, this condition always holds in the key case of the local-global principle for weighted homogeneous polynomials with isolated singularities, as demonstrated in the following example.
\begin{example}Let $g\in\ZZ[x_1,...,x_n]$ be a non-constant weighted homogeneous polynomial. If there is a non-empty subset $J$ of $\{1,...,n\}$ such that $g|_{\{x_i=0\forall i\notin J\}}$ is the zero polynomial, then we say that $g$ is degenerate. Otherwise, we say that $g$ is non-degenerate. 

Suppose that $g$ is degenerate  and $J$ is a non-empty subset of $\{1,...,n\}$ such that $g|_{\{x_i=0\forall i\notin J\}}$ is the zero polynomial.  Then every tuple $(a_1,...,a_n)\in \ZZ^n\setminus\{0,...,0\}$ such that $a_i=0$ for all $i\notin J$ is a non-zero solution of $g(x)=0$. Thus, the local-global principle holds trivially for $g$.

Suppose that $g$ is non-degenerate.  For each $1\leq j\leq n$, $g|_{\{x_i=0\forall i\neq j\}}$ is a non-zero polynomial. Therefore, $w_j\mid d$ and the term $x_j^{d/w_j}$ appears with a non-zero coefficient in $g$ for all $1\leq i\leq n$. This implies $\kappa(w)\mid d$ since $\lcm(\gcd(w_i)_{i\in J},w_j)\mid d$ for all $\emptyset\neq J\nsubseteq\{1,...,n\}$ and all $j\notin J$. Thus, Main Theorem \ref{mthm2} works for $g$.

\end{example}

To remove the condition in Main Theorem \ref{mthm2}, we might need to study in depth the relationship between the degree $d$ and the weight $w$. 
More precisely, we hope to better understand the set $\cD_w$ of natural numbers $d$ for each $w\in\NN_{\geq 1}^n$ such that there is a $w$-weighted homogeneous polynomial of degree $d$ in $n$ variables that  defines a hypersurface with an isolated singularity in $\AA_\CC^n$. Motivated by our calculation, we give a weaker condition than the one in Main Theorem \ref{mthm2} and predict that every $d\in \cD_w$ satisfies this condition. This is formulated as a conjecture  at the end of this paper (see Conjecture \ref{conjw}).  

One of the main points to prove
both Main Theorem \ref{mthm1} and Main Theorem \ref{mthm2} is to give an explicit formula for Igusa's local zeta functions $\cZ_{f,\Phi_{p,Y},\chi}$ when $p$ is large enough. It is based on the fact that the $w$-weighted blowing-up at the origin of $\AA_{\QQ}^n$ is an embedded $\QQ$-resolution of singularities of $(\AA_\QQ^n,\{f=0\})$ as shown in \cite{BBV} (see Section \ref{wiso}). Here, the notion of embedded $\QQ$-resolutions of singularities was introduced in \cite{ArMO}. In fact, we are able to establish an explicit formula for Igusa's local zeta functions associated with a given polynomial $g\in\ZZ[x_1,...,x_n]$  and a prime $p$ in terms of any given embedded $\QQ$-resolution of singularities of $(\AA_\QQ^n,\{g=0\})$ provided that $p$ is large enough. This is the content of Main Theorem \ref{mucase}  presented in Section \ref{form}. We hope that our formula will have further applications to the study of Igusa's local zeta functions in the future.

\subsection{Related works}
Let $f\in\CC[x_1,...,x_n]$  be a $w$-weighted homogeneous with respect to a weight $w\in\NN_{>0}^n$ such that $f$ defines a hypersurface with an isolated singularity in $\AA_\CC^n$. In \cite{Sebagiso}, Sebag proved that the motivic monodromy conjecture holds for the naive motivic zeta function associated with $f$ by explicitly calculating jet schemes associated with $f$. On the other hand, Blanco, Budur and van der Veer also gave alternative proof of this result in \cite{BBV} by using the work  in \cite{LMVV} that allows us to compute explicitly the naive motivic zeta function through a given embedded $\QQ$-resolution of singularities. However, since \cite{LMVV} works with the naive motivic zeta function of $f$ over $\CC$, it is not automatic to use the map $s_p$ above to compute $\cZ_{f,\textbf{1}_{\ZZ_p^n},\chi_{\textnormal{triv}}}(s)$ when $f\in\ZZ[x_1,...,x_n]$ if $p$ is large enough. Although it seems to be quite complicated,  the readers may try to give a version of \cite{LMVV}  for not only the naive motivic zeta function over $\QQ$ but also the twisted motivic zeta functions given in \cite{DLJAG} over $\QQ$ and use the ring homomorphism $s_p$ mentioned above to prove Main Theorem \ref{mthm1}.

Note that we can compute all Igusa's local zeta functions by using embedded resolutions of singularities as shown in \cite{Igusa3}. Moreover,  if $g\in\ZZ_p[x_1,...,x_n]$ and we have an embedded resolution of singularities of $(\AA_{\QQ_p}^n,\{g=0\})$
with (tame) good reduction modulo $p$ as in \cite{Denefdegree,Denef91},  then Denef gave an explicit formula for Igusa's local zeta functions $\cZ_{g,\textbf{1}_{\ZZ_p^n},\chi}(s)$ in terms of this embedded resolution.  
Our Main Theorem \ref{mucase} gives a version of Denef's formula  in the setting of embedded $\QQ$-resolutions of singularities. 
In practice, the complexity to obtain an embedded $\QQ$-resolution of singularities is often  much lower than a usual embedded resolution of singularities. For instance, if $g$ is weighted homogeneous and defines a hypersurface with an isolated singularity in $\AA_\CC^n$ then we can obtain an embedded $\QQ$-resolution of singularities of $(\AA_\CC^n,\{g=0\})$ by using only one weighted blowing-up at the origin as mentioned above while an embedded resolution of singularities of $(\AA_\CC^n,\{g=0\})$ is very complicated to compute. Moreover, the recent work of Abramovich, Temkin, and W\l odarczyk in \cite{ATW} on embedded resolution of singularities using weighted blowing-ups may be a machine to construct embedded $\QQ$-resolutions of singularities. Therefore, we have room  to apply  Main Theorem \ref{mucase} in future research.  
\subsection{The strategy and outline of the paper}

To explain our strategy, we recall the notion of embedded $\QQ$-resolutions of singularities in \cite{ArMO}. Roughly speaking, for the usual embedded resolutions of singularities as in \cite{Hir:Res}, we require that the final ambient space is smooth while for the embedded $\QQ$-resolutions of singularities, we allow the final ambient space to contain abelian quotient singularities (see Section \ref{bano} for details). In other words, an embedded  $\QQ$-resolution of singularities can be obtained by using a sequence of weighted blowing-ups rather than a sequence of blowing-ups for embedded resolutions of singularities. 

In this paper, we study integrals around abelian quotient singularities. The main ingredient  is the work of the second author in \cite{Yas} where he introduced the notion of $G$-string point count that is the sum of some special integrals associated to a given finite group $G$ and varieties over $\ZZ_p$ endowed with an action of $G$.  This notion is used to study the McKay correspondence that links some special properties of quotient singularities and the representations of finite groups. The main point is that if $G$ is a finite group with a faithful action on an affine space then the $G$-stringy point count of this affine space equals to the $\{1\}$-stringy point count of the quotient of this affine space by $G$, where $\{1\}$ is the trivial group. On the other hand, the $\{1\}$-stringy point count is a birational invariant. Our idea is that the Igusa's zeta functions with respect to the trivial character can be computed in terms of the $\{1\}$-stringy point count and this holds for all primes $p$ thus it is not too surprise that we can use \cite{Yas} to compute Igusa's zeta function with respect to the trivial character when we have an embedded $\QQ$-resolution of singularities. However, for Igusa's zeta function with respect to non-trivial characters, we need to extend the notion of stringy point count by introducing the notion of twisted stringy point count. Roughly speaking, we need to endow non-trivial characters to the setting of stringy point count.  On the other hand,  to extend the work of Denef in \cite{Denefdegree,Denef91}, we will also need a version of the stringy point count associated with finite group schemes rather than finite groups. In fact, by using the idea from \cite[Section 1.4]{DLJAG}, the readers can relate our twisted stringy point count to the twisted point count introduced by Groechenig, Wyss and Ziegler in \cite[Definition 2.14]{GWZ}. All details on stringy point count will be presented in Section \ref{stringy}. More specifically, we also compute the $G$-string point count in the case where $G\subset\GG_{m,\ZZ_p}^n$ is a finite diagonalizable subgroup scheme over $\ZZ_p$. This will be helpful in proving Main Theorem \ref{mucase}.

Note that the notion of string point count is defined with $\ZZ_p$-varieties. Thus, to use an embedded $\QQ$-resolutions of singularities over $\QQ_p$ to compute the string point count, we need to suppose at least that this embedded $\QQ$-resolution  can be defined over $\ZZ_p$. This is the reason to introduce the notion of embedded $\QQ$-resolutions of singularities over $\QQ_p$ having a good $\ZZ_p$-model in Section \ref{bano}. 
However, to extend the results of Denef in \cite{Denefdegree,Denef91} to the setting of embedded $\QQ$-resolutions  of singularities having a \emph{good $\ZZ_p$-model}, we need to require additionally that the $\ZZ_p$-models of these embedded $\QQ$-resolutions have \emph{good reduction modulo $p$} (see a precise definition in Section \ref{bano}).  In particular, let $f\in\ZZ[x_1,...,x_n]$ be as in Main Theorem \ref{mthm1} then by a simple  argument (see Lemma \ref{almostgood} below), we will see that $(\AA_{\QQ_p}^n,\{f=0\})$  admits an embedded $\QQ$-resolution of singularities with good reduction modulo $p$ if $p$ is large enough. Furthermore, an important problem is to construct a stratification and its associated numerical data of any given embedded $\QQ$-resolution of singularities with good reduction modulo $p$ such that we can express Igusa's local zeta functions in terms of them.  This problem is quite simple for the usual embedded resolutions of singularities since we have no non-trivial group scheme action here. In \cite{LMVV}, this problem is also not too difficult since we work with algebraically closed fields. In our situation, we need to deal with the ring of integers of $p$-adic fields rather than algebraically closed fields. Thus, it requires much more complicated work. We will present all details on solving this problem in Subsections \ref{weispa} and \ref{emQgood}.

In Section \ref{form}, if $g\in\ZZ_p[x_1,...,x_n]$ then for each embedded $\QQ$-resolutions of singularities of $(\AA_{\QQ_p}^n,\{g=0\})$ with good reduction modulo $p$, we will use the stringy point count to obtain a formula for Igusa's local zeta functions of $g$ in terms of any fixed stratification and its numerical data mentioned above. 

In Section \ref{wiso}, we will construct a stratification and its numerical data associated with the $w$-weighted blowing up at the origin mentioned above. This and Main Theorem \ref{mucase} allow us to compute Igusa's local zeta functions associated with weighted  homogeneous polynomials defining hypersurfaces with isolated singularities. In particular, this implies Main Theorem \ref{mthm1}. The proof of Main Theorem \ref{mthm2} is much more difficult. The easy part is to use the relation between exponential sums modulo $p^m$ and Igusa's local zeta functions mentioned above (see a precise statement in Proposition \ref{zetaex}) to compute explicitly exponential sums modulo $p^m$. The difficult part is to prove the inequality. By using results on exponential sums over finite fields in \cite{CDenSperlocal} and \cite{MoiIde}, the problem reduces to studying the  deep relation among the degree $d$ of $f$, the weight $w$ and the singular locus of polynomials $f|_{\{x_i=0\forall i\notin J\}}$ for each non-empty subset $J$ of $\{1,...,n\}$. We spend the last part of Subsection \ref{auxres} to explore these relations.   

\subsection*{Acknowledgements}

The first author contributed to the entire paper, while the second author contributed to Sections 1--3. The idea for this paper was formed during the first author's visit to the University of Osaka in 2024, supported by the Japan Society for the Promotion of Science (JSPS) Postdoctoral Research Fellowship under the grant number PE23034. We sincerely thank JSPS and the University of Osaka for their all-round support during this visit. The first author was partially supported by  Fund for Scientific Research-Flanders (Belgium) (F.W.O.) 1270923N and  the Excellence Research Chair “FLCarPA: L-functions in positive characteristic and applications” financed by the Normandy Region. The second author was supported by JSPS KAKENHI Grant Numbers 23K25767 and 24K00519.

\section{Settings and Preliminaries}\label{prep}
\subsection{Basic definitions and notations}\label{bano}
In this paper we will denote by $K$ a number field and $\cO_K$ its ring of integers. Moreover, a local field is a finite extension of $\QQ_p$ for a prime $p$. A local field over $\cO_K$ is a local field endowed with an $\cO_K$-algebra structure. 

Let $L$ be a local field over $\cO_K$. We denote by $\cO_L, \varpi_L,  k_L$ the ring of integers in $L$, a fixed uniformizing element of $\cO_L$ and the residue field $k_L:=\cO_L/(\varpi_L)$ of $L$ respectively. We denote by $\ord_L:L\to \ZZ\cup\{+\infty\}$ the valuation map of $L$. We put $q_L=\#k_L$ and denote by $p_L$ the characteristic of $k_L$. Let $\acc: L\to\cO_L^{\times}\cup\{0\}$ be  the angular component map with respect to $\varpi_L$ given by $\acc(z)=z\varpi^{-\ord_L(z)}$ if $z\neq 0$ and $\acc(0)=0$. We denote by $|.|_L:L\to \RR_{\geq 0}, z\mapsto q_L^{-\ord_L(z)}$ the usual absolute value on $L$. If $L$ is well understood in the context then we write $|.|$ instead of $|.|_L$.

A multiplicative character of $\cO_L^{\times}$ is a continuous homomorphism of groups from $\cO_L^{\times}$ to $\CC^{\times}$. We extend such a multiplicative character $\chi$ to a function on $\cO_L^{\times}\cup\{0\}$ by putting $\chi(0)=0$.  Let $\chi$ be a multiplicative character of $\cO_L^{\times}$. The conductor of $\chi$, denote by $c(\chi)$, is the smallest positive integer number $m$ such that $\chi(1+\varpi_L^m\cO_L)=1$. We will denote by $\chi_\textnormal{triv}$ the trivial multiplicative character of $\cO_L^{\times}$. We also denote by $d(\chi)$ the order of $\chi$, i.e., the smallest positive integer $\ell$ such that $\chi^{\ell}=\chi_\textnormal{triv}$.  If $\chi$ is a multiplicative character of $\cO_L^{\times}$ of conductor $1$ then $\chi$ induces a multiplicative character $\overline{\chi}$ of $k_L^{\times}$ given by $\overline{\chi}(a \mod \varpi_L\cO_L):=\chi(a)$, we also put $\overline{\chi}(0)=0$. 

An additive character of $L$ is a continuous homomorphism of groups from $L$ to $\CC^{\times}$. Let $\psi$ be a non-trivial additive character of $L$, the conductor of $\psi$, denote by $m_\psi$, is the smallest integer $m$ such that $\psi|_{\varpi_L^{m}\cO_L}=1$. If $\psi$ is an additive character of $L$ of conductor $0$ then any non-trivial additive character of $L$ of conductor $m$ is of the form 
\begin{equation}\label{psi_z}
    \psi_z:L\to \CC^{\times}, u\mapsto \psi(zu)
\end{equation}
for some $z\in L$ satisfying $\ord_L(z)=-m$.

Let $n$ be a positive integer,  we will denote by $|dx|:=|dx_1\wedge...\wedge dx_n|$ the Haar measure on $L^n$ such that the volume of $\cO_L^n$ is $1$. 

A Schwartz-Bruhat function on $L^n$ is a locally constant function with compact support.  A Schwartz-Bruhat function $\Phi$ on $L^n$ is residual if the support $\Supp(\Phi)$ of $\Phi$ is contained in $\cO_L^n$ and the value $\Phi(x)$ depends only on $\overline{x}:=x \mod \varpi_L\cO_L$. For a residual Schwartz-Bruhat function $\Phi$, we will set $\overline{\Phi}(\overline{x}):=\Phi(x)$. Let $Z$ be a subscheme of $\AA_{\cO_K}^n$ or $\AA_{k_L}^n$, we consider the Schwartz-Bruhat function $\Phi_{L,Z}$ on $L^n$ defined by
$\Phi_{L,Z}:=\textbf{1}_{\{z\in \cO_L^n| \overline{z} \in Z(k_L)\}}$. It is clear that $\Phi_{L,Z}$ is residual. 

In this paper, an $\cO_L$-variety is a flat and reduced quasi-projective $\cO_L$-scheme $X$ of equi-relative dimensional over $\spec(\cO_L)$ such that the structure map $X\to \spec(\cO_L)$  is smooth on an open dense subscheme of $X$. We say that $X$ is a $n$-dimensional $\cO_L$-variety if the relative dimensional of $X$ over $\spec(\cO_L)$ is $n$. If $f:X\to Y$ is a morphism of $\cO_L$-varieties, we denote by $\overline{f}:X_{k_L}\to Y_{k_L}$ the reduction modulo $\varpi_L\cO_L$ of $f$.

Throughout this paper, a ring is a commutative ring with $1\neq 0$.  Let $R$ be a ring. If $X, Y$ are $R$-schemes, $\varphi:X\to Y$ is an $R$-morphism and $R'$ is an $R$-algebra then we also often write $X_{R'}$ (resp.  $Y_{R'}$) instead of $X\otimes_R R'$ (resp.  $Y\otimes_R R'$) and denote by $\varphi_{R'}:X_{R'}\to Y_{R'}$ for the extension of scalars  to $R'$ of $\varphi$. By abuse of notation, if  $R\to R'$ is an embedding of rings and $f\in R[x_1,...,x_n]$ then we will still write $f$ for the image of $f$ via the embedding $R[x_1,...,x_n]\to R'[x_1,...,x_n]$ induced by the embedding $R\to R'$. Let $\mathsf{F}$ be a field.  We denote by $\overline{\mathsf{F}}$ a fixed algebraic closure of $\mathsf{F}$.   An $\mathsf{F}$-variety is a reduced separated $\mathsf{F}$-scheme of finite type. Let $X$ be an $\mathsf{F}$-variety and let $f$ be a regular function on $X$. We  will denote by $D_{\mathsf{F}}(f)$ the principal divisor on $X$ associated with $f$. Moreover, if the field $\mathsf{F}$ is understood clearly in the context then we will write $D(f)$ instead of $D_{\mathsf{F}}(f)$.  We use $\Sing(X)$ and $C_f$  to denote the singular locus of $X$ and $f:X\to\AA_\mathsf{F}^1$, respectively.

\subsection{Group schemes}
 Let $R$ be a domain, we write $\GG_{\textnormal{m},R}=\spec(R[x,y]/(1-xy))$ for the multiplicative group scheme over $R$.
 
Let $\mathsf{F}$ be a field and $G$ be a finite group scheme over $\mathsf{F}$.  We denote by $d(G)$ the order of $G$, i.e., the dimension over $\mathsf{F}$ of the Hopf algebra $\cO_G(G)$. To simplify notation, we will write $\cO_G$ instead of $\cO_G(G)$.

If $G$ is a finite flat group scheme over a Noether domain $R$ then $d:\spec R \to \NN, \mathfrak{p}\mapsto d(G\otimes_R R_{\mathfrak{p}}/\mathfrak{p})$ is a constant function whose value is denoted by $d(G)$. 

We need the following technical definitions.
\begin{defn}\label{etag}
Let $G$ be a finite étale group scheme over a field $\mathsf{F}$. A $G$-étale algebra is a finite $\mathsf{F}$-algebra $M$ of degree $d(G)$ endowed with a co-action of $\cO_G(G)$ such that its subring $M^G$ of $\cO_G(G)$-invariant elements is identical to $\mathsf{F}$. Alternatively, $M$ is a $G$-étale algebra if and only if $\spec(M)$ is an étale $G$-torsor over $\spec(\mathsf{F})$. We denote by $G\textnormal{-\'Et}(\mathsf{F})$ the set of étale $G$-algebras.
\end{defn} 

\begin{defn}[\cite{SGA3I}, Exposé I, Section 4.4]\label{diagon} Let $R$ be a domain. If $\cG$ is an abelian group of finite type, then the  functor
\begin{align*}\mathsf{D}_R(\cG):R\textnormal{-algebras}&\to \textnormal{Sets},\\ T&\mapsto \textnormal{Hom}_{\textnormal{groups}}(\cG,T^{\times})
\end{align*}
is presented by a flat affine group scheme of finite type over $R$. We still denoted this group scheme by $\mathsf{D}_R(\cG)$. In particular, if $\cG$ is finite then $\mathsf{D}_R(\cG)$ is a finite flat group scheme over $R$. A diagonalizable group scheme over $R$ (or $R$-diagonalizable group scheme) is a groups scheme over $R$ that is $R$-isomorphic with $\mathsf{D}_R(G)$ for some  abelian group $\cG$ of finite type.
\end{defn}
We have the following lemmas.
\begin{lem}Let $R,R'$ be domains such that $R'$ is endowed with an $R$-algebra structure. Let $\cG$ be an abelian group of finite type. Then $\mathsf{D}_{R'}(\cG)=\mathsf{D}_R\otimes_R R'$. 
\end{lem}
\begin{proof}
It follows immediately from Definition \ref{diagon}.
\end{proof}
\begin{lem}[\cite{SGA3I}, Exposé VIII, Corollary 5.7]\label{closimm}Let $R$ be a domain. Then a monomorphism between two diagonalizable group schemes over $R$ is a closed immersion. 
\end{lem}
\begin{lem}\label{finimugr}
Let $R$ be a domain. Then every finite diagonalizable group scheme over $R$ is $R$-isomorphic with  a finite fibre product over $R$  of $R$-schemes of the form $\mu_{l,R}:=\Spec R[x]/(x^l-1)$ for $l\in \NN_{\geq 1}$.
\end{lem}
\begin{proof}
By definition, it suffices to verify that if $\cG$ is a finite abelian group then $\mathsf{D}_R(\cG)$ is a finite fiber product over $R$ of group schemes of the form $\mu_{l,R}$. Indeed, since $\cG$ is a finite abelian group, we can write $\cG=\prod_{i=1}^r\ZZ/l_i\ZZ$ for positive integers $l_i$. Moreover, by Definition \ref{diagon},  if $\cG_1,\cG_2$ are two abelian groups of finite type then $\mathsf{D}_R(\cG_1\times\cG_2)=\mathsf{D}_R(\cG_1)\times_R\mathsf{D}_R(\cG_2)$. Thus, the assertion follows from the fact that $\mathsf{D}_R(\ZZ/l\ZZ)=\mu_{l,R}$ for all positive integers $l$. 
\end{proof}

\subsection{Weighted affine spaces and weighted projective spaces}\label{weispa}

For a ring \(R\) and for a tuple \(w=(w_1,\dots,w_n)\) of positive rational numbers, let \(R[x_1,\dots,x_n;w]\) denote the polynomial ring that is given with a \(\QQ_{\ge 0}\)-grading by putting \(\deg x_i =w_i\). It is written as the direct sum \(\bigoplus_{d\in \QQ_{\ge 0}} R[x_1,\dots,x_n;w]_d\) with \( R[x_1,\dots,x_n;w]_d\) denoting the degree-\(d\) part. Let  \(R[x_1,\dots,x_n;w]_\ZZ = \bigoplus_{d\in \NN} R[x_1,\dots,x_n;w]_d\) be the subalgebra generated by homogeneous elements of integer degrees.  

\begin{defn}
We define the \emph{\(w\)-weighted affine space over \(R\)}, denoted by \(\AA^w_R\), to be the affine scheme \(\Spec R[x_1,\dots,x_n;w]_\ZZ\). The \emph{\(w\)-weighted projective space over \(R\)}, denoted by \(\PP^w_R\), to be the projective  scheme \(\Proj R[x_1,\dots,x_n;w]_\ZZ\). 
\end{defn}

Note that the usual affine space \(\AA_R^n\) is the same as \(\AA_R^{(1,\dots,1)}\). Note also that for \((l_1,\dots,l_n)\in \ZZ_{>0}\), the weighted affine space  \(\AA^{(1/l_1,\dots,1/l_n)}_R\) is identified with the quotient scheme \(\AA_R ^n/G\), where \(G\) is the finite group scheme \(\prod _{i=1}^n \mu_{l_i,R}\subset \GG_{m,R}^n\) that has a naturally action on \(\AA_R ^n\).

\begin{defn}\label{dedivis}Let $R$ be a domain and $n$ be a positive integer $n$.  We denote by \(E_i :=D(x_i)\subset \AA_R^{n}\) the prime divisor defined by \(x_i\). For each subset $I$ of $\{1,...,n\}$, we set $E_I:=\cap_{i\in I} E_i$
 and $\overset{\circ}{E}_I:=E_I\setminus \cup_{i\notin I}E_i$.
Let $\mathfrak{p}$ be a prime ideal of $R$. We denote by  \(\AA_{R,\mathfrak{p}}^{n,\loc}\)  the affine scheme \(\Spec R[x_1,\dots,x_n]_{\mathfrak{p}+(x_1,\dots,x_n)}\) and also denote by \(E_i :=D(x_i)\subset \AA_{R,\mathfrak{p}}^{n,\loc}\) the prime divisor defined by \(x_i\). If $\mathfrak{p}=(0)$ then we write \(\AA_{R}^{n,\loc}\) instead of \(\AA_{R,\mathfrak{p}}^{n,\loc}\). 
\end{defn}

\begin{defn}[standard quotient varieties, standard pairs and standard triples]
For a domain \(R\),  
and an $R$-diagonalizable subgroup scheme \(G \)  of \(\GG_{m,R}^n\) that is finite over $R$, we define \(A_G:=\AA_R^{n}/G\), which we call a \emph{standard quotient scheme} (over \(R\)).
We denote by $\rho_G:\AA_R^n\to A_G$ the quotient morphism.  Let \(D_i\subset A_G\) denote the prime divisor defined as the image of \(E_i\) given with reduced structure. 

 Let $Y$ be an open subscheme of $A_G$ containing the image of $(x_1,...,x_n)$. 
We call $Y$ a \emph{standard scheme} over $R$ witnessed by $G$. We call a pair of the form \((Y,\sum_{i=1}^nm_iD_i|_Y)\) for $m_i\in\NN$ a \emph{standard pair} over $R$ witnessed by $G$.  Similarly, we call a triple of the form  \((Y,\sum_{i=1}^nm_iD_i|_Y,\sum_{i=1}^nl_iD_i|_Y)\) for $m_i,l_i\in\NN$ a \emph{standard triple} witnessed by $G$.  We will set $\rho_G^*(D):=\rho_G^{*}(\sum_{i=1}^n m_iD_i)$ if $(Y,D)$ is a standard pair over $R$ witnessed by $G$ and $D=\sum_{i=1}^nm_iD_i|_{Y}$. 
\end{defn}

\begin{defn}[locally standard quotient varieties, locally standard pairs and locally standard triples]
For a domain $R$, a prime ideal $\mathfrak{p}$ of $R$ and an  $R$-diagonalizable subgroup scheme \(G\)  of \(\GG_{m,R}^n\) that is finite over \(R\), we define \(A_{G,\mathfrak{p}}^{\rm loc}:=\AA_{R,\mathfrak{p}}^{n,\loc}/G\), which we call a \emph{locally standard quotient scheme} (over \(R\)). We can regard $A_{G,\mathfrak{p}}^{\rm loc}$ as the $R$-scheme associated with the local ring of $A_G$ at the image of the point $\mathfrak{p}+(x_1,...,x_n)$. If $R$ is a local ring and $\mathfrak{p}$ is its maximal ideal then we write $A_{G}^{\rm loc}$ instead of $A_{G,\mathfrak{p}}^{\rm loc}$. We also denote by $\rho_G:\AA_{R,\mathfrak{p}}^{n,\loc}\to A_{G,\mathfrak{p}}^{\rm loc}$ the quotient morphism.
Let 
\(D_i\subset A_{G,\mathfrak{p}}^{\rm loc}\) denote the prime divisor defined to be the image of \(E_i\) given with reduced structure.

We call  a pair of the form \((A_{G,\mathfrak{p}}^{\loc},\sum_{i=1}^nm_iD_i)\) for $m_i\in\NN$ a \emph{locally standard pair} witnessed by $G$. 

Similarly, we call a triple of the form  \((A_{G,\mathfrak{p}}^{\loc},\sum_{i=1}^nm_iD_i,\sum_{i=1}^nl_iD_i)\) for $m_i,l_i\in\NN$ a \emph{locally standard triple} witnessed by $G$.

\end{defn}
\begin{lem-defn}\label{factorunit}Let $R$ be a DVR ring. We fix a generator $t$ of the maximal ideal $\mathfrak{p}$ of $R$. Let $G$ be an $R$-diagonalizable subgroup scheme of $\GG_{m,R}^n$ such that $G$ is finite over $R$.  Let $(W,g)$ be a germ of a regular function on $A_G$ at the image of $\mathfrak{p}+(x_1,...,x_n)$. Suppose that  $g$ defines a divisor on $W$ whose support is contained in $W_{ R/\mathfrak{p}}\cup \left(\cup_{i=1}^n D_i\right)$. We can view $\cO_{A_G}(A_G)$ as a subring of $R[x_1,...,x_n]$ generated over $R$ by a finite set of monomials. Moreover, we can write $g=ut^l\prod_{i=1}^nx_i^{l_i}$ for natural numbers $l,l_i$ and a unit element $u\in \cO_{A_G}(W)$ such that $\prod_{i=1}^nx_i^{l_i}\in \cO_{A_G}(A_G)$. We call $u$ the standard  unit factor of $g$ on $W$.  We also call $\mathfrak{u}(g):=u(\rho_G(\mathfrak{p}+(x_1,...,x_n)))\in (R/\mathfrak{p})^{\times}$ the standard angular component of $g$ at $\rho_G(\mathfrak{p}+(x_1,...,x_n))$.

\end{lem-defn}
\begin{proof}Since $G$ is an $R$-diagonalizable  subgroup scheme of $\GG_{m,R}^n$ and $G$ is finite over $R$, the ring $R[x_1,...,x_n]^G$ of $G$-invariant elements in $R[x_1,...,x_n]$  is finitely generated over $R$ by monomials. 
Let $K(\AA_R^n)$ and $K(A_G)$ be the function fields of $\AA_R^n$ and $A_G$ respectively.   Then we can regard $K(A_G)$ as a subfield of $K(\AA_R^n)$ if we identify $\cO_{A_G}(A_G)$ with $R[x_1,...,x_n]^G$. Since $g\in K(A_G)\subset K(\AA_R^n)$, we can consider the divisor $\Div(g)$ on $\AA_R^n$. Let $l$ be the coefficient of $\Div(t)$ in $\Div(g)$ and $l_i$ be the coefficient of $E_i$ in $\Div(g)$ for $1\leq i\leq n$. Since $g$ is regular at $\mathfrak{p}+(x_1,...,x_n)$, we have $l\in\NN$ and $l_i\in \NN$ for all $i$. Moreover, since $g\in K(A_G)$ and $R[x_1,...,x_n]^G$ is finitely generated over $R$ by monomials, we have $\prod_{i=1}^nx_i^{l_i}\in R[x_1,...,x_n]^G=\cO_{A_G}(A_G)$. Thus, there are elements $g_1,g_2\in R[x_1,...,x_n]^G$ such that $\Div(g_1)$ and $\Div(g_2)$ are two divisors on $\AA_R^n$ that do not have any common irreducible component and $gg_2=g_1t^l\prod_{i=1}^nx_i^{l_i}$. Since $g$ is regular on $W$, $\Div(g)|_{\rho_G^{-1}(W)}$ is an effective divisor. This together with the definition of $l,l_i$ and the assumption that $\Div(g_1)$ and $\Div(g_2)$ have no common component imply that $g_2$ is a unit element in $\cO_{\AA_R^n}(\rho_G^{-1}(W))$. Therefore,  $u=\frac{g_1}{g_2}$ is a regular function on $W$ and $g=ut^l\prod_{i=1}^nx_i^{l_i}$. Since $g$ defines a divisor on $W$ whose support is contained in $W_{ R/\mathfrak{p}}\cup \left(\cup_{i=1}^n D_i\right)$, $g_1$ must be a unit element in $\cO_{A_G}(W)$. Thus, $u=\frac{g_1}{g_2}$ is a unit element in $\cO_{A_G}(W)$.
\end{proof}
\begin{defn}Let $R$ be a domain. 
A morphism of pairs  \((X_1,\cD_1)\) and \((X_2,\cD_2)\) over $R$ of an $R$-scheme and a divisor  is an $R$-morphism $\phi: X_1\to X_2$ such that $\phi^{*}(\cD_2)=\cD_1$.

We similarly define morphisms 
of triples  \((X_1,\cD_1,\cE_1)\) and \((X_2,\cD_2,\cE_2)\) over $R$ with \(\cE_1,\cE_2\) extra divisors. 

We can define an $R$-isomorphism of pairs (resp. triples), an $R$-étale morphism of pairs (resp. triples) 
in a natural way. 
\end{defn}
\begin{lem}\label{etalelocalquoti}
Let $R$ be a domain. Let $G$ be an $R$-diagonalizable subgroup scheme  of \(\GG_{m,R}^n\) that is finite over \(R\). Let $\tilde{R}$ be a local domain endowed with a structure of $R$-algebra, $\mathfrak{m}$ be the maximal ideal of $\tilde{R}$, $\tilde{\mathsf{F}}=\tilde{R}/\mathfrak{m}$ and  $(A_G,\sum_{i=1}^nm_iD_i)$ be a standard pair over $R$ witnessed by $G$. Then there are natural numbers $N_i,\mathfrak{v}_i$ independent of $\tilde{R}$ for $1\leq i\leq n$  and  a natural number $\mathfrak{l}(\tilde{R})$ such that $\mathfrak{l}(\tilde{R})=0$ if $d(G)$ is prime with the characteristic of $\tilde{\mathsf{F}}$, $\mathfrak{v}_i\geq 1$ for all $i$, $\rho_{G_{\tilde{R}}}^*(\sum_{i=1}^nm_iD_i)=\sum_{i=1}^nN_iE_i$ and \(K_{\AA_{\tilde{R}}^n/A_{G_{\tilde{R}}}}=\mathfrak{l}(\tilde{R})\AA_{\tilde{\mathsf{F}}}^n+\sum_{i=1}^n(\mathfrak{v}_i-1)E_i\). Let $z$ be an $\tilde{\mathsf{F}}$-point of $A_G$ and $I\subset\{1,...,n\}$ such that $z\in \rho_G(\overset{\circ}{E}_I)(\tilde{\mathsf{F}})$. Let $G_I$ be the maximal subgroup scheme of $G$ such that the action of $G_I$ on $E_I$ is trivial. Then $G_I$ is a finite diagonalizable group scheme over $R$ and we have \begin{equation}\label{puGI}\rho_{G_I\otimes_R \tilde{R}}^*(\sum_{i\in I}m_iD_i)=\sum_{i\in 
I}N_iE_i
\end{equation} and \begin{equation}\label{puGI1}K_{\AA_{\tilde{R}}^n/A_{G_I\otimes_R\tilde{R}}}=\mathfrak{l}(I,\tilde{R})\AA_{\tilde{\mathsf{F}}}^n+\sum_{i\in I}(\mathfrak{v}_i-1)E_i
\end{equation}
for some $\mathfrak{l}(I,\tilde{R})\in\NN$ such that $\mathfrak{l}(I,\tilde{R})=0$ if $\mathfrak{l}(\tilde{R})=0$. Moreover, there is an étale morphism $\theta_{I,z}$ of $\tilde{R}$-schemes from  a standard scheme over $\tilde{R}$ witnessed by $G_{I}\otimes_R\tilde{R}$ to an open neighbourhood of $z$ in $A_{G_{\tilde{R}}}$ that induces  an étale morphism of locally standard pairs between  $(A_{G_I\otimes_R \tilde{R}}^{\loc},\sum_{i\in I}e_iD_i)$ and $(\Spec\cO_{A_{G_{\tilde{R}}},z},\sum_{i=1}^ne_iD_i)$ for all natural numbers $e_i$.
\end{lem}
\begin{proof}It is easy to see that the coefficients of $E_i$ in $\rho_{G_{\tilde{R}}}^*(\sum_{i=1}^nm_iD_i)$ and $K_{\AA_{\tilde{R}}^n/A_{G_{\tilde{R}}}}$ are independent of $\tilde{R}$. Thus, it is sufficient to prove the assertions in the case that $R$ is a local domain and $R=\tilde{R}$.

Note that  $K_{\AA_R^n/A_{G}}$ is an effective divisor on $\AA_R^n$ whose support is contained in the ramified locus of $\rho_G$, so contained in $(\cup_{i=1}^nE_i)\cup \AA_{\tilde{\mathsf{F}}}^n$. Thus, it follows from the definition of $\rho_G$ together with the facts that $m_i$ for $1\leq i\leq n$ are natural numbers, to prove the first assertion, it remains to show that $\mathfrak{l}(R)=0$ if $d(G)$ is prime to the characteristic of $\tilde{\mathsf{F}}$. Suppose that  $d(G)$ is prime to the characteristic of $\tilde{\mathsf{F}}$.  Then $\rho_G$ is unramified at the point over the generic point of $A_G\otimes_{R}\tilde{\mathsf{F}}$. Therefore $\mathfrak{l}(R)=0$. 

Now, we prove the second assertion. We consider the projection $\pi_I:\GG_{m,R}^n\to \prod_{i\notin I}\GG_{m,R}$. Then $G_I$ is the kernel of the homomorphism $\pi_I|_G:G\to \prod_{i\notin I}\GG_{m,R}$. Thus, we can observe that $\GG_{m,R}^r=\mathsf{D}_R(\ZZ^r)$ for all positive integers $r$ and use \cite[Exposé VIII, Corollaries 1.6 and 3.3]{SGA3I} to conclude that $G_I$ is a finite diagonalizable group scheme over $R$. Moreover, (\ref{puGI}) and (\ref{puGI1}) follow from the definition of $G_I$ together with the facts that $K_{\AA_R^n/A_G}=K_{\AA_R^n/A_{G_I}}+\rho_{G_I}^*(K_{A_{G_I}/A_G})$ and $K_{A_{G_I}/A_G}$ is an effective divisor on $A_{G_I}$.

For the last assertion, we set $H_I=G/G_I$. Then $H_I$ is also a finite diagonalizable group scheme over $R$ (see \cite[Exposé VIII, Theorem 3.1 and Corollary 3.4]{SGA3I}).  Since \(z\in \rho_G(\overset{\circ}{E}_I)(\tilde{\mathsf{F}})\), we have \((\Spec\cO_{A_G,z},\sum_{i\in I}m_iD_i)=(\Spec\cO_{A_G,z},\sum_{i=1}^nm_iD_i)\). 
Note that we have the canonical isomorphism $A_G=\AA_R^n/G\simeq A_{G_I}/H_I=(\AA_R^n/G_I)/H_I$. Firstly, suppose that there exists $t\in \AA_R^n(\tilde{\mathsf{F}})$ such that $\rho_G(t)=z$. Let $\tilde{t}=(\tilde{t}_1,...,\tilde{t}_n)$ be a lift of $t$ in $\AA_{R}^n(R)$ such that $\tilde{t}_i=0$ iff $i\in I$. Since the action of $G_I$ on $E_I$ is trivial, the translation $\tilde{T}_{\tilde{t}}:R[x_1,...,x_n]\to R[x_1,...,x_n], x_i\to x_i+\tilde{t}_i$ is $G_I$-equivariant. Then $\tilde{T}_{\tilde{t}}$ induces an $R$-automorphism $T_{\tilde{t}}$ of $A_{G_I}$ that sends $\rho_{G_I}(\mathfrak{m}+(x_1,...,x_n))$ to $\rho_{G_I}(\mathfrak{m}+ (x_1-\tilde{t}_1,...,x_n-\tilde{t}_n))=\rho_{G_I}(t).$ On the other hand, it follows from the definition of $G_I$ that the morphism $\eta:A_{G_I}\to A_{G_{I}}/H_I\simeq A_G$ is  étale over a Zariski neighbourhood of $z$ and $\eta(\rho_{G_I}(t))=z$. Thus, it is sufficient to take $\theta_{I,z}=\eta\circ T_{\tilde{t}}$. In the general case where the existence of $t$ is absent, we can do as follows. Since  $\overset{\circ}{E}_I$ is stable under the action of $\GG_{m,R}^n$ on $\AA_R^n$ and we have a $\GG_{m,R}^n$-equivariant isomorphism $\overset{\circ}{E}_I\simeq \prod_{i\notin I}\GG_{m,R}$, where the action of $\GG_{m,R}^n$ on $\prod_{i\notin I}\GG_{m,R}$ is induced by the projection $\pi_I$ and the natural action of $\prod_{i\notin I}\GG_{m,R}$ on itself.  Thus,  the 
action of  $\GG_{\textnormal{m},R}^{n}(\tilde{\mathsf{F}})$ on $\rho_G(\overset{\circ}{E}_I)(\tilde{\mathsf{F}})$ induced by the action of  $\prod_{i\notin I}\GG_{\textnormal{m},R}(\tilde{\mathsf{F}})$ on  $(\prod_{i\notin I}\GG_{\textnormal{m},R}/\pi_I(G))(\tilde{\mathsf{F}})$ is transitive. Therefore, there is $\varrho_z\in \GG_{m,R}^n(\tilde{\mathsf{F}})$ such that the image of $\varrho_z$ via the projection $\tilde{\pi}_{I}:\GG_{m,R}^n\to \prod_{i\in I}\GG_{m,R}$ is $1$ and the $\tilde{\mathsf{F}}$-isomorphism $\varrho_z:A_G\otimes \tilde{\mathsf{F}}\to A_G\otimes \tilde{\mathsf{F}}$ sends $z'$ to $z$, where $z'=\rho_{G_{\tilde{\mathsf{F}}}}((t_1,...,t_n))\in \rho_G(\overset{\circ}{E}_I)(\tilde{\mathsf{F}})$ with $t_i=1$ if $i\notin I$ and $t_i=0$ if $i\in I$.  Let $\tilde{\varrho}_z$ be any lift of $\varrho_z$ in $\GG_{m,R}^n(R)$ then $\tilde{\varrho}_z:A_G\to A_G$ is an $R$-isomorphism sending $z'$ to $z$. To be simpler, we will take $\tilde{\varrho}_z$ so that $\tilde{\pi}_I(\tilde{\varrho}_z)=1$. On the other hand, let $\tilde{t}=(\tilde{t}_1,...,\tilde{t}_n)$ with $\tilde{t}_i=0$ if $i\in I$ and $\tilde{t}_i=1$ if $i\notin I$. As shown above, the morphism $T_{\tilde{t}}$ is an $R$-automorphism of $A_{G_I}$ sending $\rho_{G_I}(\mathfrak{m}+(x_1,...,x_n))$ to $\rho_{G_I}(\mathfrak{m}+(x_1-\tilde{t}_1,...,x_n-\tilde{t}_n))$. Moreover, since $z'\subset \rho_G(\overset{\circ}{E}_I)(\tilde{\mathsf{F}})$, we also see  that  $\eta$ is  étale over a Zariski neighbourhood of $z'$. Since $\eta\circ\rho_{G_I}=\rho_G$, we have $\tilde{\varrho}_z\circ\eta\circ T_{\tilde{t}}\circ\rho_{G_I}(\mathfrak{m}+(x_1,...,x_n))=z$. Now, we can verify that $\theta_{I,z}:=\tilde{\varrho}_z\circ\eta\circ T_{\tilde{t}}$ induces an étale morphism  of $R$-schemes from the locally standard pair $(A_{G_I}^{\loc},\sum_{i\in I}e_iD_i)$ to $(\Spec\cO_{A_G,z},\sum_{i=1}^ne_iD_i)$ for all natural numbers $e_i$.  
\end{proof}
\begin{cor}\label{angucomp}
We keep the notations and assumptions of Lemma \ref{etalelocalquoti}. Moreover, suppose that $R=\tilde{R}$ is a DVR ring. Let $(W,g)$ be a germ of a regular function on $A_G$ at $\rho_G(\mathfrak{m}+(x_1,...,x_n))$. Suppose that $\rho_G^*(g)=t^l\left(\prod_{i=1}^nx_i^{l_i}\right)|_{\rho_G^{-1}(W)}h$ for natural numbers $l,l_i$ and a regular function  $h$ on $\rho_G^{-1}(W)$. Let $I\subset \{1,...,n\}$ and $z\in \rho_G(\overset{\circ}{E}_I)(\tilde{\mathsf{F}})\cap W(\tilde{\mathsf{F}})$. Suppose that there is $\tilde{z}=(\tilde{z}_1,...,\tilde{z}_n)\in \overset{\circ}{E}_I(\tilde{\mathsf{F}})$ such that $z=\rho_G(\tilde{z})$ and $h(\tilde{z})\neq 0$.   
We view $g$ as a regular function on $\Spec\cO_{A_{G},z}$. Let $\theta_{I,z}$ be the morphism constructed in the proof of Lemma \ref{etalelocalquoti}. We also view $\theta_{I,z}^*(g)$ as a regular function on $A_{G_I,\mathfrak{m}}^{\loc}$. Then  $\theta_{I,z}^*(g)$ defines a divisor on an open neighbourhood $W_I$ of $\rho_{G_I}(\mathfrak{m}+(x_1,...,x_n))$ in $A_{G_I}$  whose support is contained in $\left(W_I\otimes_R \tilde{\mathsf{F}}\right)\cup \left(\cup_{i\in I}D_i\right)$ and the standard angular component of $\theta_{I,z}^*(g)$ is $h(z)\prod_{i\notin I}\tilde{z}_i^{l_i}.$
\end{cor}
\begin{proof}Let $t$ be a fixed generator of $\mathfrak{m}$. 
We view $\cO_{A_G}(A_G)$ and $\cO_{A_{G_I}}(A_{G_I})$ as the subrings $R[x_1,...,x_n]^G$ and $R[x_1,...,x_n]^{G_I}$ of $R[x_1,...,x_n]$. Since $G_I$ is a subgroup scheme of $G$, we have  $R[x_1,...,x_n]^G\subset R[x_1,...,x_n]^{G_I}$. 

By Lemma \ref{etalelocalquoti},  there is an open neighbourhood $W_z$ of $z$ in $W$ and an open neighbourhood $W_I$ of $P_I:=\rho_{G_I}(\mathfrak{m}+(x_1,...,x_n))$ in $A_{G_I}$ such that $\theta_{I,z}:W_I\to W_z$ is an étale morphism and $\theta(P_I)=z$.  Since $\theta_{I,z}$ is étale,  the divisor $\Div(\rho_{G_I}^*(\theta_{I,z}^*(g)))$ on $U_I:=\rho_{G_I}^{-1}(W_I)$ is equal to $lU_I\otimes_{R}\tilde{\mathsf{F}}+\sum_{i\in I}l_iE_i|_{U_I}$. Thus, $\theta_{I,z}^*(g)$ defines a divisor on an open neighbourhood $W_I$ of $P_I$ in $A_{G_I}$  whose support is contained in $\left(W_I\otimes_R \tilde{\mathsf{F}}\right)\cup \left(\cup_{i\in I}^nD_i\right)$ as desired. Moreover, by Lemma-Definition \ref{factorunit}, we can  write $\theta_{I,z}^*(g)=vt^l\prod_{i\in I}x_i^{l_i}$ in $\cO_{A_{G_I}}(W_I)$ with a unit element $v\in \cO_{A_{G_I}}(W_I)$. On the other hand, it follows from the construction of $\theta_{I,z}$ that $\theta_{I,z}^*(x_i)=x_i$  for all $i\in I$ and $\theta_{I,z}^*(t)=t$, where the equalities are considered in $R[x_1,...,x_n]$. Thus, we have $v=\theta_{I,z}^*(h\prod_{i\notin I}x_i^{l_i})$ in $\cO_{A_{G_I}}(W_I)$ and $v(P_I)=(h\prod_{i\notin I}x_i^{l_i})(\theta_{I,z}(P_I))=(h\prod_{i\notin I}x_i^{l_i})(z)=h(z)\prod_{i\notin I}\tilde{z}_i^{l_i}$. This completes our proof. 
\end{proof}
\subsection{Embedded \(\QQ\)-resolutions and their good reductions}\label{emQgood}
Let $\mathsf{F}$ be a field.  Let $X$ be a smooth $\mathsf{F}$-variety and let $f$ be a regular function on $X$.  We define 
embedded $\QQ$-resolutions of $D(f)$ as follows.
\begin{defn}\label{q-reso} 
An {\em embedded $\QQ$-resolution} of $(X,D(f))$ is a projective birational map $\pi:Z\to X$  which is an isomorphism over $X\setminus \Sing(D(f))$ and such that the pair \((Z,\pi^{*}(D(f)))\) is  locally a standard pair. In precise, the last condition means that there is a finite open cover $(Y_i)_{1\leq i\leq r}$ of $Z$  such that $(Y_i,\pi^{*}(D(f))|_{Y_i})$ is $\mathsf{F}$-isomorphic to a standard pair over $\mathsf{F}$ for all $i$. 

\end{defn}
\begin{lem}\label{rem}
Let $\pi:Z\to X$ be an embedded $\QQ$-resolution of \((X,D(f))\). Let $V$ be an open affine subvariety of $Z$ such that there is a standard pair \((Y,\cD)\) over $\mathsf{F}$ together with an $\mathsf{F}$-isomorphism of pair \(\phi: (Y,\cD)\to (
V,\pi^{*}(D(f))|_V)\). Let $G$ be the $\mathsf{F}$-diagonalizable subgroup scheme  of $\GG_{m,\mathsf{F}}^n$ witnessing \((Y,\cD)\). Let $N_1,...,N_n\in \NN$ such that $\rho_G^*(\cD) =\sum_{i=1}^n N_iE_i$. We set $W:=\rho_G^{-1}(Y)$. Then there is a unit element $u$ of $\cO_W(W)$ such that $f\circ\pi\circ\phi\circ\rho_G=u\prod_{i=1}^n x_i^{N_i}$ in $\cO_W(W)$. Moreover, there is a standard triple \((Y,\cD,\cD')\) such that \(\phi\) induces an isomorphism of triples over $\mathsf{F}$ between \((Y,\cD,\cD')\) and \((V,\pi^{*}(D(f))|_V,K_{Z/X}|_V)\), where $K_{Z/X}$ is the relative canonical divisor (with respect to $\pi$). In particular, there are positive integers $\nu_i$ for $1\leq i\leq n$ such that $K_{W/X}=\sum_{i=1}^n(\nu_i-1)E_i|_W$. 

\end{lem}
\begin{proof}
The first assertion follows immediately from the definition. To prove the second assertion, we just need to show that $\phi^{*}(K_{Z/X}|_V)$  is an effective divisor on $Y$ whose support is contained in $\cup_{i=1}^nD_i|_Y$. Indeed, since $X$ is smooth and $\pi$ is an isomorphism over $X\setminus \Sing(D(f))$, $K_{V/X}=K_{Z/X}|_V$ is an  effective divisor whose support is contained in $\pi^{*}(D(f))$. Since $\phi$ is an isomorphism of pairs, $K_{Y/X}= \phi^*(K_{V/X})$ is an effective divisor whose support is contained in $\cup_{i=1}^nD_i|_Y$.
The last assertion follows from the facts that $K_{W/X}=K_{W/Y}+\rho_G^*(K_{Y/X})$, $K_{W/Y}$ is an effective divisor whose support is contained in the ramification locus $\Delta$ of $\rho_G|_W$  and $\Delta\subset \cup_{i=1}^n E_i|_W$.
\end{proof}

\begin{prop}\label{stra}
Let $\pi : Z\to X$ be an embedded $\QQ$-resolution of $(X,D(f))$.  Then, there exists a decomposition \(Z=\bigsqcup_{e=1}^r Z_e \) of \(Z\) into finitely many locally closed subsets satisfying the following conditions: 
\begin{itemize}
\item For each \(1\le e \le r\), there exists a locally standard triple \((A_{G_e}^{\rm loc}, \cD_e ,\cD_e')\) such that for each point \(z\in Z_e(\sfF')\) with \(\sfF'/\sfF\) a field extension, there is an $\mathsf{F}'$-étale morphism of triples
\[(A_{G_e}^{\rm loc}, \cD_e ,\cD_e')_{\sfF'}\to 
 (\Spec \cO_{Z_{\sfF'},z}, \pi_{\mathsf{F}'}^{*}D_{\mathsf{F}'}(f), K_{Z_{\mathsf{F}'}/X_{\mathsf{F}'}})\]
where \( (A_{G_e}^{\rm loc}, \cD_e ,\cD_e')_{\sfF'}\) is the triple induced from \((A_{G_e}^{\rm loc}, \cD_e ,\cD_e')\) by the scalar extension \(\sfF'/\sfF\). 
\item If we set $\rho_{G_e}^*(\cD_e)=\sum_{i=1}^nN_{ei}E_i$ and $\rho_{G_e}^*(\cD_e')+K_{\AA_\mathsf{F}^n/A_{G_e}}=\sum_{i=1}^n(\nu_{ei}-1)E_i$ then $N_{ei}\in \NN$ and $\nu_{ei}\in\NN_{\geq 1}$.
\item If there exists a finite extension $\mathsf{F}'$ of $\mathsf{F}$  and $z\in Z_e(\mathsf{F}')$ such that $f(\pi(z))\notin\Sing(D(f))$ then $G_e=\{1\}$, $\cD_e=\cD'_e=0$ and $N_{ei}=\nu_{ei}-1=0$ for all $1\leq i\leq n$.
\end{itemize}
\end{prop}

\begin{proof}
The assertion follows immediately from Lemmas \ref{etalelocalquoti} and \ref{rem}.
 \end{proof}
\begin{defn}\label{dat}Let $\pi:Z\to X$ be an embedded $\QQ$-resolution of $(X,D(f))$ and $(Z_{e})_{1\leq e\leq r}$ be a stratification of $Z$ as in Proposition \ref{stra}. For each $1\leq e\leq r$, let  \((A_{G_e}^{\rm loc}, \cD_e ,\cD_e')\) be the locally standard triples as in Proposition \ref{stra}. We set $\rho_{G_e}^*(\cD_e)=\sum_{i=1}^nN_{ei}E_i$ and $\rho_{G_e}^*(\cD_e')+K_{\AA_\mathsf{F}^n/A_{G_e}}=\sum_{i=1}^n(\nu_{ei}-1)E_i$. 
The tuples $(Z_e,G_e, \nu_{ei},N_{ei})_{1\leq e\leq r, 1\leq i\leq n}$ and $(\nu_{ei},N_{ei})_{1\leq e\leq r, 1\leq i\leq n}$ are called  \emph{data} of the $\QQ$-resolution $(Z,\pi)$ and its \emph{numerical data} respectively.
\end{defn}
\begin{defn}[see \cite{Denefdegree}]\label{good}
Let $L$ be a local field, $X=\AA_L^n$ and $N$ be a positive integer. If $\tilde{W}$ is a subscheme of  $\PP_{\spec(\cO_L[x_1,...,x_n])}^N$, we denote by $\overline{W}$ the reduction modulo $\varpi_L\cO_L$ of $\tilde{W}$, i.e., $\overline{W}:=\tilde{W}\otimes_{\cO_L}k_L$.
Let $Z$ be a closed subscheme of $\PP^N_{X}$, then we denote by $\tilde{Z}$ the Zariski closure of $Z$ in $\PP_{\spec(\cO_L[x_1,...,x_n])}^N\supset\PP_{X}^N$.  If  $\pi:Z\to X$ is the restriction of the projection $\PP^N_{X}\to X$, we denote by $\tilde{\pi}$ the restriction of the projection  $\PP_{\spec(\cO_L[x_1,...,x_n])}^N\to \spec(\cO_L[x_1,...,x_n])$ to $\tilde{Z}$ and $\overline{\pi}:=\tilde{\pi}\otimes k_L: \overline{Z}\to \AA_{k_L}^n$.
\end{defn}

Now, we define the notion of good reductions for embedded $\QQ$-resolutions. This notion is analogous to the one in \cite[Definition 2.2]{Denefdegree} with respect to the usual embedded resolutions.

\begin{defn}\label{q-good}
Let $L$ be a local field, $X=\AA_L^n$ and $f$ be a non-constant polynomial in $\cO_L[x_1,...,x_n]$ such that $\overline{f}:=f\mod \varpi_L\cO_L$ is not constant. Let $(Z,\pi)$ be an embedded $\QQ$-resolution  of $(X,D(f))$. We set $\tilde{X}:=\Spec\cO_L[x_1,...,x_n]$ and $ \tilde{D}(f):=\Spec\cO_L[x_1,...,x_n]/(f)$.  Then we can suppose that $Z$ is a closed subscheme of $\PP^N_X$ for some positive integer $N$ and $\pi$ is the restriction of the projection $\PP^N_X\to X$ to $Z$. With the notations in Definition \ref{good}, we say that the embedded $\QQ$-resolution $(Z,\pi)$ has a \emph{good $\cO_L$-model} if there is a finite open cover $(\tilde{Y}_i)_{1\leq i\leq r}$ of $\tilde{Z}$  and natural numbers $\mathfrak{l}_i$ for $1\leq i\leq r$ such that $\overline{Y}_i=\tilde{Y}_i\otimes_{\cO_L}k_L\neq\emptyset$  and there is an $\cO_L$-isomorphism $\phi_i$ from a standard pair $(\tilde{W}_i,\tilde{\cD}_i)$  over $\cO_L$ witnessed by 
$G_i\subset \GG_{m,\cO_L}^n$ to $(\tilde{Y}_i,\tilde{\pi}^{*}(\tilde{D}(f))|_{\tilde{Y}_i}-\mathfrak{l}_i\overline{Y}_i)$  for all $i$. 

If $(Z,\pi)$ has a good $\cO_L$-model as in the above then we say that it has $\QQ$-\emph{good reduction modulo $\varpi_L\cO_L$} if $\gcd(d(G_i),p_L)=1$ and $\mathfrak{l}_i=0$  
for all $1\leq i\leq r$. In particular, if $(Z,\pi)$ has good reduction modulo $\varpi_L\cO_L$ then $\overline{\pi}:\overline{Z}\to \AA_{k_L}^n$ is an embedded $\QQ$-resolution of $(\AA_{k_L}^n,D(\overline{f}))$. 

In addition, suppose that $(Z,\pi)$ has $\QQ$-good reduction modulo $\varpi_L\cO_L$. For each $1\leq i\leq r$, we write $\rho_{G_i}^*(\tilde{\cD}_i)=\sum_{1\leq j\leq n}N_{ij}E_j$. If for each $1\leq i\leq r$, the condition $(N_{i1},...,N_{in})\in \varpi_L\cO_L^n$ implies $(N_{i1},...,N_{in})=(0,...,0)$,   then we say that $(Z,\pi)$ has $\QQ$-{\em tame good reduction modulo $\varpi_L\cO_L$}.
\end{defn}
In the situation of Definition \ref{q-good}, we also have statements similar to Lemma \ref{rem} and Proposition \ref{stra}  as follows.
\begin{lem}\label{rem1} With the notations and assumptions of Definition \ref{q-good}, let  $\pi:Z\to X$ be an embedded $\QQ$-resolution of \((X,D(f))\). Suppose that $(Z,\pi)$ has a good $\cO_L$-model. Let $\tilde{V}$ be an open affine subvariety of $\tilde{Z}$ such that $\overline{V}=\tilde{V}\otimes_{\cO_L}k_L\neq\emptyset$ and  there is a standard pair \((\tilde{W},\tilde{\cD})\) over $\cO_L$ together with an $\cO_L$-isomorphism \(\phi: (\tilde{W},\tilde{\cD})\to (
\tilde{V},\tilde{\pi}^{*}(\tilde{D}(f))|_{\tilde{V}}-\mathfrak{l}\overline{V})\) for a natural number $\mathfrak{l}$. Let $G$ be the $\cO_L$-diagonalizable subgroup scheme  of $\GG_{m,\cO_L}^n$ that is finite over $\cO_L$ and witnesses \((\tilde{W},\tilde{\cD})\). Let $N_1,...,N_n\in \NN$ such that $\rho_G^*(\tilde{\cD})=\sum_{i=1}^n N_iE_i$. We set $\tilde{Y}:=\rho_G^{-1}(\tilde{W})$. Then there is a unit element $u$ of $\cO_{\tilde{Y}}(\tilde{Y})$ and a natural number $\mathfrak{N}$ such that $f\circ\tilde{\pi}\circ\phi\circ\rho_G=u\varpi_L^{\mathfrak{N}}\prod_{i=1}^n x_i^{N_i}$ in $\cO_{\tilde{Y}}(\tilde{Y})$. Moreover, there is a standard triple \((\tilde{W},\tilde{\cD},\tilde{\cD}')\) over $\cO_L$ and a natural number $\mathfrak{l}'$ such that \(\phi\) induces an  $\cO_L$-isomorphism of triples between \((\tilde{W},\tilde{\cD},\tilde{\cD}')\) and \((\tilde{V},\tilde{\pi}^{*}(\tilde{D}(f))|_{\tilde{V}}-\mathfrak{l}\overline{V},K_{\tilde{Z}/\tilde{X}}|_{\tilde{V}}-\mathfrak{l}'\overline{V})\), where $K_{\tilde{Z}/\tilde{X}}$ is the relative canonical divisor with respect to $\tilde{\pi}$. In particular, there are positive integers $\nu_i$ for $1\leq i\leq n$ and a natural number $\mathfrak{n}$ such that $K_{\tilde{Y}/\tilde{X}}=\mathfrak{n}\overline{Y}+\sum_{i=1}^n(\nu_i-1)E_i|_{\tilde{Y}}$.
\end{lem}
\begin{lem}\label{rem2}
Keeping the notations and assumptions of Lemma \ref{rem1}. Moreover, assume that $(Z,\pi)$ has $\QQ$-good reduction modulo $\varpi_L\cO_L$ then $\mathfrak{n}=\mathfrak{l}=\mathfrak{l}'=\mathfrak{N}=0$. 
\end{lem}
\begin{proof}[Proof of Lemmas \ref{rem1} and \ref{rem2}]
The assertions  follow by doing similarly to Lemma \ref{rem} together with using Lemma \ref{etalelocalquoti}. 
\end{proof}
\begin{prop}\label{stragood} Let $L$ be a local field, $X=\AA_L^n$ and $f$ be a non-constant polynomial in $\cO_L[x_1,...,x_n]$ such that $\overline{f}:=f\mod \varpi_L\cO_L$ is not constant.
Let $\pi : Z\to X$ be an embedded $\QQ$-resolution of $(X,D(f))$. Suppose that $(Z,\pi)$ has a good $\cO_L$-model. Then, there exists a decomposition \(\overline{Z}=\bigsqcup_{e=1}^r \overline{Z}_e \) of \(\overline{Z}\) into finitely many locally closed subsets satisfying the following condition: 
\begin{itemize}
\item For each \(1\le e \le r\), there exists a locally standard triple \((A_{G_e}^{\rm loc}, \tilde{\cD}_e ,\tilde{\cD}_e')\) over $\cO_L$ and natural numbers $\mathfrak{l}_e, \mathfrak{l}'_e$
such that for each point \(z\in \overline{Z}_e(k_{L'})\) with \(L'/L\) a finite extension, there is an $\cO_{L'}$-étale morphism $\phi_z$ from a standard scheme over $\cO_{L'}$ witnessed by $G_e\otimes_{\cO_L}\cO_{L'}$ to $\tilde{Z}_{\cO_{L'}}$ that induces an   $\cO_{L'}$-étale morphism of triples
\[(A_{G_e}^{\rm loc}, \tilde{\cD}_e ,\tilde{\cD}_e')_{\cO_{L'}}\to 
 (\Spec \cO_{\tilde{Z}_{\cO_{L'}},z}, (\tilde{\pi}^{*}\tilde{D}(f))_{\cO_{L'}}-\iota_{L'/L}\mathfrak{l}_e\Spec\cO_{\tilde{Z}_{k_{L'}},z}, (K_{\tilde{Z}/\tilde{X}})_{\cO_{L'}}-\iota_{L'/L}\mathfrak{l}'_e\Spec\cO_{\tilde{Z}_{k_{L'}},z})
,\]
where \( (A_{G_e}^{\rm loc}, \tilde{\cD}_e ,\tilde{\cD}_e')_{\cO_{L'}}\) is the triple induced from \((A_{G_e}^{\rm loc}, \tilde{\cD}_e ,\tilde{\cD}_e')\) by the scalar extension \(\cO_{L'}/\cO_L\) and $\iota_{L'/L}$ is the ramification index of the finite extension $L'/L$. 
\item For each $1\leq e\leq r$, there are natural numbers $\mathfrak{n}_e, \mathfrak{N}_e, N_{ei}, \nu_{ei}$ for $1\leq i\leq n$  such that $\nu_{ei}\geq 1$ for all $i$, $\rho_{G_e}^*(\tilde{\cD}_e)=\sum_{i=1}^nN_{ei}E_i$, $$(\tilde{\pi}_{\cO_{L'}}\circ\phi_z\circ\rho_{G_e\otimes_{\cO_L}\cO_{L'}})^*(\tilde{D}(f)_{\cO_{L'}})|_{\AA_{\cO_{L'}}^{n,\loc}}=\iota_{L'/L}\mathfrak{N}_e\AA_{k_{L'}}^{n,\loc}+\sum_{i=1}^nN_{ei}E_i$$ and 
\[\rho_{G_e\otimes_{\cO_L}\cO_{L'}}^*(\tilde{\cD}_e'\otimes_{\cO_L}\cO_{L'}+\iota_{L'/L}\mathfrak{l}'_e\AA_{G_e\otimes_{\cO_L} k_{L'}}^{\loc})+K_{\AA_{\cO_{L'}}^{n,\loc}/A_{G_e\otimes_{\cO_L}\cO_{L'}}^{\loc}}=\iota_{L'/L}\mathfrak{n}_e\AA_{k_{L'}}^{n,\loc}+\sum_{i=1}^n(\nu_{ei}-1)E_i
\]  for all finite extensions $L'$ of $L$ and all $z\in \overline{Z}_e(k_{L'})$.
\item If $(Z,\pi)$ has $\QQ$-good reduction modulo $\varpi_L\cO_L$ then $\mathfrak{n}_e=\mathfrak{N}_e=\mathfrak{l}_e=\mathfrak{l}'_e=0$  for all $1\leq e\leq r$.
\item For each $1\leq e\leq r$ such that $\overline{Z}_e\setminus \overline{\pi}^{-1}(\Sing(D(\overline{f})))\neq \emptyset$,  we have $G_e=\{1\}$, $\cD_e=\cD'_e=0$ and $\nu_e=\mathfrak{N}_e=\mathfrak{l}_e=\mathfrak{l}'_e=N_{ei}=\nu_{ei}-1=0$ for all $1\leq i\leq n$.
\end{itemize}
\end{prop}
\begin{proof}
 We observe that the relative canonical divisor with respect to an étale morphism is zero.  Thus,   
we have $K_{A_{G_e\otimes_{\cO_L}\cO_{L'}}^{\rm loc}/\Spec\cO_{\tilde{Z}_{\cO_{L'}},z}}=0.$ 
Therefore, the assertions follow from Lemmas \ref{etalelocalquoti}, \ref{rem1} and \ref{rem2}.
\end{proof}

\begin{defn}\label{destrata}Let $L$ be a local field, $X=\AA_L^n$ and $f$ be a non-constant polynomial in $\cO_L[x_1,...,x_n]$ such that $\overline{f}:=f\mod \varpi_L\cO_L$ is not constant.  
Let $(Z,\pi)$ be an embedded $\QQ$-resolution  of $(X,D(f))$. 

If $(Z,\pi)$ has a good $\cO_L$-model, we fix a stratification $(\overline{Z}_{e})_{1\leq e\leq r}$ of $\overline{Z}$ together with $\cO_L$-group schemes $G_e$, the natural numbers $\mathfrak{n}_e,\mathfrak{N}_e, N_{ei},\nu_{ei}$  for $1\leq e\leq r$ and $1\leq i\leq n$ mentioned in Proposition \ref{stragood}. In addition, with the notations in  Lemma-Definition \ref{factorunit} and Proposition \ref{stragood}, we also consider the map $\mathfrak{u}_e: \overline{Z}_e(\overline{k}_L)\to \overline{k}_L$ that sends $z$ to the standard angular component of $f\circ\tilde{\pi}_{\cO_{L'}}\circ\phi_z$ if $z\in \overline{Z}_e(k_{L'})$ and $L'$ is a finite extension of $L$. It is clear that if $L'$ is a finite extension of $L$ and $z\in \overline{Z}_e(k_{L'})$ then $\mathfrak{u}_e(z)\in k_{L'}$.  
We call the tuple $(\overline{Z}_e,G_e, \nu_{ei},N_{ei},\mathfrak{n}_e,\mathfrak{N}_e,\mathfrak{u}_e)_{1\leq e\leq r, 1\leq i\leq n}$ and $(G_e, \nu_{ei},N_{ei},\mathfrak{n}_e,\mathfrak{N}_e,\mathfrak{u}_e)_{1\leq e\leq r, 1\leq i\leq n}$ \emph{almost good data} of the embedded $\QQ$-resolution $(Z,\pi)$ and their \emph{numerical data}, respectively. For each $1\leq e\leq r$, we also call $(G_e, \nu_{ei},N_{ei},\mathfrak{n}_e,\mathfrak{N}_e,\mathfrak{u}_e)_{1\leq i\leq n}$ the \emph{numerical data} of $\overline{Z}_e$.

If $(Z,\pi)$ has $\QQ$-good reduction modulo $\varpi_L\cO_L$ and $(\overline{Z}_e,G_e, \nu_{ei},N_{ei},\mathfrak{n}_e,\mathfrak{N}_e,\mathfrak{u}_e)_{1\leq e\leq r, 1\leq i\leq n}$ is almost good data of  $(Z,\pi)$ then $\mathfrak{n}_e=\mathfrak{N}_e=0$ for all $e$. If this is the case, we call the tuples $(\overline{Z}_e,G_e, \nu_{ei},N_{ei},\mathfrak{u}_e)_{1\leq e\leq r, 1\leq i\leq n}$ and $(G_e,\nu_{ei},N_{ei},\mathfrak{u}_e)_{1\leq e\leq r, 1\leq i\leq n}$  \emph{good data} of $(Z,\pi)$ and 
 their \emph{numerical data}, respectively. Moreover,  for each $1\leq e\leq r$, we call the tuple $(G_e,\nu_{ei},N_{ei},\mathfrak{u}_e)_{1\leq i\leq n}$ the \emph{numerical data} of $\overline{Z}_e$.
\end{defn}

\begin{remark}\label{choiceangu} The map $\mathfrak{u}_e$ in Definition \ref{destrata} depends on the choice of the $\cO_L$-étale morphism $\phi_z$ for each $z\in \overline{Z}_e(\overline{k}_L)$. However, the formula in Main Theorem \ref{mucase} does not depend on the choice of $\mathfrak{u}_e$ (see also Remark \ref{trivindep} below). In fact, it is easy to verify that if $\mathfrak{u}_{e1}$ and $\mathfrak{u}_{e2}$ are two different choices of the map $\mathfrak{u}_e$ then for each finite extension $L'$ of $L$ and each $z\in \overline{Z}_{e}(k_{L'})$, there are elements $a_i\in k_{L'}^{\times}$ for $1\leq i\leq n$ such that $\mathfrak{u}_{e1}(z)=\mathfrak{u}_{e2}(z)\prod_{i=1}^na_i^{N_{ei}}.$
\end{remark}

\begin{lem}\label{almostgood}Let $K$ be a number field and $f$ be a non-constant polynomial in $\cO_K[x_1,...,x_n]$. Let $(Z,\pi)$ be an embedded $\QQ$-resolution of $(\AA_K^n,D_K(f))$. Then there exists an integer $M$ such that $(Z\otimes_K L, \pi\otimes_K L)$ is an embedded $\QQ$-resolution of ($\AA_L^n,D_L(f))$ with $\QQ$-(tame) good reduction modulo $\varpi_L\cO_L$ for all local fields $L$ over $\cO_K$ of residue field characteristic at least $M$. 
Moreover, data of  $(Z,\pi)$ given by Definition \ref{dat} induce good data of $(Z\otimes_K L, \pi\otimes_K L)$ for each local field $L$ over $\cO_K$ of residue field characteristic at least $M$.
\end{lem}
\begin{proof}For each local field $L$ over $K$,  it is clear that $(Z_L:=Z\otimes_L L,\pi_L:=\pi\otimes_K L)$ is an embedded $\QQ$-resolution of $(\AA_L^n,D_L(f))$. 
Now, we fix a finite affine open cover $(Y_i)_{1\leq i\leq r}$ of $Z$ together with a standard pair $(W_i,\sum_{j=1}^nm_iD_i|_{W_i})$ over $K$ witnessed by $G_i\subset\GG_{m,K}^n$ and a $K$-isomorphism $\phi_i$ from $(W_i,\sum_{j=1}^nm_iD_i|_{W_i})$ to $ (Y_i,\pi^{*}(D_K(f))|_{Y_i})$ for each $1\leq i\leq r$.  
There exists a finite set $S$ of places of $K$ containing all infinite places of $K$ such that $Z,\pi, Y_i, G_i, W_i,\phi_i$ for $1\leq i\leq r$ can be defined over $\cO_S$. Let $\tilde{Z},\tilde{Y}_i,\tilde{W}_i,\tilde{G}_i$ be the $\cO_S$-schemes with the same defining equations as $Z,Y_i,W_i,G_i$ respectively for $1\leq i\leq N$. Moreover, let $\tilde{\pi}: \tilde{Z}\to\AA_{\cO_S}^n$ and $\tilde{\phi}_i:(\tilde{W}_i,\sum_{j=1}^nm_iD_i|_{\tilde{W}_i})\to (\tilde{Y}_i,\tilde{\pi}^{*}(D_{\cO_S}(f))|_{\tilde{Y}_i})$ for $1\leq i\leq r$ be the $\cO_S$-models of $\pi$ and $\phi_i$, where $D_{\cO_S}(f)=\Spec\cO_S[x_1,...,x_n]/(f)$. By enlarging $S$ if needed we can  suppose that $\tilde{\phi}_i$ is an isomorphism over $\cO_S$, $\tilde{\pi}^{*}(D_{\cO_S}(f))\otimes_{\cO_S}\cO_S/\mathfrak{p}\subsetneq \tilde{Y}_i\otimes_{\cO_S}\cO_S/\mathfrak{p}$ for all finite places $\mathfrak{p}\notin S$ of $K$,  $\tilde{G}_i$ is flat over $\cO_S$ and $d(\tilde{G}_{i})$ is a unit in $\cO_S$ for all $i$. 
 Then it is easy to see that the embedded $\QQ$-resolution $(Z_L,\pi_L)$ of $(\AA_L^n,D_L(f))$  has  $\QQ$-good reduction modulo $\varpi_L\cO_L$ if $L$ is a local field over $\cO_K$ of residue field characteristic at least  $M:=\max\{\cN(\mathfrak{p})| \mathfrak{p} \mbox{ is a finite place in } S\}$, where if $\mathfrak{p}$ is a finite place of $K$ then $\cN(\mathfrak{p}):=\#(\cO_K/\mathfrak{p})$ is its norm. 
Moreover, by using Lemma \ref{rem} for the above affine open subsets $Y_i$ of $Z$, we see that the embedded $\QQ$-resolution $(Z_L,\pi_L)$ of $(\AA_L^n,D_L(f))$  has  $\QQ$-tame good reduction modulo $\varpi_L\cO_L$ if $L$ is a local field over $\cO_K$ of residue field characteristic at least $M_1$ for some integer $M_1\geq M$. Similarly to Proposition \ref{stragood}, from  data of $(Z,\pi)$, we can use Lemma \ref{etalelocalquoti} to construct good data of $(Z_L, \pi_L)$ for each local field $L$ over $\cO_K$ of residue field characteristic at least $M$.
\end{proof}
\begin{remark}\label{comparQgoo}
In Definition \ref{q-reso}, if for a given embedded $\QQ$-resolution, all standard pairs can be taken to be witnessed by the trivial group scheme $\{1\}$, then this embedded $\QQ$-resolution becomes a usual embedded resolution. 

Similarly, in Definition \ref{q-good}, an embedded $\QQ$-resolution with good reduction modulo $\varpi_L\cO_L$ becomes an embedded resolution with good reduction modulo $\varpi_L\cO_L$ defined in \cite[Definition 2.2]{Denefdegree} if the group schemes $G_i$ can be taken to be the trivial group scheme $\{1\}$. However, our condition on the numerical data for $\QQ$-tame good reduction  is a bit weaker than the one in \cite[Section 3]{DenefBour}  for tame good reduction.
\end{remark}

\subsection{Igusa's local zeta functions and exponential sums modulo $p^m$}\label{Igex}
Let  $f$ be  a non-constant polynomial in $L[x_1,...,x_n]$ with a local field $L$.
For each multiplicative character $\chi$ of $\cO_L^{\times}$ and each Schwartz-Bruhat function $\Phi$ on $L^n$,  the Igusa zeta function associated with $f,\Phi,\chi$ 
is defined by
$$\cZ_{f,\Phi,\chi}(s):=\int_{L^n}\Phi(x)\chi(\acc(f(x)))|f(x)|^s|dx|,$$
where $s\in \CC$ with $\Re(s)>0$. By the work of Igusa (see for instance \cite{Igusa3}), $\cZ_{f,\Phi,\chi}(s)$ is a rational function in $q_L^{-s}$, thus it has a meromorphic continuation to $\CC$. Moreover, the real part of any pole of $\cZ_{f,\Phi,\chi}(s)$ is a negative rational number.

Igusa's local zeta functions are closely related to exponential sums, which are defined as follows: 
Let $\psi$ be an additive character of $L$. The \emph{exponential sum} associated with $\Phi, \psi, f$ is defined by 
$$E_{f,\Phi,\psi}:=\int_{L^n}\Phi(x)\psi(f(x))|dx|.$$
We have the following relation between $E_{f,\Phi,\psi}$ and $\cZ_{f,\Phi,\chi}(s)$.

\begin{prop}[\cite{DenefBour}, Proposition 1.4.4]\label{zetaex}
Assume that $\psi$ has conductor $0$. Let $z=u\varpi_L^{-m}$ with  $u\in\cO_{L}^{\times}$ and $m\geq 1$. Suppose that $\Supp(\Phi)\cap C_{f}\subset f^{-1}(0)$.  Then
\begin{align*}
E_{f,\Phi,\psi_z} =
\cZ_{f,\Phi,\chi_\textnormal{triv}}(0)&+\mbox{\textnormal{Coeff}}_{t^{m-1}}\left(\dfrac{(t-q_L)\cZ_{f,\Phi,\chi_\textnormal{triv}}(s)}{(q_L-1)(1-t)}\right)\\
&+\sum_{\chi\neq\chi_\textnormal{triv}}g_{\chi^{-1}}\chi(u)\mbox{\textnormal{Coeff}}_{t^{m-c(\chi)}}(\cZ_{f,\Phi,\chi}(s)),
\end{align*}
where $t=q_L^{-s}$ and $g_{\chi}$ is the Gaussian sum
\[
g_{\chi}=\frac{q_L^{1-c(\chi)}}{q_L-1}\sum_{\overline{v}= v\mod \varpi_L^{c(\chi)}\in (\cO_{L}/\varpi_L^{c(\chi)})^{\times}}\chi(v)\psi\bigl(\frac{v}{\varpi_L^{c(\chi)}}\bigr).
\]
\end{prop}

By writing $\cZ_{f,\Phi,\chi}(s)$ in partial fractions and the fact that there are only finitely many $\chi$ such that $\cZ_{f,\Phi,\chi}(s)$ is not holomorphic (see \cite{Igusa3}), Proposition \ref{zetaex} implies the following  two corollaries.

\begin{cor}[see \cite{Igusa1,Igusa3} or \cite{DenefBour}, Corollary 1.4.5] \label{singp}
In the situation of Proposition \ref{zetaex}, if  $|z|$ is large enough,  then $E_{f,\Phi,\psi_z}$ is a finite $\CC$-linear combination of functions of the form 
\[
\chi(\acc(z))|z|^{\lambda}|\ord_L(z)|^{\beta},
\]
where $\lambda\in\CC$ is a pole of $(q_L^{s+1}-1)\cZ_{f,\Phi,\chi_{\textnormal{triv}}}(s)$ or $\cZ_{f,\Phi,\chi}(s)$ for $\chi\neq\chi_{\textnormal{triv}}$ and $\beta$ is a natural number such that $\beta+1$ is at most the multiplicity of the pole $\lambda$.  Moreover, 
coefficients of this $\CC$-linear combination  are independent of $z$, and for each pole \(\lambda\) of $(q_L^{s+1}-1)\cZ_{f,\Phi,\chi_{\textnormal{triv}}}(s)$ or of $\cZ_{f,\Phi,\chi}(s)$ with $\chi\neq\chi_{\textnormal{triv}}$, a function as above for some \(\beta\) appears in the linear combination with a nonzero coefficient.
\end{cor}

\begin{cor}\label{upperbound}
Keeping the notations and assumptions of Proposition \ref{zetaex}, let $\sigma= \sigma(f,\Phi)\in\RR $ be the smallest real number such that $-\sigma$ is the real part of a pole of $(q_L^{s+1}-1)\cZ_{f,\Phi,\chi_{\textnormal{triv}}}(s)$ or $\cZ_{f,\Phi,\chi}(s)$ for $\chi\neq\chi_{\textnormal{triv}}$ and let $\theta$ be the largest positive integer such that $\theta+1$ is the multiplicity of such a pole $\lambda$ with $\Re(\lambda)=-\sigma$.  Then there is a constant $c$ depending on $f$ and $\Phi$ such that 
$$|E_{f,\Phi,\psi_z}|\leq c|z|^{-\sigma}|\ord_L(z)|^{\theta}$$
for all $z\in L^{\times}$.
\end{cor}
\begin{proof}Let $\chi$ be a multiplicative character of $\cO_L^{\times}$, we set $\epsilon_\chi(s)=1$ if $\chi\neq\chi_{\rm triv}$ and $\epsilon_{\chi_{\rm triv}}(s)=q_L^{s+1}-1$. Let $\cA_{f,\Phi}$ be the set of multiplicative characters $\chi$ of $\cO_L^{\times}$ such that $\cZ_{f,\Phi,\chi}(s)$ is not holomorphic. As mentioned above, $\cA_{f,\Phi}$ is a finite set. For each $\chi\in \cA_{f,\Phi}$, we denote by $\cB_{\chi,f,\Phi}$ the set consisting of the real parts of poles of  $\epsilon_\chi(s)\cZ_{f,\Phi,\chi}(s)$.  For each $\lambda\in \cB_{\chi,f,\Phi}$, we set 
\[\mathfrak{m}(\lambda)=\max\{\mathfrak{m}\mid \exists s_0, \Re(s_0)=\lambda, \left((s_0-s)^{\mathfrak{m}}\epsilon_\chi(s)\cZ_{f,\Phi,\chi}(s)\right)(s_0)=\infty\}\]

By Corollary \ref{singp}, if $\ord_L(z)\leq-m_0$ for a large enough integer $m_0$, then   \[E_{f,\Phi,\psi_z}=\sum_{\chi\in \cA_{f,\Phi}}\sum_{\lambda\in \cB_{\chi,f,\Phi}}\sum_{0\leq \beta\leq \mathfrak{m}(\lambda)}a(L,\chi,\lambda,\beta)\chi(\acc(z))|z|^{\lambda}|\ord_L(z)|^{\beta}\]
for complex numbers $a(L,\chi,\lambda,\beta)$ independent of $z$. By enlarging $m_0$, we can suppose that 
\[|z|^{\lambda}|\ord_L(z)|^{\beta}\leq |z|^{-\sigma}|\ord_L(z)|^{\theta} \]
if $\ord_L(z)\leq -m_0$.
On the other hand, it is clear that 
\[\left|E_{f,\Phi,\psi_z}\right|\leq C:=\int_{L^n}|\Phi(x)||dx|\]
for all $z$. Thus, the assertion follows from taking \[c=\max\bigg\{\sum_{\chi\in \cA_{f,\Phi}}\sum_{\lambda\in \cB_{\chi,f,\Phi}}\sum_{0\leq \beta\leq \mathfrak{m}(\lambda)}|a(L,\chi,\lambda,\beta)|, C\max\{q_L^{m\sigma}|m|^{-\theta}\mid  m\leq m_0\}\bigg\}.\]
\end{proof}
It is clear that poles of Igusa's local zeta functions play an important role in the above picture. In fact, the poles of $\cZ_{f,\Phi,\chi}(s)$ are conjectured to relate to the Bernstein-Sato polynomials of $f$. Here, if $g$ is a non-constant polynomial over a field $\mathsf{F}$ of characteristic $0$ then its Bernstein-Sato polynomial $b_g(s)$ is the monic polynomial of smallest degree such that there is an element $P\in \mathsf{F}[x_1,...,x_n,\frac{\partial}{\partial x_1},...,\frac{\partial}{\partial x_n},s]$ satisfying $Pg^{s+1}=b_g(s)g^s$. 
Now, we suppose that the coefficients of $f$ belongs to the ring $\cO_K$ of integers in a number field $K$. 
The conjecture mentioned above is stated in precise as follows:
\begin{conj}[Strong monodromy conjecture]\label{conjmono}  Let $Y$ be a subscheme of $\AA_{\cO_K}^n$. Then there is a natural number $M$ such that for all local fields $L$ over $\cO_K$ of large enough residue field characteristic, all multiplicative characters $\chi$ of $\cO_L^{\times}$ and all poles $s_0$ of $\cZ_{f,\Phi_{L,Y},\chi}(s)$, one has $b_f(\Re(s_0))=0$, where $\Phi_{L,Y}$ was defined in Section \ref{bano}.
\end{conj}

Note that Main Theorem \ref{mthm1} and \cite[Corollary 3.4.1]{BBV} imply a special case of this conjecture for weighted homogeneous polynomials with isolated singularities. 
Relating to the number $\sigma$ in Corollary \ref{upperbound}, we recall the definition of motivic oscillation index and the general form of Igusa's conjecture for exponential sums stated in \cite{CMN,NguyenVeys}.  Consider a non-constant polynomial \(f\) with coefficients in a number field \(K\).  Let  $Y$ be a subscheme of $\AA_{\cO_K}^n$ such that  $f(Y(\CC))=0$. For each local field \(L\) over \(K\),
we have the number \(\sigma(f, \Phi_{L,Y})\) introduced in Corollary \ref{upperbound}.
\begin{defn}\label{moi} 
The \emph{motivic oscillation index }  of $f$ at $Y$ over $K$ is defined to be 
$$\moi_{K}(f,Y):=\lim_{M\to+\infty} \inf\{\sigma(f,\Phi_{L,Y})\mid L\textnormal{ is a local field over } \cO_K, p_L>M\}.$$
\end{defn}

 Main Theorem \ref{mthm2} is a special case of the following conjecture, which gives a uniform upper bound of  exponential sums:

\begin{conj}[Igusa's conjecture for exponential sums] \label{conjex} 
There exist an integer  $M$ and a constant $c$ depending only on $f$ such that 
$$|E_{f,\Phi_{L,Y},\psi}|\leq cm_\psi^{n-1}q_L^{-\moi_{K}(f,Y)m_\psi}$$
for all subschemes $Y$ of $\AA_{\cO_K}^n$ with $f(Y(\CC))=0$, all  local fields $L$ over $\cO_K$ of residue field characteristic at least $M$ and all additive characters $\psi$ of $L$ of conductor $m_\psi\geq 2$. 
\end{conj}

\section{Stringy point count}\label{stringy}
In this section, we fix a local field  $L$ and a finite flat group scheme $G$ over $\cO_L$.
\subsection{Log pairs and \(p\)-adic measures}

We will recall the results in \cite{Yas} on the stringy point count of log pairs equipped with an action of a finite group. Note that these results extend quite automatically if we replace finite groups by finite group schemes over the ring of integers in local fields. Thus, we aim to state these results in the more general setting  with these finite group schemes and explain some adaptation if needed.  Moreover, we also generalize slightly the result in \cite{Yas} by introducing the notion of stringy count point twisted by a pair $(g,\chi)$ over  $L$, where $g$ is a morphism to $\AA_{\cO_L}^1$ and $\chi$ is a multiplicative character of $\cO_L^{\times}$. 

Let $X$ be a $n$-dimensional $\cO_L$-variety. We write $X_L, X_{k_L}$ for $X\otimes_{\cO_L} L$ and $X\otimes_{\cO_L}k_L$ respectively. Let $X_L^{\rm sm}$ be the smooth locus of $X_L$ and set $X(\cO_L)^{\circ}:=X(\cO_L)\cap X_L^{\rm sm}(L)$. Then it is easy to see that $X(\cO_L)^{\circ}$ is equipped with a natural structure of an $L$-analytic manifold. Let $\Omega_{X/\cO_L}^n$ be the sheaf of $n$-differential forms of $X$ over $\cO_L$. Let $r$ be a positive integer and $\cI$ be an invertible $\cO_X$-submodule of $(\Omega_{X/\cO_L}^n)^{\otimes r}\otimes K(X)$, where $K(X)$ is the sheaf of total quotient rings of $X$. Then we can equip naturally $X(\cO_L)^{\circ}$ with a measure $\mu_\cI$ given by local generators of $\cI$ (see \cite{Yas}). This measure extends automatically to $X(\cO_L)$ by setting the measure of $X(\cO_L)\setminus X(\cO_L)^{\circ}$ to be zero.

\begin{defn} A log pair is a pair $(X,D)$ of a normal $\cO_L$-variety $X$ and a $\QQ$-divisor $D$ (a Weil divisor with rational coefficients) on $X$  such that $K_X+D$ is $\QQ$-Cartier.
\end{defn}
Let $(X,D)$ be a log pair and let $r\in \NN$ such that $r(D+K_X)$ is Cartier. Then the invertible sheaf $\cO_X(r(D+K_X))$ can be viewed as a subsheaf of $(\Omega_{X/\cO_L}^n)^{\otimes r}\otimes K(X)$. Let $\mu_{X,D}$ be the measure on $X(\cO_L)$ associated to the sheaf $\cO_X(r(D+K_X))$ as mentioned above.

We recall the following definition in \cite{Yas}.
\begin{defn}
A morphism of log pairs $f:(Y,E)\to (X,D)$, where $X,Y$ are $n$-dimensional $\cO_L$-varieties, is a morphism $f:Y\to X$. A morphism of log pairs $f:(Y,E)\to (X,D)$ is crepant if for a positive integer $r$ such that both $r(D+K_X)$ and $r(E+K_Y)$ are Cartier, the canonical morphism $f^{*}(\Omega_{X/\cO_L}^n)^{\otimes r}\to (\Omega_{Y/\cO_L}^n)^{\otimes r}$ induces an isomorphism  $f^{*}(\cO_X(r(D+K_X)))\to\cO_Y(r(E+K_Y))$.
\end{defn}
Since $G$ is a finite group scheme over $\cO_L$,  $\cO_G(G)$ is integral over $\cO_L$.  Thus $G(\cO_L)=G(L)$. By slightly generalizing \cite[Theorem 4.8]{Yas}, we obtain the  following theorem which provides a change of measure under a morphism of log pairs.
\begin{thm}\label{change} Let $f:(Y,E)\to (X,D)$ be a crepant morphism of log pairs. Suppose that there is a finite  group $\cO_L$-scheme $G$ acts faithful on $Y$ and trivially on $X$ such that $G_L$ is étale over $L$, $f:Y\to X$ is $G$-equivariant and the induced morphism $Y/G\to X$ is birational. Let $A$ be a $G$-stable open subset of $Y(\cO_L)$ then we have 
$$\frac{1}{\#G(\cO_L)}\mu_{Y,E}(A)=\mu_{X,D}(f(A)).$$
\end{thm}

\subsection{Stringy point count}
We need a slight generalization of the notion of stringy point counts considered in \cite{Yas}.

\begin{defn}\label{defnstringy}
Let $(X,D)$ be a log pair for a $n$-dimensional $\cO_L$-variety $X$ and $C$ be a subscheme of $X_{k_L}$. We denote by  $X(\cO_L)_C$ the set of $\cO_L$-points of $X$ sending the closed point of $\spec(\cO_L)$ into $C$. Let $\chi$ be a multiplicative character of $\cO_L^{\times}$ and $g$ be an $\cO_L$-morphism from a Zariski neighbourhood of $C$ in $X$ to $\AA_{\cO_L}^1$. The \emph{stringy point count of $(X,D)$ along $C$ twisted by $(g,\chi)$} is defined by 
$$\#_{\rm st}^{g,\chi}(X,D)_C:=q_L^n\int_{X(\cO_L)_C}\chi(\acc(g))\mu_{X,D}.$$
\end{defn}

It is clear that 
\begin{equation}\label{pointwise}
    \#_{\rm st}^{g,\chi}(X,D)_C=\sum_{x\in C(k_L)}\#_{\rm st}(X,D)_x^{g,\chi}.
\end{equation}
By Theorem \ref{change}, we have the following generalization of \cite[Corrollary 4.10]{Yas}.

\begin{prop}\label{relation1}
Let $f:(Y,E)\to (X,D)$ be a crepant proper birational morphism of log pairs for $\cO_L$-varieties $X,Y$. Let $C$ be a subscheme of $X_{k_L}$, $\chi$ be a multiplicative character of $\cO_L^{\times}$ and $g$ be an $\cO_L$-morphism from a Zariski neighbourhood of $C$ in $X$ to $\AA_{\cO_L}^1$. Then we have 
$$\#_{\rm st}^{g,\chi}(X,D)_C=\#_{\rm st}^{g\circ f,\chi}(Y,E)_{\overline{f}^{-1}(C)}.$$
\end{prop}
We also have a weaker version of Proposition \ref{relation1} in the case of crepant étale morphisms.

\begin{prop}\label{relation2}Let $f:(Y,E)\to (X,D)$ be a crepant étale morphism of log pairs for $\cO_L$-varieties $X,Y$. Let $y$ be a $k_L$-point of $Y_{k_L}$, $\chi$ be a multiplicative character of $\cO_L^{\times}$ and $g$ be an $\cO_L$-morphism from a  Zariski neighbourhood of $\overline{f}(y)$ in $X$ to $\AA_{\cO_L}^1$. Then we have 
$$\#_{\rm st}^{g,\chi}(X,D)_{\overline{f}(y)}=\#_{\rm st}^{g\circ f,\chi}(Y,E)_{y}.$$
\end{prop}
\begin{proof}
If $f$ is étale, the map $Y(\cO_L)_y \to X(\cO_L)_{\overline{f}(y)}$ is bijective. Since $(Y,E)\to (X,D)$ is crepant, the two measures $\mu_{Y,E}$ and $\mu_{X,D}$ as well as $g\circ f$ and $g$ correspond to each other through this bijection. This proves the assertion.
\end{proof}
The following example relates the stringy point count and Igusa's local zeta functions.
\begin{example}\label{exmp}Let $\tilde{X}=\AA_{\cO_L}^n$, $f$ be a non-constant polynomial in $n$ variables over $\cO_L$ and $\chi$ be a multiplicative character of $\cO_L^{\times}$. Let $C$ be a subscheme of $\AA_{k_L}^n$. Let $s$ be a positive rational number. We consider the $\QQ$-divisor $D_{f,s}:=-s\spec(\cO_L[x]/(f))$ of $X$. Then $(\tilde{X},D_{f,s})$ is a log pair.  Moreover,  it is easy to show from the definition that 
$$\#_{\rm st}^{f,\chi}(\tilde{X},D_{f,s})_{C}=q^n\int_{\{x\in \cO_L^n|\overline{x}\in C(k_L)\}}\chi(\acc(f(x)))|f(x)|^{s}|dx|=q^n\cZ_{f,\Phi_{L,C},\chi}(s).$$
\end{example}

\subsection{Actions of finite group schemes}
Recall that we fixed a finite flat group scheme $G$ over $\cO_L$. Since $L$ is of characteristic zero, $G_L$ is étale over $\spec(L)$. Let $M$ be a $G_L$-étale algebra. We denote by $\cO_M$ the integral closure of $\cO_L$ in $M$.
\begin{notn}\label{GHop}
To simplify the notation, we will write the automorphism group $\textnormal{Aut}_{L}^{G_L}(M)$ of $M$ as $_M\cH^{\rm op}$ and its opposite group as $_M\cH$. If $G$ is abelian, we have   $_M\cH=_M\cH^{\rm op}=G(L)$.  
\end{notn}
Let $Y$ be an $\cO_L$-variety equipped with a faithful action of $G$ and let $X=Y/G$ be the quotient variety. For an étale $G_L$-algebra $M$, let $Y(\cO_M)^G$ be the set of $G$-equivariant $\cO_M$-points of $Y$. Here an $\cO_M$-point of $Y$ is said to be $G$-equivariant if the induced $M$-point is $G_L$-equivariant. 
We have a natural map $Y(\cO_M)^G\to X(\cO_L)$. Indeed, given a $G$-equivariant $\cO_M$-point of $Y$, then taking $G_L$-quotients, we get an $L$-point of $X$. We also have the induced $\cO_M$-point of $X$. Thus, get a point in $X(\cO_L)=X(L)\cap X(\cO_M)$. Since this map is $_M \cH$-invariant, we get a map $\lambda_M:Y(\cO_M)^G/_M\cH\to X(\cO_L)$.
\begin{defn}Let $X(\cO_L)^{\natural}$ be the set of $\cO_L$-points $\spec(\cO_L)\to X$ sending the generic point into the unramified locus of $Y\to X$ in $X$. If $M\in G\textnormal{-\'Et}(L)$, let $Y(\cO_M)^{G,\natural}$ be the set of $G$-equivariant $\cO_M$-points sending the generic points into the unramified locus of the same morphism in $Y$. The map $\lambda_M$ restricts to $Y(\cO_M)^{G,\natural}$ yielding the map $\lambda_M:Y(\cO_M)^{G,\natural}/_M\cH\to X(\cO_L)^{\natural}$.
\end{defn}
\begin{prop}[Proposition 5.4, \cite{Yas}]With the above notation, the map
$$\bigsqcup_{M\in G\textnormal{-\'Et}(L)}Y(\cO_M)^{G,\natural}/_M\cH\to X(\cO_L)^{\natural}$$
is bijective.
\end{prop}

\subsection{Untwisting}

Now, we recall the untwisting technique in \cite{Yas}. Let us fix a linear faithful action of $G$ on $\AA_{\cO_L,x}^n=\spec(\cO_L[x_1,...,x_n])$. We denote by $\cO_L[x]_1$ the linear part of $\cO_L[x_1,...,x_n]$. Then $\cO_L[x]_1$ is endowed with a natural action of $G$. Let $M\in G\textnormal{-\'Et}(L)$, we put $\Xi_M=\textnormal{Hom}_{\cO_L}^G(\cO_L[x]_1,\cO_M)$ the set of $G$-equivariant $\cO_L$-homomorphisms and call it the tuning module of $M$. 
Here an $\cO_L$-homomorphism $\cO_L[x]_1\to \cO_M$ is said to be $G$-equivariant if it is $G_L$-equivariant after tensored with $L$. 
It is natural to identify $\Xi_M$ with the set $\AA_{\cO_L,x}^n(\cO_M)^G$ of $G$-equivariant $\cO_M$-points  of $\AA_{\cO_L,x}^n$.

\begin{lem}\label{free}With the above setting, 
$\Xi_{M}$ is a free $\cO_L$-module of rank $n$.
\end{lem}
\begin{proof} 
The proof is basically the same as the one of \cite[Lemma 3.2]{Wood-Yasuda-I}.
The tuning module $\Xi_M$ is torsion-free as an $\cO_L$-module. To see that it has rank $n$, we may take a scalar extention so that we can reduce to the case where $G$ is a constant group over $L$ and the torsor is trivial.
\end{proof}  
Let a basis $\phi_1,...,\phi_n$ of $\Xi_M$ and its dual basis $y_1,...,y_n$ such that
$$\textnormal{Hom}_{\cO_L}(\Xi_M,\cO_L)=\bigoplus_{j=1}^n\cO_Ly_j$$
and
$$\textnormal{Hom}_{\cO_L}(\Xi_M,\cO_M)=\bigoplus_{j=1}^n\cO_My_j.$$
We identify $\textnormal{Hom}_{\cO_L}(\Xi_M,\cO_L)$ and $\textnormal{Hom}_{\cO_L}(\Xi_M,\cO_M)$ with the linear parts of $\cO_L[y]:=\cO_L[y_1,...,y_n]$ and $\cO_M[y]:=\cO_M[y_1,...,y_n]$ respectively. Let us put $\AA_{\cO_L,y}^n:=\spec(\cO_L[y])$ and $\AA_{\cO_M,y}^n:=\spec(\cO_M[y])$.
\begin{defn}\label{untwm}We denote by $u_M^*$ the $\cO_L$-homomorphism from $\cO_L[x]$ to $\cO_M[y]$ given by 
$$u_M^*(x_i):=\sum_{1\leq j\leq n}\phi_j(x_i)y_j$$
and its corresponding $\cO_L$-morphism of schemes $u_M:\AA_{\cO_M,y}^n\to\AA_{\cO_L,x}^n$.
\end{defn}
It is easy to show that $u_M$ is a $G$-equivariant morphism with respect to the action of $G$ on $\AA_{\cO_L,x}^n$ and the action of $G$ on $\AA_{\cO_M,y}^n$ induced by the action of $G$ on $\spec(\cO_M)$ (see \cite[Lemma 6.6]{Yas}).
Moreover, we can prove that $u_M^*(\cO_L[x]^G)\subset \cO_L[y]$ (see \cite[Proposition 6.7]{Yas}). It induces a morphism $\rho^{|M|}:\AA_{\cO_L,y}^n\to \AA_{\cO_L,x}^n/G$ so that the following diagram

$$\begin{tikzcd}
\AA_{\cO_M,y}^n \arrow[r, "\otimes_{\cO_L}\cO_M"] \arrow[d, "u_M"]
& \AA_{\cO_L,y}^n \arrow[d, "\rho^{|M|}" black] \\
\AA_{\cO_L,x}^n \arrow[r, black, "\rho" black]
& |[black, ]| \AA_{\cO_L,x}^n/G
\end{tikzcd}
$$
is commutative, where $\rho$ is the quotient morphism.

The action of $_M\cH=(\textnormal{Aut}^G_M)^{\rm op}(\cO_L)$ on $\Xi_M$ induces $_M\cH$-actions on $\AA_{\cO_L,y}^n(\cO_L)$ and $\AA_{\cO_M,y}^n(\cO_M)^G$, where $\AA^n_{\cO_M,y}(\cO_M)^G$ be the set of $G$-equivariant morphisms $\spec(\cO_M)\to \AA^n_{\cO_M,y}$. In particular, the map $\alpha_1:\AA_{\cO_L,y}^n(\cO_L)\to \AA^n_{\cO_M,y}(\cO_M)^G$ induced by $\otimes_{\cO_L}\cO_M$ is $_M\cH$-equivariant and bijective (see \cite[Lemma 6.11]{Yas}).  On the other hand, the map $\alpha$ that sends $\gamma$ to $u_M\circ\gamma$ is a bijective and $_M\cH$-equivariant map from $\AA_{\cO_M,y}^n(\cO_M)^G$ to $\AA_{\cO_L,x}^n(\cO_M)^G\equiv \Xi_M$ (see \cite[Lemma 6.12]{Yas}).
\begin{defn}
Let $\AA_{\cO_L,y}^n(\cO_L)^{\natural}$ (resp. $\AA_{\cO_M,y}^n(\cO_M)^{G,\natural}$) be the set of $\cO_L$-points (resp. $G$-equivariant $\cO_M$-points) sending the generic points into the locus where $\AA_{\cO_L,y}^n\to \AA_{\cO_L,x}^n/G$ (resp. $\AA_{\cO_M,y}^n\to \AA_{\cO_L,x}^n/G$) is étale. 
\end{defn}
By the above discussion, we have:
\begin{prop}[Proposition 6.14, \cite{Yas}]
The morphism $\Psi^{|M|}$ induces a bijection 
$$\bigsqcup_{M\in G\textnormal{-\'Et}}\AA_{\cO_L,y}^n(\cO_L)^{\natural}/_M\cH\to (\AA_{\cO_L,x}^n/G)(\cO_L)^{\natural}.$$
\end{prop}
Let $V\subset \AA_{\cO_L,x}^n$ be a $G$-stable normal closed $\cO_L$-subvariety such that $G$ induces a faithful action on $V$. Let $W$ be the quotient variety $V/G$ and $\overline{W}$ be the image of $V$ in $\AA_{\cO_L,x}^n/G$. The canonical map $W\to \overline{W}$ is finite and birational. Let $U\subset W$ and $\overline{U}\subset \overline{W}$ be the loci where $V\to W$ and $V\to\overline{W}$ are unramifield then the morphism $W\to\overline{W}$ induces an isomorphism $U\to\overline{U}$.
\begin{defn}Let $M\in G\textnormal{-\'Et}(L)$. With the above notations,  we denote by $V^{|M|}$ and $V^{\langle M\rangle}$ the preimages (with the reduced structure) of $\overline{W}$ under the morphisms $\rho^{|M|}$ and  $\rho^{|M|}\circ (\otimes_{\cO_L}\cO_M)$ respectively. Let $\tilde{V}^{|M|}$ and  $\tilde{V}^{\langle M\rangle}$ be the normalizations of $V^{|M|}$ and $V^{\langle M\rangle}$ respectively. Then the normalization morphisms $\tilde{V}^{|M|}\to V^{|M|}$ and $V^{\langle M\rangle}\to\tilde{V}^{\langle M\rangle}$ are isomorphisms over $V^{|M|}\otimes L$ and  $V^{\langle M\rangle}\otimes M$ respectively.
\end{defn}
\begin{defn}Let $\overline{W}(\cO_L)^{\natural}$ be the set of $\cO_L$-points of $\overline{W}$ sending the generic points into the unramifield locus of $V\to \overline{W}$ (or equivalent $V^{|M|}\to X$). We also let $V^{|M|}(\cO_L)^{\natural}$ (resp. $\tilde{V}^{|M|}(\cO_L)^{\natural}$) be the set of $\cO_L$-points of $V^{|M|}$ (resp. $\tilde{V}^{|M|}$) sending the   generic point into the unramifield locus of $V^{|M|}\to\overline{W}$ (resp.  $\tilde{V}^{|M|}\to\overline{W}$).

\end{defn}
We have the following result.
\begin{prop}[Proposition 6.16, \cite{Yas}]\label{decompo}
With the above notations, the maps 
$$\bigsqcup_{M\in G\textnormal{-\'Et}(L)}\frac{V^{|M|}(\cO_L)^{\natural}}{_M\cH}\to \overline{W}(\cO_L)^{\natural}$$
and 
$$\bigsqcup_{M\in G\textnormal{-\'Et}(L)}\frac{\tilde{V}^{|M|}(\cO_L)^{\natural}}{_M\cH}\to W(\cO_L)^{\natural}$$
are bijections. 
\end{prop}

\subsection{A \(p\)-adic McKay correspondence}\label{twpMcK}

\begin{defn}A $G$-log pair is a log pair $(V,E)$ with a faithful action $\varphi$ of $G$ on $V$ such that $E$ is $G$-stable, i.e. we have an action of $G$ on $E$ induced by restricting $\varphi$ to $E$.
\end{defn}

Let $(V,E)$ be a $G$-log pair and $W=V/G$ then there is a unique $\QQ$-divisor $D$ of $W$ such that  $(W,D)$ is a log-pair and the quotient morphism $V\to W$ induces a crepant morphism $(V,E)\to (W,D)$ (see \cite[Lemma 7.2]{Yas}). Suppose that $V$ is affine then there exists a positive integer $n$, a faithful action of $G$ on $\AA_{\cO_L}^n$ and a $G$-equivariant closed embedding $V\hookrightarrow\AA_{\cO_L}^n$ (see \cite[Section 7.2]{Yas}). With the above notations, we can define  $G$-log pair $(\tilde{V}^{\langle M\rangle},\tilde{E}^{\langle M\rangle})$, log pair $(\tilde{V}^{|M|},\tilde{E}^{|M|})$ such that the morphisms $(\tilde{V}^{\langle M\rangle},\tilde{E}^{\langle M\rangle})\to (\tilde{V}^{|M|},\tilde{E}^{|M|})$ and  $(\tilde{V}^{\langle M\rangle},\tilde{E}^{\langle M\rangle})\to (V,E)$ are crepant and so is the induced morphism $(\tilde{V}^{|M|},\tilde{E}^{|M|})\to (W,D)$. For a subscheme $C$ of $W_{k_L}$ and $M\in G\textnormal{-\'Et}(L)$, we denote by $\tilde{C}^{|M|}$ the preimage of $C$ in $\tilde{V}_{k_L}^{|M|}$ (with the reduced structure). Let $g$ be an $\cO_L$-morphism from a Zariski neighbourhood $W_0$ of $C$ in $W$ to $\AA_{\cO_L}^1$, we denote by $g_{|M|}:(\tilde{\rho}_V^{|M|})^{-1}(W_0)\to \AA_{\cO_L}^1$ the composition of $g$ with the morphism $\tilde{\rho}_V^{|M|}:(\tilde{V}^{|M|},\tilde{E}^{|M|})\to (W,D)$.
\begin{defn}\label{M-st} Let $G,V,E$ be as above and $C$ be a subscheme of $W_{k_L}$. Let $g$ be an $\cO_L$-morphism from a Zariski neighbourhood of $C$ in $W$ to $\AA_{\cO_L}^1$ and $\chi$ be a multiplicative character of $\cO_L^{\times}$. For $M\in G\textnormal{-\'Et}(L)$, we define the $M$-stringy point count of $(V,E)$ along $C$ twisted by $(g,\chi)$ as
$$\#_{\rm st}^{g,\chi,M}(V,E)_C:=\#_{\rm st}^{g_{|M|},\chi}(\tilde{V}^{|M|},\tilde{E}^{|M|})_{\tilde{C}^{|M|}}/\#(_M\cH)$$
and the $G$-stringy point count of $(V,E)$ along $C$ twisted by $(g,\chi)$ as
$$\#_{\rm st}^{g,\chi,G}(V,E)_C:=\sum_{M\in G\textnormal{-\'Et}(L)} \#_{\rm st}^{g,\chi,M}(V,E)_C.$$
\end{defn}

\begin{thm}\label{stringc}Let $(V,E)$ be a $G$-log pair with $V$ affine, let $(W,D)$ be as above. Let $C$ be a subscheme of $W_{k_L}$ and $g$ be an $\cO_L$-morphism from a Zariski neighbourhood of $C$ in  $W$ to $\AA_{\cO_L}^1$. Let $\chi$ be a multiplicative character of $\cO_L^{\times}$. Then we have
$$\#_{\rm st}^{g,\chi,G}(V,E)_C=\#_{\rm st}^{g,\chi}(W,D)_{C}.$$
\end{thm}
\begin{proof}Note that for each $M\in G\textnormal{-\'Et}(L)$,  the map $\sigma_M:\tilde{V}^{|M|}(\cO_L)^{\natural}\to W(\cO_L)^{\natural}$ is a $_M\cH$-equivariant étale morphism of $L$-analytic manifolds and the quotient manifold  $\tilde{V}^{|M|}(\cO_L)^{\natural}/_M\cH$ is isomorphic with the image of $\sigma_M$ as showed in Proposition \ref{decompo}. Moreover, the definition of $_M\cH$ assures that the action of $_M\cH$ on $\tilde{V}^{|M|}(\cO_L)^{\natural}$ is faithful. Thus we still have a version of Theorem \ref{change} with our situation following by the proof  of \cite[Theorem 4.8]{Yas}. Therefore we can proceed similarly as in \cite[Theorem 7.4]{Yas}.
\end{proof}
\begin{remark}The condition that $V$ is affine in the above statements is not necessary as explained in \cite[Section 7.3]{Yas}. 
\end{remark}
\subsection{Stringy point count associated with finite diagonalizable group schemes}\label{compumu}
Suppose that $G$ is an $\cO_L$-diagonalizable subgroup scheme of $\GG_{m,\cO_L}^n$ and $G$ is finite over $\cO_L$. Let $\chi$ be a multiplicative character of $\cO_L^{\times}$ and $(W_0,g)$ be a germ of a regular function (over $\cO_L$) on $A_G=\AA_{\cO_L,x}^n/G$ at $P=\rho_G(\varpi_L\cO_L+(x_1,...,x_n))$, where $\rho_G: \AA_{\cO_L,x}^n\to A_G$ is the quotient morphism and $W_0$ be an open subscheme of $A_G$ containing $P$. The relative canonical divisor $K_{\AA_{\cO_L,x}^n,A_G}$ is effective and its support is contained in $\AA_{k_L}^n\cup\left(\cup_{i=1}^n E_i\right)$, where $E_i$ is defined in Definition \ref{dedivis}.  Thus, we can write $K_{\AA_{\cO_L,x}^n,A_G}=\mathfrak{l}\AA_{k_L}^{n}+\sum_{i=1}^n (\mathfrak{v}_i-1)E_i$ for natural numbers $\mathfrak{l}\geq 0$ and $\mathfrak{v}_i\geq 1$ for $1\leq i\leq n$.  Let $D$ be a $\QQ$-divisor on $A_G$ such that $(A_G,D)$ is an $\cO_L$-log pair. Suppose that $\Div(g\circ\rho_G)=\mathfrak{N}\AA_{k_L}^{n}|_{\rho_G^{-1}(W_0)}+\sum_{i=1}^n N_iE_i|_{\rho_G^{-1}(W_0)}$ and $\rho_G^{*}(D)=\tilde{\mathfrak{N}}\AA_{k_L}^{n}+\sum_{i=1}^n \tilde{N}_iE_i$, where  $\mathfrak{N}\geq 0, N_i\geq 0$  are natural numbers and $\tilde{\mathfrak{N}}, \tilde{N}_i$ are rational numbers for $1\leq i\leq n$. We will compute $\#_{\rm st}^{g,\chi}(A_G,D)_{P}$.

Since $G$ is an $\cO_L$-diagonalizable subgroup scheme of $\GG_{m,\cO_L}^n$, Lemma \ref{closimm} implies that there is a surjective Hopf algebra homomorphism $\mathbf{h}:\cO_L[z_1,z_1^{-1},...,z_n,z_n^{-1}]\to\cO_G$. 
 By Lemma \ref{finimugr}, 
 we can suppose that $G=\prod_{j=1}^\ell\mu_{d_j,\cO_L}=\Spec\cO_L[\zeta_1,...,\zeta_\ell]/(\zeta_1^{d_1}-1,...,\zeta_{\ell}^{d_\ell}-1)$ for positive integers $d_j$.
Then we have \[\cO_G=\cO_L[\zeta_1,...,\zeta_\ell]/(\zeta_1^{d_1}-1,...,\zeta_{\ell}^{d_\ell}-1)\]
and the comultiplication $\Delta:\cO_G\to \cO_G\otimes_{\cO_L}\cO_G, \zeta_j\mapsto \zeta_j\otimes \zeta_j$ for all $j$. The Hopf algebra homomorphism $\mathbf{h}$ sends each  group-like element of $\cO_L[z_1,z_1^{-1},...,z_n,z_n^{-1}]$ to a group-like element of $\cO_G$, thus we can suppose that $\mathbf{h}(z_i)=\prod_{1\leq j\leq \ell}\zeta_j^{\alpha_{ij}}$ for $1\leq i\leq n$ and suitable natural numbers  $\alpha_{ij}$. The action of $\GG_{m,\cO_L}^n$ on $\AA_{\cO_L,x}^n$ is given by the ring homomorphism $\varphi:\cO_L[x]=\cO_L[x_1,...,x_n]\to\cO_L[x_1,...,x_n]\otimes_{\cO_L} \cO_L[z_1,z_1^{-1},...,z_n,z_n^{-1}], x_i\mapsto x_i\otimes z_i$. The action of $G$ on $\AA_{\cO_L}^n$ induced by $\varphi$ and $\mathbf{h}$ is given by 
\begin{align*}
\varphi_G:=(\mathbf{id}_{\cO_L[x]}\otimes \textbf{h})\circ\varphi: \cO_L[x_1,...,x_n]&\to \cO_L[x_1,...,x_n]\otimes_{\cO_L}\cO_L[\zeta_1,...,\zeta_\ell]/(\zeta_1^{d_1}-1,...,\zeta_\ell^{d_\ell}-1),\\x_i&\mapsto x_i\otimes\prod_{1\leq j\leq \ell}\zeta_j^{\alpha_{ij}}.\end{align*}

By Hilbert's theorem 90, any $M\in G\textnormal{-\'Et}(L)$ is isomorphic with an $L$-algebra of the form 
$M_{(\upsilon,\xi)}:=L[t_1,...,t_\ell]/(t_1^{d_1}-\xi_1\varpi_L^{\upsilon_1},...,t_\ell^{d_\ell}-\xi_\ell\varpi_L^{\upsilon_\ell}),$ where $(\upsilon,\xi)=((\upsilon_1,\xi_1),...,(\upsilon_\ell,\xi_\ell))\in (\NN\times \cO_L^{\times})^\ell$ and the action of $G_L$ on $M_{(\upsilon,\xi)}$ is given by $\tau_{\upsilon,\xi}:M_{(\upsilon,\xi)}\to M_{(\upsilon,\xi)}\otimes_{\cO_L}\cO_G, t_j\mapsto t_j\otimes\zeta_j$. For each $1\leq j\leq \ell$, let $A_{j}=\{\xi_{j1},...,\xi_{jk_j}\}$ be a set of representatives of coset in $\cO_L^{\times}/(\cO_L^{\times})^{d_j}$ then we can take  $G\textnormal{-\'Et}(L)=\{M_{(\upsilon,\xi)}\mid \xi_j\in A_j, 0\leq \upsilon_j\leq d_j-1 \textnormal{ for } 1\leq j\leq \ell\}$. So we can identify $G\textnormal{-\'Et}(L)$ with the set $\cS=\{(\upsilon,\xi)\mid \xi_j\in A_j, 0\leq \upsilon_j\leq d_j-1 \textnormal{ for } 1\leq j\leq \ell\}$. 

 Let $(\upsilon,\xi)\in \cS$. We will find a basis of the tuning module $\Xi_{M_{(\upsilon,\xi)}}$ associated with $G, M_{(\upsilon,\xi)}, \varphi_G, \tau_{\upsilon,\xi}$. By definition, we have 
$$\Xi_{M_{(\upsilon,\xi)}}:=\{\phi\in \textnormal{Hom}_{\cO_L}(\cO_L[x]_1\to \cO_{M_{(\upsilon,\xi)}})\mid (\phi_L\otimes\textbf{id}_{\cO_G})\circ \varphi_{GL}=\tau_{\upsilon,\xi}\circ\phi_L\},$$
where $\phi_L:L[x]_1\to M_{(\upsilon,\xi)}$ (resp.  $\varphi_{GL}:L[x]\to L[x]\otimes_{\cO_L}\cO_G$) is the extension of scalars to $L$ of $\phi$ (resp. $\varphi_G$). 
\begin{lem}\label{generator}Let $1\leq i\leq n$ and $r_{i}$ be the smallest positive integer such that \[\left(\prod_{1\leq j\leq \ell}\zeta_j^{\alpha_{ij}}\right)^{r_{i}}=1.\] Then,  there is an element $a_{i,\upsilon,\xi}\in \cO_{M_{(\upsilon,\xi})}\setminus\{0\}$ such that 
$$\{a\in \cO_{M_{(\upsilon,\xi)}}\mid \exists \phi\in \Xi_{M_{(\upsilon,\xi)}}, \phi(x_i)=a\}=\{ca_{i,\upsilon,\xi}\mid c\in \cO_L\}\subset\{a\in \cO_{M_{(\upsilon,\xi)}}\mid a^{r_{i}}\in\cO_L\}.$$
More precisely, we can take 
$$a_{i,\upsilon,\xi}=\varpi_L^{\beta_{i,\upsilon}}\prod_{1\leq j\leq \ell}t_j^{\alpha_{ij}}$$ with $$\beta_{i,\upsilon}=-\bigg\lfloor\sum_{1\leq j\leq \ell}\frac{\alpha_{ij}\upsilon_j}{d_j}\bigg\rfloor,$$
 where if $a$ is a real number, then $\lfloor a\rfloor$ is the largest integer less than or equal to $a$ as usual.
\end{lem}
\begin{proof}
Firstly, we will show that 
\begin{equation}\label{eq0001}\{a\in \cO_{M_{(\upsilon,\xi)}}\mid \exists \phi\in \Xi_{M_{(\upsilon,\xi)}}, \phi(x_i)=a\}=\{a\in \cO_{M_{(\upsilon,\xi)}}\mid a\otimes \prod_{1\leq j\leq \ell}\zeta_j^{\alpha_{ij}} =\tau_{\upsilon,\xi}(a)\}\neq \{0\}.
\end{equation}
Indeed, if $\phi\in \Xi_{M_{(\upsilon,\xi)}}$ and $a=\phi(x_i)$, then we have 

$$((\phi_L\otimes\textbf{id}_{\cO_G})\circ \varphi_{GL})(x_i)=(\tau_{\upsilon,\xi}\circ\phi_L)(x_i)
\Leftrightarrow a\otimes\prod_{1\leq j\leq \ell}\zeta_j^{\alpha_{ij}}=\tau_{\upsilon,\xi}(a).$$
Moreover, if $a\in\cO_{M_{(\upsilon,\xi)}}$ such that $a\otimes \prod_{1\leq j\leq \ell}\zeta_j^{\alpha_{ij}} =\tau_{\upsilon,\xi}(a)$ and $\phi_{ia}:\cO_L[x]_1\to \cO_{M_{(\upsilon,\xi)}}$ is the $\cO_L$-module homomorphism given by $\phi_{ia}(x_i)=a$ and $\phi_{ia}(x_j)=0$ if $j\neq i$ then $\phi_{ia}\in \Xi_{M_{(\upsilon,\xi)}}$. On the other hand, we have 
\begin{equation}\label{eqtau} \prod_{1\leq j\leq \ell}t_j^{\alpha_{ij}} \neq 0 \textnormal{ and } \tau_{\upsilon,\xi}(\prod_{1\leq j\leq \ell}t_j^{\alpha_{ij}})=(\prod_{1\leq j\leq \ell}t_j^{\alpha_{ij}})\otimes (\prod_{1\leq j\leq \ell}\zeta_j^{\alpha_{ij}}).
\end{equation}
Thus, (\ref{eq0001}) is proved. 

Let $a\in \cO_{M_{(\upsilon,\xi)}}$ such that $a\otimes\prod_{1\leq j\leq \ell}\zeta_j^{\alpha_{ij}}=\tau_{\upsilon,\xi}(a)$. Then, we have \[a^{r_{i}}\otimes 1=a^{r_{i}}\otimes\left(\prod_{1\leq j\leq \ell}\zeta_j^{\alpha_{ij}}\right)^{r_{i}}=\tau_{\upsilon,\xi}(a^{r_{i}}).\] Since $M_{(\upsilon,\xi)}\in G\textnormal{-\'Et}(L)$, we have $$M_{(\upsilon,\xi)}^{G_L}=\{b \in M_{(\upsilon,\xi)}| \tau_{\upsilon,\xi}(b)=b\otimes 1\}=L.$$
Therefore,  $a^{r_{i}}\in L$. Note that $\cO_{M_{(\upsilon,\xi)}}$ is the integral closure of $\cO_L$ in $M_{(\upsilon,\xi)}$, $a\in \cO_{M_{(\upsilon,\xi)}}$ and $\cO_L$ is integrally closed domain. Thus, $a^{r_i}\in L\cap \cO_{M_{(\upsilon,\xi)}}=\cO_L$.  This implies that \begin{equation}\label{eq0002}\{a\in \cO_{M_{(\upsilon,\xi)}}\mid a\otimes \prod_{1\leq j\leq \ell}\zeta_j^{\alpha_{ij}} =\tau_{\upsilon,\xi}(a)\}\subset\{a\in \cO_{M_{(\upsilon,\xi)}}\mid a^{r_{i}}\in\cO_L\}.
\end{equation}
By (\ref{eq0001}) and (\ref{eq0002}), to prove the first assertion, it remains to show that if $a,b\in \cO_{M_{(\upsilon,\xi)}}\setminus\{0\}$ such that $a\otimes \prod_{1\leq j\leq \ell}\zeta_j^{\alpha_{ij}}=\tau_{\upsilon,\xi}(a)$  and $b\otimes\prod_{1\leq j\leq \ell}\zeta_j^{\alpha_{ij}}=\tau_{\upsilon,\xi}(b)$ then  $a\in \cO_Lb$ or $b\in \cO_La$.   Indeed, for such $a$ and $b$, it follows from (\ref{eq0002}) that  $a^{r_{i}}, b^{r_{i}}\in \cO_L$. On the other hand, we have $M_{(\upsilon,\xi)}=\cO_{M_{(\upsilon,\xi)}}\otimes_{\cO_L} L=\prod_{1\leq j\leq l_{\upsilon,\xi}}L_{j,\upsilon,\xi}$ for some finite extensions $L_{j,\upsilon,\xi}$ of $L$. Thus, $\cO_{M_{(\upsilon,\xi)}}=\prod_{1\leq j\leq l_{\upsilon,\xi}}\cO_{L_{j,\upsilon,\xi}}$ is a reduced ring and  $a^{r_{i}}, b^{r_{i}}\in \cO_L\setminus\{0\}$.  Therefore, $a$ and $b$ must be invertible elements of $M_{(\upsilon,\xi)}$.  Since $\cO_L$ is a valuation ring, we have either $(\frac{a}{b})^{r_{i}}\in\cO_L$ or  $(\frac{b}{a})^{r_{i}}\in \cO_L$. Since $\cO_{M_{(\upsilon,\xi)}}$ is the integral closure of $\cO_L$ in $M_{(\upsilon,\xi)}$, we must have  $\frac{a}{b}\in \cO_{M_{(\upsilon,\xi)}}$ or $\frac{b}{a}\in \cO_{M_{(\upsilon,\xi)}}$. Without loss of generality, suppose that $c=\frac{a}{b}\in \cO_{M_{(\upsilon,\xi)}}$.  We have $\tau_{\upsilon,\xi}(b)\tau_{\upsilon,\xi}(c)=\tau_{\upsilon,\xi}(a)=a\otimes\prod_{1\leq j\leq \ell}\zeta_j^{\alpha_{ij}}=(b\otimes\prod_{1\leq j\leq \ell}\zeta_j^{\alpha_{ij}}).(c\otimes 1)=\tau_{\upsilon,\xi}(b).(c\otimes 1)$. But $\tau_{\upsilon,\xi}(b)^{r_{i}}=b^{r_{i}}\otimes 1\in \cO_L\setminus\{0\}\subset L\setminus\{0\}\subset (M_{(\upsilon,\xi)}\otimes_{\cO_L}\cO_G)^{\times}$, therefore $\tau_{\upsilon,\xi}(b)$ is not a zero divisor of $M_{(\upsilon,\xi)}\otimes_{\cO_L}\cO_G$. This implies $c\otimes 1=\tau_{\upsilon,\xi}(c)$. Thus, $c\in M_{(\upsilon,\xi)}^{G_L}\cap \cO_{M_{(\upsilon,\xi)}}=L\cap \cO_{M_{(\upsilon,\xi)}}=\cO_L$ and $a\in\cO_Lb$.

The second assertion follows from the first assertion,  (\ref{eq0001}), (\ref{eqtau}), (\ref{eq0002}) and the fact that if $e\in\cO_L$ then $\ord_L(e)\in\ZZ\cup\{+\infty\}$ and
$$\left(e\prod_{1\leq j\leq \ell}t_j^{\alpha_{ij}}\right)^{r_{i}}\in \cO_L\Leftrightarrow \ord_L(e)+\sum_{j=1}^\ell \frac{\alpha_{ij}\upsilon_j}{d_j}\geq 0\Leftrightarrow \ord_L(e)\geq \beta_{i,\upsilon}.$$
\end{proof}
\begin{cor}\label{corab}Let $r_{i}$ and $a_{i,\upsilon,\xi}$ be as in Lemma \ref{generator} for $1\leq i\leq n$. Then $b_{i,\upsilon,\xi}:=a_{i,\upsilon,\xi}^{r_{i}}\in \cO_L\setminus\{0\}$ and $|b_{i,\upsilon,\xi}|^{\frac{1}{r_{i}}}=q_L^{-\gamma_{i,\upsilon}}$ for all $1\leq i\leq n$, where $$\gamma_{i,\upsilon}=\left(\sum_{1\leq j\leq \ell}\frac{\alpha_{ij}\upsilon_j}{d_j}\right)-\bigg\lfloor\sum_{1\leq j\leq \ell}\frac{\alpha_{ij}\upsilon_j}{d_j}\bigg\rfloor.$$ Moreover, $\Xi_{M_{(\upsilon,\xi)}}$ is a free $\cO_L$-module of rank $n$ with a basis consisting of the elements $\phi_{i,\upsilon,\xi}:\cO_L[x]_1\to \cO_{M_{(\upsilon,\xi)}}, x_j\mapsto \delta_{ij}a_{i,\upsilon,\xi}$ for $1\leq j\leq n$, where $\delta_{ij}$ is the Kronecker delta.
\end{cor}
\begin{proof}
The first assertion follows immediately from Lemma \ref{generator}. The second assertion follows from Lemmas \ref{free} and \ref{generator}. 
\end{proof}
By Definition \ref{untwm}, we have the ring homomorphism 
\begin{align}\label{mapU}u_{M_{(\upsilon,\xi)}}^*: \cO_L[x_1,...,x_n]&\to\cO_{M_{(\upsilon,\xi)}}[y_1,...,y_n],\\
x_i&\mapsto a_{i,\upsilon,\xi}y_i.\nonumber
\end{align}

We use the notations in Subsection \ref{twpMcK}. We set $V=\AA_{\cO_L,x}^n$. Then $$\tilde{V}^{\langle M_{(\upsilon,\xi)}\rangle}=\AA_{\cO_{M_{(\upsilon,\xi)}},y}^n \textnormal{ and }\tilde{V}^{|M_{(\upsilon,\xi)}|}=\AA_{\cO_L,y}^n.$$ Now, we set $E=\rho_G^*(D)-K_{V/A_G}$. Then it follows from \cite[Lemma 4.6]{Yas} that $(V,E)$ is an $\cO_L$-log pair and the morphism $\rho_G:(V,E)\to (A_G,D)$ is crepant. In order to compute the $M_{(\upsilon,\xi)}$-stringy point count of $(V,E)$
along $P$ twisted by $(g,\chi)$, we will compute the preimage $\tilde{P}^{|M_{(\upsilon,\xi)}|}$ of $P$ under the morphism  $\rho_G^{|M_{(\upsilon,\xi)}|}:\tilde{V}^{|M_{(\upsilon,\xi)}|}\to A_G$ together with $\tilde{E}^{|M_{(\upsilon,\xi)}|}$ and $g_{|M_{(\upsilon,\xi)}|}$.

\begin{lem}\label{comdivi}With the assumptions at the beginning of this subsection and the notations in Corollary  \ref{corab}, we have 
$$\tilde{E}^{|M_{(\upsilon,\xi)}|}=\left(\sum_{i=1}^n (\tilde{N}_{i}-\mathfrak{v}_{i}+1)\hat{E}_i\right) + (\tilde{\mathfrak{N}}-\mathfrak{l})(\varpi_L)+\left(\sum_{i=1}^n(\tilde{N}_{i}-\mathfrak{v}_{i})\frac{1}{r_{i}}(b_{i,\upsilon,\xi})\right),$$
where $\hat{E}_i$, $(\varpi_L)$ and $(b_{i,\upsilon,\xi})$ are the divisors of $\tilde{V}^{|M_{(\upsilon,\xi)}|}$ associated with the principal ideals generated by $y_i$, $\varpi_L$ and $b_{i,\upsilon,\xi}$ respectively.
\end{lem}
\begin{proof}

We can write $$u_{M_{(\upsilon,\xi)}}=(\otimes_{\cO_L,x}\cO_{M_{(\upsilon,\xi)}})\circ \Pi_{M_{(\upsilon,\xi)}},$$ where $$\otimes_{\cO_L,x}\cO_{M_{(\upsilon,\xi)}}:\AA_{\cO_{M_{(\upsilon,\xi)},x}}^n\to V$$ is the scalar extension  and $$\Pi_{M_{(\upsilon,\xi)}}:\tilde{V}^{\langle M_{(\upsilon,\xi)}\rangle}=\AA_{\cO_{M_{(\upsilon,\xi)},y}}^n \to \AA_{\cO_{M_{(\upsilon,\xi)},x}}^n$$ is the morphism associated to the ring homomorphism \begin{align}\label{mapPi}\Pi_{M_{(\upsilon,\xi)}}^*: \cO_{M_{(\upsilon,\xi)}}[x_1,...,x_n]&\to\cO_{M_{(\upsilon,\xi)}}[y_1,...,y_n],\\ x_i&\mapsto a_{i,\upsilon,\xi}y_i.\nonumber
\end{align} 

 Moreover, it is easy to verify that \begin{equation}\label{canoequ}K_{\tilde{V}^{\langle M_{(\upsilon,\xi)}\rangle}/\tilde{V}^{|M_{(\upsilon,\xi)}|}}=\Pi_{M_{(\upsilon,\xi)}}^*(K_{\AA_{\cO_{M_{(\upsilon,\xi)}},x}^n/V}),
\end{equation}
where both morphisms $\tilde{V}^{\langle M_{(\upsilon,\xi)}\rangle}\to \tilde{V}^{|M_{(\upsilon,\xi)}|}$ and $\AA_{\cO_{M_{(\upsilon,\xi)}},x}^n\to V$ are given by the scalar extension $\otimes_{\cO_L}\cO_{M_{(\upsilon,\xi)}}$.
Thus, it follows from  (\ref{mapU}), (\ref{mapPi}), (\ref{canoequ}), Subsection \ref{twpMcK} and \cite[Lemma 4.6]{Yas} that \begin{align*}
\tilde{E}^{\langle M_{(\upsilon,\xi)}\rangle}&=u_{M_{(\upsilon,\xi)}}^{*}(E)-K_{\tilde{V}^{\langle M_{(\upsilon,\xi)}\rangle}/V}\\
&=u_{M_{(\upsilon,\xi)}}^{*}(\rho_G^*(D)-K_{V/A_G})-K_{\tilde{V}^{\langle M_{(\upsilon,\xi)}\rangle}/\AA_{\cO_{M_{(\upsilon,\xi)}},x}^n}+ \Pi_{M_{(\upsilon,\xi)}}^*(K_{\AA_{\cO_{M_{(\upsilon,\xi)}},x}^n/V})\\
&=u_{M_{(\upsilon,\xi)}}^{*}\left((\tilde{\mathfrak{N}}-\mathfrak{l})(\varpi_L)+\sum_{i=1}^n(\tilde{N}_i-\mathfrak{v}_i+1)E_i\right)-\left(\sum_{i=1}^n(a_{i,\upsilon,\xi})\right)+ K_{\tilde{V}^{\langle M_{(\upsilon,\xi)}\rangle}/\tilde{V}^{|M_{(\upsilon,\xi)}|}}\\
&=(\tilde{\mathfrak{N}}-\mathfrak{l})(\varpi_L)+\sum_{i=1}^n \left[(\tilde{N}_i-\mathfrak{v}_i+1)\tilde{\cE}_i+(\tilde{N}_i-\mathfrak{v}_i)(a_{i,\upsilon,\xi})\right]+ K_{\tilde{V}^{\langle M_{(\upsilon,\xi)}\rangle}/\tilde{V}^{|M_{(\upsilon,\xi)}|}},
\end{align*}  
where $\tilde{\cE}_i$, $(\varpi_L)$ and $(a_{i,\upsilon,\xi})$ are divisors of $\tilde{V}^{\langle M_{(\upsilon,\xi)}\rangle}$ associated with the principal ideals generated by $y_i$, $\varpi_L$ and $a_{i,\upsilon,\xi}$ respectively. Therefore,  \cite[Lemma 7.2]{Yas} and Corollary \ref{corab} imply 
\begin{align*} 
\tilde{E}^{|M_{(\upsilon,\xi)}|}=&\frac{1}{d(G)}(\otimes_{\cO_L}\cO_{M_{(\upsilon,\xi)}})_{*}\left(\tilde{E}^{\langle M_{(\upsilon,\xi)}\rangle}-K_{\tilde{V}^{\langle M_{(\upsilon,\xi)}\rangle}/\tilde{V}^{|M_{(\upsilon,\xi)}|}}\right)
\\
=&\frac{1}{d(G)}(\otimes_{\cO_L}\cO_{M_{(\upsilon,\xi)}})_{*}\left(\sum_{i=1}^n (\tilde{N}_{i}-\mathfrak{v}_{i}+1)\tilde{\cE}_i\right)\\
&+ \frac{1}{d(G)}(\otimes_{\cO_L}\cO_{M_{\upsilon,\xi)}})_{*}\left((\tilde{\mathfrak{N}}-\mathfrak{l})(\varpi_L)+\sum_{i=1}^n(\tilde{N}_i-\mathfrak{v}_i)(a_{i,\upsilon,\xi})\right)\\=&\left(\sum_{i=1}^n (\tilde{N}_{i}-\mathfrak{v}_{i}+1)\hat{E}_i\right)+ (\tilde{\mathfrak{N}}-\mathfrak{l})(\varpi_L)+\left(\sum_{i=1}^n(\tilde{N}_{i}-\mathfrak{v}_{i})\frac{1}{d(G)}(\textnormal{Nm}_{\cO_M/\cO_L}(a_{i,\upsilon,\xi}))\right)\\
=&\left(\sum_{i=1}^n (\tilde{N}_{i}-\mathfrak{v}_{i}+1)\hat{E}_i\right) + (\tilde{\mathfrak{N}}-\mathfrak{l})(\varpi_L)+\left(\sum_{i=1}^n(\tilde{N}_{i}-\mathfrak{v}_{i})\frac{1}{r_{i}}(b_{i,\upsilon,\xi})\right).
\end{align*}
\end{proof}
\begin{lem}\label{compg}We set $V_0=\rho_G^{-1}(W_0)$ and $\tilde{W}_0=\left(\rho_G^{|M_{(\upsilon,\xi)}|}\right)^{-1}(W_0)$. Let $w$ be the standard unit factor of $g$ on $W_0$ as in Lemma-Definition \ref{factorunit}, then the quality $$g\circ\rho_G=(w\circ\rho_G)\varpi_L^{\mathfrak{N}}\prod_{i=1}^n x_i^{N_i}$$
holds in $\cO_V(V_0)$. Moreover, with the notations in Subsection \ref{twpMcK}, the equality
$$g_{|M_{(\upsilon,\xi)}|}=w_{|M_{(\upsilon,\xi)}|}\varpi_L^{\mathfrak{N}}\prod_{i=1}^n a_{i,\upsilon,\xi}^{N_{i}}y_i^{N_{i}}$$
holds in $\cO_{\tilde{V}^{|M_{(\upsilon,\xi)}|}}(\tilde{W}_0)$. In particular, for all $Q_1\in \tilde{P}^{|M_{(\upsilon,\xi)}|}=\left(\rho_G^{|M_{(\upsilon,\xi)}|}\right)^{-1}(P)$ and $Q_2\in \rho_G^{-1}(P)$, we have 
$$w(\rho_G(Q_2))=w(P)=w_{|M_{(\upsilon,\xi)}|}(Q_1)=\mathfrak{u}(g),$$
where $\mathfrak{u}(g)$ is defined in Lemma-Definition \ref{factorunit}.
In addition,  we have $\prod_{i=1}^n a_{i,\upsilon,\xi}^{N_{i}}\in\cO_L$ and $d_j\mid \sum_{i=1}^n\alpha_{ij}N_i$ for all $1\leq j\leq \ell$.
\end{lem}
\begin{proof}Note that the morphism $\rho_G^*:\cO_{A_G}(A_G)\to \cO_V(V)$ is the inclusion morphism $\cO_L[x_1,...,x_n]^G\to \cO_L[x_1,...,x_n]$. Since $\Div(g\circ\rho_G)=\mathfrak{N}\AA_{k_L}^{n}|_{V_0}+\sum_{i=1}^n N_iE_i|_{V_0}$, we use Lemma-Definition \ref{factorunit} to  have $x_1^{N_{1}}...x_n^{N_{n}}\in\cO_L[x_1,...,x_n]^G$ and  $$g\circ\rho_G=(w\circ\rho_G)\varpi_L^{\mathfrak{N}}\prod_{i=1}^n x_i^{N_i}.$$
Let $\hat{g}$ be the $\cO_L$-morphism from $\AA_{\cO_L}^n/G_e\to \AA_{\cO_L}^1$ corresponding to the element $\varpi_L^{\mathfrak{N}}\prod_{i=1}^nx_i^{N_{i}}\in \cO_L[x_1,...,x_n]^G$ then $\rho_{G}^*(\hat{g})=\varpi_L^{\mathfrak{N}}\prod_{i=1}^nx_i^{N_{i}}$.  Moreover, since  $\rho_G^{|M_{(\upsilon,\xi)}|}\circ (\otimes_{\cO_L}\cO_{M_{(\upsilon,\xi)}})=\rho_G\circ u_{M_{(\upsilon,\xi)}}$, we have    $$\left(\rho_G^{|M_{(\upsilon,\xi)}|}\circ (\otimes_{\cO_L}\cO_{M_{(\upsilon,\xi)}})\right)^{*}(\hat{g})=u_{M_{(\upsilon,\xi)}}^{*}\rho_{G}^*(\hat{g})=\varpi_L^{\mathfrak{N}}\prod_{i=1}^n a_{i,\upsilon,\xi}^{N_{i}}y_i^{N_{i}}.$$
By using the notations in Definition \ref{M-st} and the fact that the map $(\otimes_{\cO_L}\cO_{M_{(\upsilon,\xi)}})^{*}$ is the inclusion $\cO_L[y_1,...,y_n]\to \cO_{M_{(\upsilon,\xi)}}[y_1,...,y_n]$, we have  $\prod_{i=1}^n a_{i,\upsilon,\xi}^{N_i}\in \cO_L$ and $\hat{g}_{|M_{(\upsilon,\xi)}|}=\hat{g}\circ \rho_G^{|M_{(\upsilon,\xi)}|}=\varpi_L^{\mathfrak{N}}\prod_{i=1}^n a_{i,\upsilon,\xi}^{N_{i}}y_i^{N_{i}}$. Thus, we have  $$g_{|M_{(\upsilon,\xi)}|}=g\circ \rho_G^{|M_{(\upsilon,\xi)}|}=\left(\hat{g}_{|M_{(\upsilon,\xi)}|}\right)\big|_{\tilde{W}_0}w_{|M_{(\upsilon,\xi)}|}=w_{|M_{(\upsilon,\xi)}|}\varpi_L^{\mathfrak{N}}\prod_{i=1}^n a_{i,\upsilon,\xi}^{N_{i}}y_i^{N_{i}}.$$

The third assertion follows immediately from Lemma-Definition \ref{factorunit} together with the facts that $\rho_G(Q_2)=P$ and $w_{|M_{(\upsilon,\xi)}|}=w\circ\rho_G^{|M_{(\upsilon,\xi)}|}.$

For the last assertion, we showed above that 
$\prod_{i=1}^n a_{i,\upsilon,\xi}^{N_i}\in \cO_L$. By Lemma \ref{generator}, $\prod_{i=1}^n a_{i,\upsilon,\xi}^{N_i}\in \cO_L$ if and only if $\prod_{j=1}^\ell  t_j^{\sum_{i=1}^n\alpha_{ij}N_i}\in \cO_L$. On the other hand, $\{\prod_{j=1}^\ell  t_j^{e_j}\mid 0\leq e_j\leq d_j-1\hspace{0.1cm}\forall j\}$ is a basis of the $L$-vector space $M_{(\upsilon,\xi)}=L[t_1,...,t_\ell]/(t_1^{d_1}-\xi_1\varpi_L^{\upsilon_1},...,t_\ell^{d_\ell}-\xi_{\ell}\varpi_L^{\upsilon_\ell})$. Thus, $\prod_{j=1}^\ell t_j^{\sum_{i=1}^n\alpha_{ij}N_i}\in \cO_L$ if and only if $d_j\mid \sum_{i=1}^n\alpha_{ij}N_i$ for all $1\leq j\leq \ell$.
\end{proof}
\begin{lem}\label{compC}With the notations in Corollary \ref{corab}, we have 
$$\tilde{P}^{|M_{(\upsilon,\xi)}|}=\{\forall 1\leq i\leq n, y_i=0 \textnormal{ iff } b_{i,\upsilon,\xi}\notin\varpi_L\cO_L\}=\{\forall 1\leq i\leq n, y_i=0 \textnormal{ iff } \gamma_{i,\upsilon}=0\}\subset\AA_{k_L}^n=\tilde{V}_{k_L}^{|M_{(\upsilon,\xi)}|}.$$
\end{lem}
\begin{proof}
The preimage of $P$ in $V_{k_L}$ corresponds to  the maximal ideal $(x_1,...,x_n,\varpi_L)$ of $\cO_L[x_1,...,x_n]$. Thus, the preimage $P^{\langle M_{(\upsilon,\xi)}\rangle}$ of $P$ in $\tilde{V}^{\langle M_{(\upsilon,\xi)}\rangle}_{k_L}$ corresponds to the ideal of $\cO_{M_{(\upsilon,\xi)}}[y_1,...,y_n]$  generated by $\varpi_L$ and $a_{i,\upsilon,\xi}y_i$ for $1\leq i\leq n$. The assertion follows from Corollary \ref{corab} and the fact that    $$\tilde{P}^{|M_{(\upsilon,\xi)}|}=\otimes_{\cO_L}\cO_{M_{(\upsilon,\xi)}}\left(P^{\langle M_{(\upsilon,\xi)}\rangle}\right).$$ 
\end{proof} 
\begin{cor}\label{strequality}For each $(\upsilon,\xi)\in \cS$ and $1\leq i\leq n$, we put $A_{i,\upsilon,\xi}=\cO_L$ if $\gamma_{i,\upsilon}>0$ and $A_{i,\upsilon,\xi}=\varpi_L\cO_L$ if $\gamma_{i,\upsilon}=0$. Then,
we have 
\begin{align*}&\#_{\textnormal{st}}^{g,\chi,M_{(\upsilon,\xi)}}(V,E)_{P} 
\\=&\frac{q_L^{n}}{\#G(\cO_L)}\int_{\prod_{i=1}^nA_{i,\upsilon,\xi}}\chi\left(\acc\left(w_{|M_{(\upsilon,\xi)}|}(y)\prod_{i=1}^n a_{i,\upsilon,\xi}^{N_{i}}y_i^{N_{i}}\right)\right)q_L^{\tilde{\mathfrak{N}}-\mathfrak{l}}\prod_{i=1}^nq_L^{\gamma_{i,\upsilon}(\tilde{N}_i-\mathfrak{v}_i)}|y_i|^{-\tilde{N}_i+\mathfrak{v}_{i}-1}|dy|
\end{align*}
for all $(\upsilon,\xi)\in \cS$.
Consequently, we have 
\begin{align*}&\#_{\textnormal{st}}^{g,\chi}(A_G,D)_P\\=&\frac{q_L^{n}}{\#G(\cO_L))}\sum_{(\upsilon,\xi)\in \cS} \int_{\prod_{i=1}^nA_{i,\upsilon,\xi}}\chi\left(\acc\left(w_{|M_{(\upsilon,\xi)}|}(y)\prod_{i=1}^n a_{i,\upsilon,\xi}^{N_{i}}y_i^{N_{i}}\right)\right)q_L^{\tilde{\mathfrak{N}}-\mathfrak{l}}\prod_{i=1}^nq_L^{\gamma_{i,\upsilon}(\tilde{N}_i-\mathfrak{v}_i)}|y_i|^{-\tilde{N}_i+\mathfrak{v}_{i}-1}|dy|
\end{align*}
\end{cor}
\begin{proof}
By Definitions \ref{defnstringy}, \ref{M-st}, Notation \ref{GHop}, Lemmas \ref{comdivi}, \ref{compg}, \ref{compC} and Corollary \ref{corab}, we have 
\begin{align*}
&\#_{\textnormal{st}}^{g,\chi,M_{(\upsilon,\xi)}}(V,E)_{P} 
\\=&\#_{\textnormal{st}}^{g_{|M_{(\upsilon,\xi)}|},\chi}(\tilde{V}^{|M_{(\upsilon,\xi)}|},\tilde{E}^{|M_{(\upsilon,\xi)}|})_{\tilde{P}^{|M_{(\upsilon,\xi)}|}}
\\=&\frac{q_L^{n}}{\#G(\cO_L)}\int_{\prod_{i=1}^nA_{i,\upsilon,\xi}}\chi\left(\acc\left(w_{|M_{(\upsilon,\xi)}|}(y)\varpi_L^{\mathfrak{N}}\prod_{i=1}^n a_{i,\upsilon,\xi}^{N_{i}}y_i^{N_{i}}\right)\right)|\varpi_L|^{-\tilde{\mathfrak{N}}+\mathfrak{l}}\prod_{i=1}^n|b_{i,\upsilon,\xi}|^{\frac{-\tilde{N}_i+\mathfrak{v}_i}{r_{i}}}|y_i|^{-\tilde{N}_i+\mathfrak{v}_{i}-1}|dy|\\
=&\frac{q_L^{n}}{\#G(\cO_L)}\int_{\prod_{i=1}^nA_{i,\upsilon,\xi}}\chi\left(\acc\left(w_{|M_{(\upsilon,\xi)}|}(y)\prod_{i=1}^n a_{i,\upsilon,\xi}^{N_{i}}y_i^{N_{i}}\right)\right)q_L^{\tilde{\mathfrak{N}}-\mathfrak{l}}\prod_{i=1}^nq_L^{\gamma_{i,\upsilon}(\tilde{N}_i-\mathfrak{v}_i)}|y_i|^{-\tilde{N}_i+\mathfrak{v}_{i}-1}|dy| 
\end{align*}
The second assertion follows from the first assertion, Definition \ref{M-st} and Theorem \ref{stringc}. 
\end{proof}
\begin{prop}\label{larcond}Assume that $c(\chi)>1$, $N_i\notin \varpi_L\cO_L$ for at least one index $i$ and $\tilde{N}_j<\mathfrak{v}_j$ for all $1\leq j\leq n$. Then, we have
$$\#_{\textnormal{st}}^{g,\chi}(A_G,D)_P=0$$
\end{prop}
\begin{proof}
It follows from Corollary \ref{strequality} and the proof of \cite[Theorem 2.1]{Denef91}.
\end{proof}
\begin{prop}\label{conducte1}
Suppose that $c(\chi)=1$ and $\tilde{N}_i<\mathfrak{v}_i$ for all $1\leq i\leq n$.
\begin{enumerate}
\item If $d(\chi)\nmid N_{i}$ for some $i$, then we have
$\#_{\textnormal{st}}^{g,\chi}(A_G,D)_P=0$.
\item If $d(\chi)\mid N_{i}$ for all $1\leq i\leq n$, then we have  
$$\#_{\textnormal{st}}^{g,\chi}(A_G,D)_P=\Lambda_{G,D,\chi}\overline{\chi}(\mathfrak{u}(g))\prod_{i=1}^n\frac{(q_L-1)q_L^{\tilde{N}_{i}-\mathfrak{v}_{i}}}{1-q_L^{\tilde{N}_i-\mathfrak{v}_{i}}},$$ where 
$$\Lambda_{G,D,\chi}=\frac{q_L^{\tilde{\mathfrak{N}}-\mathfrak{l}}}{\#(G(\cO_L))}\sum_{(\upsilon,\xi)\in\cS}\chi\left(\acc\left(\prod_{i=1}^n a_{i,\upsilon,\xi}^{N_i}\right)\right)\prod_{i=1}^nq_L^{(\delta_{i,\upsilon}-\gamma_{i,\upsilon})(-\tilde{N}_i+\mathfrak{v}_{i})}$$
with $\delta_{i,\upsilon}=1$ if $\gamma_{i,\upsilon}>0$ and $\delta_{i,\upsilon}=0$ if $\gamma_{i,\upsilon}=0$.
\end{enumerate}  
\end{prop}
\begin{proof} It follows from Lemmas \ref{compg} and \ref{compC} that
\begin{equation}\label{inv}
w_{|M|}(y) \mod \varpi_L\cO_L=\mathfrak{u}(g)
\end{equation} 
for all $(\upsilon,\xi)\in\cS$ and all $y\in \prod_{i=1}^nA_{i,\upsilon,\xi}$.

Since $c(\chi)=1$, we have $\chi(\acc(w_{|M|}(y)))=\overline{\chi}(\mathfrak{u}(g))$
for all  $(\upsilon,\xi)\in\cS$ and all $y\in \prod_{i=1}^nA_{i,\upsilon,\xi}$.  By using Corollary \ref{strequality} and repeating the argument in  \cite[Theorem 2.2]{Denef91}, if $d(\chi)\nmid N_i$ for some $i$ then $$\#_{\textnormal{st}}^{g,\chi}(A_G,D)_P=0.$$
Item (1) is proved. 

Suppose that $d(\chi)\mid N_{i}$ for all $i$.  Then we use \cite[Theorem 2.2]{Denef91} again to have 
\begin{align*}
&\int_{\prod_{i=1}^nA_{i,\upsilon,\xi}}\chi\left(\acc\left(w_{|M_{(\upsilon,\xi)}|}(y)\prod_{i=1}^n a_{i,\upsilon,\xi}^{N_{i}}y_i^{N_{i}}\right)\right)q_L^{\tilde{\mathfrak{N}}-\mathfrak{l}}\prod_{i=1}^nq_L^{\gamma_{i,\upsilon}(\tilde{N}_i-\mathfrak{v}_i)}|y_i|^{-\tilde{N}_i+\mathfrak{v}_{i}-1}|dy| \\
=&q_L^{\tilde{\mathfrak{N}}-\mathfrak{l}}\left(\prod_{i=1}^nq_L^{\gamma_{i,\upsilon}(\tilde{N}_i-\mathfrak{v}_i)}\right)\chi\left(\acc\left(\prod_{i=1}^n a_{i,\upsilon,\xi}^{N_{i}}\right)\right)\overline{\chi}(\mathfrak{u}(g))\prod_{i=1}^n\int_{A_{i,\upsilon,\xi}}\chi\left(\acc(y_i)^{N_{i}}\right)|y_i|^{-\tilde{N}_{i}+\mathfrak{v}_{i}-1}|dy_i|\\
=&q_L^{\tilde{\mathfrak{N}}-\mathfrak{l}}\left(\prod_{i=1}^nq_L^{\gamma_{i,\upsilon}(\tilde{N}_i-\mathfrak{v}_i)}\right)\chi\left(\acc\left(\prod_{i=1}^n a_{i,\upsilon,\xi}^{N_{i}}\right)\right)\overline{\chi}(\mathfrak{u}(g))\prod_{i=1}^n\int_{A_{i,\upsilon,\xi}}|y_i|^{-\tilde{N}_{i}+\mathfrak{v}_{i}-1}|dy_i|\\
=&q_L^{\tilde{\mathfrak{N}}-\mathfrak{l}}\left(\prod_{i=1}^nq_L^{\gamma_{i,\upsilon}(\tilde{N}_i-\mathfrak{v}_i)}\right)\chi\left(\acc\left(\prod_{i=1}^n a_{i,\upsilon,\xi}^{N_{i}}\right)\right)\overline{\chi}(\mathfrak{u}(g))\prod_{i=1}^n\frac{(q_L-1)q_L^{(1-\delta_{i,\upsilon})(\tilde{N}_i-\mathfrak{v}_{i})}}{q_L(1-q_L^{\tilde{N}_i-\mathfrak{v}_{i}})}.
\end{align*}
This and Corollary \ref{strequality} imply Item (2).
\end{proof}
\begin{prop}\label{Gamsom}We set $$\eta_{j}=\frac{\sum_{1\leq i\leq n}\alpha_{ij}N_{i}}{d_j}$$ for $1\leq j\leq \ell$. Then $\eta_j\in \NN$ for all $1\leq j\leq \ell$.  We keep the assumptions of Proposition \ref{conducte1}. Moreover, we suppose that $d(\chi)\mid N_i$ for all $1\leq i\leq n$ and $\gcd(d(G),p_L)=1$. Then we have 
$$\Lambda_{G,D,\chi}=q_L^{\tilde{\mathfrak{N}}-\mathfrak{l}}\sum_{\upsilon\in \prod_{j=1}^\ell\{0,...,d_j-1\}}q_L^{\sum_{1\leq i\leq n}(\delta_{i,\upsilon}-\gamma_{i,\upsilon})(-\tilde{N}_{i}+\mathfrak{v}_{i})}$$
if $d(\chi)\mid \eta_{j}$ for all $j$. If $d(\chi)\nmid \eta_{j}$ for some $j$ then $\Lambda_{G,D,\chi}=0$.
\end{prop}
\begin{proof}
By Lemma \ref{compg}, we have $\eta_j\in\NN$ for all $1\leq j\leq \ell$.
Recall that we are assuming $c(\chi)=1$ and $d(\chi)\mid N_{i}$ for all $1\leq i\leq n$. By Lemma \ref{generator}, we have 
$$\acc\left(\prod_{i=1}^na_{i,\upsilon,\xi}^{N_{i}}\right)=\acc\left(\left(\prod_{i=1}^n\varpi_L^{N_i\beta_{i,\upsilon}}\right)\left(\prod_{j=1}^\ell t_j^{\eta_jd_j}\right)\right)=\acc\left(\prod_{j=1}^\ell t_j^{\eta_jd_j}\right)=\acc\left(\prod_{j=1}^\ell\xi_j^{\eta_{j}}\right).$$ Since $d(\chi)\mid  N_{i} $ for all $i$ and $\chi^{d(\chi)}=1$, we have $d(\chi)\mid \eta_jd_j$ for all $1\leq j\leq \ell$ and $$\chi\left(\acc\left(\prod_{j=1}^\ell\xi_j^{\eta_{j}}\right)\right)=\chi\left(\acc\left(\prod_{j=1}^\ell\left(\xi_j\epsilon_j^{d_j}\right)^{\eta_{j}}\right)\right)$$ if $\epsilon_j\in\cO_L^{\times}$ for all $1\leq j\leq \ell$.  Thus $\chi\left(\acc\left(\prod_{j=1}^\ell\xi_j^{\eta_{j}}\right)\right)$ only depends on the cosets of $\xi_j$ in $\cO_L^{\times}/(\cO_L^{\times})^{d_j}$ for $1\leq j\leq \ell$. By our assumptions, we have  $\gcd(\prod_{j=1}^m d_j,p_L)=\gcd(d(G),p_L)=1$. Thus, the reduction modulo $\varpi_L\cO_L$ yields a group isomorphism  $\cO_L^{\times}/(\cO_L^{\times})^{d_j}\to k_L^{\times}/(k_L^{\times})^{d_j}$. Moreover, we have $\#(\mu_{d_j,\cO_L}(\cO_L))=\#(\mu_{d_j,\cO_L}(k_L))$, 
$\#(\mu_{d_j,\cO_L}(k_L))\#((k_L^{\times})^{d_j})=q_L-1$ and 
\begin{align*}
\chi\left(\acc\left(\prod_{j=1}^\ell\xi_j^{\eta_{j}}\right)\right)&=\overline{\chi}\left(\prod_{j=1}^\ell\overline{\xi}_j^{\eta_{j}}\right)\\&=\frac{1}{\prod_{j=1}^\ell\#((k_L^{\times})^{d_j})}\sum_{(\theta_1,...,\theta_\ell)\in \prod_{j=1}^\ell\overline{\xi}_j(k_L^{\times})^{d_j}}\overline{\chi}\left(\prod_{j=1}^\ell\theta_j^{\eta_{j}}\right)\\
&=\frac{\#(G(\cO_L))}{(q_L-1)^\ell}\sum_{(\theta_1,...,\theta_\ell)\in \prod_{j=1}^\ell\overline{\xi}_j(k_L^{\times})^{d_j}}\overline{\chi}\left(\prod_{j=1}^\ell\theta_j^{\eta_{j}}\right)
\end{align*} 
Therefore, we have 
\begin{align*}
\Lambda_{G,D,\chi}&=\frac{q_L^{\tilde{\mathfrak{N}}-\mathfrak{l}}}{\#(G(\cO_L))}\sum_{(\upsilon,\xi)\in\cS}\chi\left(\acc\left(\prod_{i=1}^n a_{i,\upsilon,\xi}^{N_i}\right)\right)\prod_{i=1}^nq_L^{(\delta_{i,\upsilon}-\gamma_{i,\upsilon})(-\tilde{N}_i+\mathfrak{v}_{i})}\\
&=\frac{q_L^{\tilde{\mathfrak{N}}-\mathfrak{l}}}{\#(G(\cO_L))}\sum_{(\upsilon,\xi)\in \cS}\overline{\chi}\left(\prod_{j=1}^\ell\overline{\xi}_j^{\eta_{j}}\right)\prod_{i=1}^nq_L^{(\delta_{i,\upsilon}-\gamma_{i,\upsilon})(-\tilde{N}_i+\mathfrak{v}_{i})}\\
&=\frac{q_L^{\tilde{\mathfrak{N}}-\mathfrak{l}}}{(q_L-1)^\ell}\sum_{(\upsilon,\xi)\in \cS}\sum_{(\theta_1,...,\theta_\ell)\in \prod_{j=1}^\ell\overline{\xi}_j(k_L^{\times})^{d_j}}\overline{\chi}\left(\prod_{j=1}^\ell\theta_j^{\eta_{j}}\right)q_L^{\sum_{1\leq i\leq n}(\delta_{i,\upsilon}-\gamma_{i,\upsilon })(-\tilde{N}_{i}+\mathfrak{v}_{i})}\\
&=\frac{q_L^{\tilde{\mathfrak{N}}-\mathfrak{l}}\sum_{(\theta_1,...,\theta_\ell)\in \GG_{m,\cO_L}^\ell(k_L)}\overline{\chi}\left(\prod_{j=1}^\ell\theta_j^{\eta_{j}}\right)}{(q_L-1)^\ell}\sum_{\upsilon\in \prod_{j=1}^\ell\{0,...,d_j-1\}}q_L^{\sum_{1\leq i\leq n}(\delta_{i,\upsilon}-\gamma_{i,\upsilon})(-\tilde{N}_{i}+\mathfrak{v}_{i})}.
\end{align*}
If $d(\chi)\mid \eta_{j}$ for all $j$, then we have $\overline{\chi}\left(\prod_{j=1}^\ell\theta_j^{\eta_{j}}\right)=1$ for all $(\theta_1,...,\theta_\ell)\in \GG_{m,\cO_L}^\ell(k_L)$. Thus, we have 
$$\Lambda_{G,D,\chi}=q_L^{\tilde{\mathfrak{N}}-\mathfrak{l}}\sum_{\upsilon\in \prod_{j=1}^\ell\{0,...,d_j-1\}}q_L^{\sum_{1\leq i\leq n}(\delta_{i,\upsilon}-\gamma_{i,\upsilon})(-\tilde{N}_{i}+\mathfrak{v}_{i})}.$$

Suppose that $d(\chi)\nmid \eta_{j}$ for some $j$. Without loss of generality,  we can suppose that $j=\ell$. Then we have 
$$\sum_{(\theta_1,...,\theta_\ell)\in \GG_{m,\cO_L}^\ell(k_L)}\overline{\chi}\left(\prod_{j=1}^\ell\theta_j^{\eta_{j}}\right)
=\sum_{(\theta_1,...,\theta_{\ell-1})\in \GG_{m,\cO_L}^{\ell-1}(k_L)}\overline{\chi}\left(\prod_{j=1}^{\ell-1}\theta_j^{\eta_{j}}\right)\sum_{\theta_\ell\in \GG_{m,\cO_L}(k_L)}\overline{\chi}\left(\theta_\ell^{\eta_{\ell}}\right)$$
Since $c(\chi)=1$, the order of  $\overline{\chi}$ is also $d(\chi)$. It follows from the condition $d(\chi)\nmid \eta_{\ell}$ that
$$\sum_{\theta_\ell\in \GG_{m,\cO_L}(k_L)}\overline{\chi}\left(\theta_\ell^{\eta_{\ell}}\right)=0.$$
Therefore $\Lambda_{G,D,\chi}=0$ in this case.
\end{proof}

\section{A formula for Igusa's twisted local zeta functions in terms of $\QQ$-embedded resolutions of singularities}\label{form}

In this section, we give a formula for Igusa's twisted local zeta functions. 

\subsection{Set up}
We start with setting up some notations and assumptions.

In this section, we use the notions and notations in Section \ref{prep}.  Let $L$ be a local field and $\chi$ be a multiplicative character of $\cO_L^{\times}$. Let $f\in\cO_L[x]:=\cO_L[x_1,...,x_n]$ be a polynomial such that $\overline{f}:=f\mod \varpi_L\cO_L$ is non constant. Recall that we set $D(f)=\spec(L[x]/(f)), X=\AA_L^n, \tilde{D}(f)=\Spec\cO_L[x]/(f), \tilde{X}=\AA_{\cO_L}^n$. Suppose that the pair $(X, D(f))$ admits an embedded $\QQ$-resolution $(Z,\pi)$ and $(Z,\pi)$ has a good $\cO_L$-model. 
We use the notations $(\tilde{Z},\tilde{\pi})$ and $(\overline{Z},\overline{\pi})$ in Definition \ref{q-good}.  We fix almost good data $(\overline{Z}_e,G_e,\nu_{ei},N_{ei}, \mathfrak{n}_e,\mathfrak{N}_e,\mathfrak{u}_e)_{1\leq e\leq l, 1\leq i\leq n}$ of $(Z,\pi)$. If $(Z,\pi)$ has $\QQ$-good reduction modulo $\varpi_L\cO_L$, then we have $\mathfrak{n}_e=\mathfrak{N}_e=0$ for all $e$. Note that $G_e$ is an $\cO_L$-diagonalizable subgroup scheme of $\GG_{m,\cO_L}^n$ and $G_e$ is finite over $\cO_L$. By Lemma \ref{finimugr}, for each $1\leq e\leq l$, we can suppose that $G_e=\prod_{1\leq j\leq \ell_e}\mu_{d_{ej},\cO_L}=\Spec\cO_L[\zeta_1,...,\zeta_{\ell_e}]/(\zeta_1^{d_{e1}}-1,...,\zeta_{\ell_e}^{d_{e\ell_e}}-1)$ for positive integers $d_{ej}$. The action of $G_e$ on $\AA_{\cO_L}^n$ is of the form $$\varphi_e:\cO_L[x_1,...,x_n]\to \cO_L[x_1,...,x_n]\otimes \cO_{G_e}, x_i\mapsto x_i\otimes\prod_{1\leq j\leq \ell_e}\zeta_j^{\alpha_{eij}}$$
for natural numbers $\alpha_{eij}, 1\leq i\leq n, 1\leq j\leq \ell_e$. For each $1\leq e\leq l, 1\leq i\leq n, 1\leq j\leq \ell_e$ and $\upsilon=(\upsilon_1,...,\upsilon_{\ell_e})\in \prod_{1\leq j\leq \ell_e}\{0,1,...,d_{ej}-1\}$, we set 
$$\gamma_{e,i,\upsilon}:=\left(\sum_{1\leq j\leq \ell_e}\frac{\alpha_{eij}\upsilon_j}{d_{ej}}\right)-\bigg\lfloor\sum_{1\leq j\leq \ell_e}\frac{\alpha_{eij}\upsilon_j}{d_{ej}}\bigg\rfloor,$$
$$\delta_{e,i,\upsilon }:=\begin{cases}1 \textnormal{ if } \gamma_{e,i,\upsilon }>0,\\ 0 \textnormal{ if } \gamma_{e,i,\upsilon}=0,\end{cases}$$
and 
$$\eta_{ej}=\frac{\sum_{1\leq i\leq n}\alpha_{eij}N_{ei}}{d_j}$$ for $1\leq j\leq \ell_e$.

By Proposition \ref{stragood}, for each \(1\le e \le l\), there exists a  locally standard triple \((A_{G_e}^{\rm loc}, \cD_e ,\cD_e')\) over $\cO_L$ and natural numbers $\mathfrak{l}_e,\mathfrak{l}'_e$ such that for each point \(z\in \overline{Z}_e(k_{L})\), there is an $\cO_{L}$-étale morphism $\phi_{z}$ from a standard scheme over $\cO_L$ witnessed by $G_e$ to $\tilde{Z}$ that induces an $\cO_L$-étale morphism of triples
\[\phi_{z}^{\loc}:(A_{G_e}^{\rm loc}, \cD_e ,\cD_e')\to 
 (\Spec \cO_{\tilde{Z},z}, \tilde{\pi}^{*}\tilde{D}(f)-\mathfrak{l}_e\Spec\cO_{\tilde{Z}_{k_{L}},z}, K_{\tilde{Z}/\tilde{X}}-\mathfrak{l}'_e\Spec\cO_{\tilde{Z}_{k_{L}},z})
.\]
 
For each $1\leq e\leq l$ and each $z\in \overline{Z}_e(k_L)$, we set $g_z:=f\circ\tilde{\pi}\circ\phi_{z}$. By  Proposition \ref{stragood} together with Definitions \ref{dedivis} and \ref{destrata}, if $1\leq e\leq l$ and  $z\in \overline{Z}_e(k_L)$ then \begin{equation}\label{eqdivi}\Div(g_z\circ\rho_{G_e})|_{V_z}=\mathfrak{N}_e\AA_{k_L}^n|_{V_z}+\sum_{i=1}^nN_{ei}E_i|_{V_z}
\end{equation} for a suitable open neighbourhood $V_z$ of $P_e:=\rho_{G_e}((\varpi_L)+(x_1,...,x_n))$ in $A_{G_e}$ . By Lemma-Definition \ref{factorunit} and Definition \ref{destrata}, the standard angular component $\mathfrak{u}(g_z)$ of $g_z$ at $P_e$ is equal to $\mathfrak{u}_e(z)$ for all $1\leq e\leq l$ and all $z\in \overline{Z}_e(k_L)$. Moreover,  by Proposition \ref{stragood}, we also have \begin{equation}\label{eqredivi}\rho_{G_e}^*(\cD_e'+\mathfrak{l}'_e\AA_{G_e\otimes_{\cO_L} k_L}^{\loc})+K_{\AA_{\cO_L}^{n,\loc}/A_{G_e}^{\loc}}=\mathfrak{n}_e\AA_{k_L}^{n,\loc}+\sum_{i=1}^n(\nu_{ei}-1)E_i
\end{equation}for all $1\leq e\leq l$ and all $z\in\tilde{Z}(k_L).$

  Let $s$ be a positive rational number. We consider the divisor $\tilde{D}_{f,s}=-s\tilde{D}(f)$ of $\tilde{X}$. We set $\mathfrak{D}_s:=\tilde{\pi}^*(\tilde{D}_{f,s})-K_{\tilde{Z}/\tilde{X}}$.

Let $C$ be a subscheme of $\AA_{k_L}^n$. We consider the Schwartz-Bruhat function $\Phi_{L,Z}$ as mentioned in Subsection \ref{bano}.  

\subsection{The explicit formula}
Now, we establish our formula for Igusa's twisted local zeta functions in terms of the above setting. 
\begin{lem}\label{crepmo}
The morphism $\tilde{\pi}:(\tilde{Z},\mathfrak{D}_s)\to (\tilde{X},\tilde{D}_{f,s})$ is a crepant proper birational morphism of $\cO_L$-log pairs. Moreover, for each $1\leq e\leq l$ and each $z\in \overline{Z}_e(k_L)$, there is an open neighbourhood $W_z$ of $P_e$ such that 
$$\phi_z:(W_z, \phi_z^{*}(\mathfrak{D}_s))\to (\tilde{Z},\mathfrak{D}_s)$$
is a crepant $\cO_L$-étale morphism of log pairs. Moreover, the support of $\rho_{G_e}^*(\phi_z^{*}(\mathfrak{D}_s))$ is contained in $\AA_{k_L}^n|_{U_z}\cup\left(\cup_{i=1}^nE_i|_{U_z}\right)$ and  we have 
$$\rho_{G_e}^*\left(\phi_z^{*}(\mathfrak{D}_s)\right)-K_{\AA_{\cO_L}^n/A_{G_e}}|_{U_z}=-(\mathfrak{N}_es+\mathfrak{n}_e)\AA_{k_L}^n|_{U_z}-\sum_{i=1}^n(N_{ei}s+\nu_{ei}-1)E_i|_{U_z}$$
for all $1\leq e\leq l$ and all $z\in \overline{Z}_e(k_L)$, where $U_z=\rho_{G_e}^{-1}(W_z)$.
\end{lem}
\begin{proof}
The first and second assertions follow from \cite[Lemma 4.6]{Yas} together with the facts that $\pi=\tilde{\pi}\otimes_{\cO_L}L$ is a birational morphism and $\phi_z$ is an $\cO_L$-étale morphism. 

Let $1\leq e\leq l$ and $z\in \overline{Z}_e$. Since $\phi_z^{\loc}$ is an $\cO_L$-étale morphism of triples, we have $K_{W_z/\tilde{Z}}=0$ and
\begin{equation}\label{pulldivi}\left(\phi_z^{\loc}\right)^*(K_{\tilde{Z}/\tilde{X}})-\mathfrak{l}'_e\AA_{k_L}^{n,\loc}=\left(\phi_z^{\loc}\right)^*(K_{\tilde{Z}/\tilde{X}}-\mathfrak{l}'_e\Spec\cO_{\tilde{Z}_{k_L},z})=\cD'_e.
\end{equation}
By (\ref{eqdivi}), (\ref{eqredivi}), (\ref{pulldivi}) and shrinking $W_z$ if needed, we have 
\begin{align*}
&\rho_{G_e}^*\left(\phi_z^{*}(\mathfrak{D}_s)\right)-K_{\AA_{\cO_L}^n/A_{G_e}}|_{U_z}\\=&\rho_{G_e}^*\phi_z^{*}\tilde{\pi}^*(\tilde{D}_{f,s})-\rho_{G_e}^*\left(\phi_z^{*}\left(K_{\tilde{Z}/\tilde{X}}\right)\right)-K_{\AA_{\cO_L}^n/A_{G_e}}|_{U_z}\\
=&-\left(\mathfrak{N}_es\AA_{k_L}^n|_{U_z}+\sum_{i=1}^nN_{ei}sE_i|_{U_z}\right)-\rho_{G_e}^*\left(\phi_z^{*}(\cD'_e+\mathfrak{l}'_e\AA_{k_L}^{n}) \right)-K_{\AA_{\cO_L}^n/A_{G_e}}|_{U_z}
\\
=&-(\mathfrak{N}_es+\mathfrak{n}_e)\AA_{k_L}^n|_{U_z}-\sum_{i=1}^n(N_{ei}s+\nu_{ei}-1)E_i|_{U_z}
\end{align*}
and 
$$\rho_{G_e}^*\left(\phi_z^{*}(\mathfrak{D}_s)\right)=-(\mathfrak{N}_es+\mathfrak{n}_e)\AA_{k_L}^n|_{U_z}-\sum_{i=1}^n(N_{ei}s+\nu_{ei}-1)E_i|_{U_z}+ K_{\AA_{\cO_L}^n/A_{G_e}}|_{U_z}.$$
Moreover, since $K_{\AA_{\cO_L}^n/A_{G_e}}$ is an effective divisor whose support is contained in $\AA_{k_L}^n\cup\left(\cup_{i=1}^nE_i\right)$, the last assertion is proved.
\end{proof}
\begin{cor}\label{namdivi}Let $1\leq e\leq l$ and $z\in \overline{Z}_e(k_L)$. We write $\phi_z^{*}(\mathfrak{D}_s)=\mathfrak{m}_{es}W_z\otimes_{\cO_L}k_L+\sum_{i=1}^nm_{ies}D_i|_{W_z}$ for rational numbers $\mathfrak{m}_{es},m_{ies}$. Then $\mathcal{E}_{es}:=\mathfrak{m}_{es}A_{G_e}\otimes_{\cO_L}k_L+\sum_{i=1}^nm_{ies}D_i$ depends only on $e$ and $s$, where $D_i$ is defined in  Definition \ref{dedivis}. Moreover, we have 
$$\rho_{G_e}^*(\cE_{es})-K_{\AA_{\cO_L}^n/A_{G_e}}=-(\mathfrak{N}_es+\mathfrak{n}_e)\AA_{k_L}^n-\sum_{i=1}^n(N_{ei}s+\nu_{ei}-1)E_i.$$
\end{cor}
\begin{proof}
 It follows immediately from the last assertion of Lemma \ref{crepmo}.   
\end{proof}
\begin{lem}\label{sumstring}With the divisor $\cE_{es}$ in Corollary \ref{namdivi}, we have
\begin{align*}
q_L^n\cZ_{f,\Phi_{L,C},\chi}(s)=\sum_{e=1}^l\sum_{z\in \overline{\pi}^{-1}(C)(k_L)\cap \overline{Z}_e(k_L)}\#_{\textnormal{st}}^{g_z,\chi}(A_{
G_e},\cE_{es})_{P_e}
\end{align*}
for all positive rational numbers $s$. 
\end{lem}
\begin{proof}
We have 
\begin{align*}
q_L^n\cZ_{f,\Phi_{L,C},\chi}(s)&=\#_{\textnormal{st}}^{f,\chi}(\tilde{X},\tilde{D}_{f,s})_{C} & (\text{Example \ref{exmp}} )\\
&=\sum_{x\in C(k_L)}\#_{\textnormal{st}}^{f,\chi}(\tilde{X},\tilde{D}_{f,s})_{\overline{x}}& (\text{Equality \eqref{pointwise}} )\\
&=\sum_{z\in \overline{\pi}^{-1}(C)(k_L)}\#_{\textnormal{st}}^{f\circ\tilde{\pi},\chi}(\tilde{Z},\mathfrak{D}_{s})_{\overline{z}}& (\text{Proposition \ref{relation1} and Lemma \ref{crepmo}})\\
&=\sum_{e=1}^l\sum_{z\in \overline{\pi}^{-1}(C)(k_L)\cap \overline{Z}_e(k_L)}\#_{\textnormal{st}}^{g_z,\chi}(W_z,(\cE_{es})|_{W_z})_{P_e}& (\text{Proposition \ref{relation2} and Lemma \ref{crepmo}})\\
&=\sum_{e=1}^l\sum_{z\in \overline{\pi}^{-1}(C)(k_L)\cap \overline{Z}_e(k_L)}\#_{\textnormal{st}}^{g_z,\chi}(A_{G_e},\cE_{es})_{P_e}& (\text{Definition \ref{defnstringy}}).
\end{align*}
\end{proof}
\begin{prop}\label{ngood}Suppose that there is an embedded $\QQ$-resolution $(Z,\pi)$ of $(\AA_{L}^n,D(f))$ having a good $\cO_L$-model and $(\overline{Z}_e,G_e, \nu_{ei},N_{ei},\mathfrak{n}_e,\mathfrak{N}_e,\mathfrak{u}_e)_{1\leq e\leq l, 1\leq i\leq n}$ are  $\cO_L$-almost good data of $(Z,\pi)$ as above. Then the real part of every pole of $\cZ_{f,\Phi_{L,C},\chi}(s)$ belongs to the set $\{-\frac{\nu_{ei}}{N_{ei}}\mid 1\leq e\leq r, 1\leq i\leq n,  N_{ei}>0, d(\chi)\mid N_{ei}\}$. Moreover, the multiplicity of the pole $\alpha$ of $\cZ_{f,\Phi_{L,C},\chi}(s)$ is at most $\max_{1\leq e\leq r}\#\{1\leq i\leq n\mid N_{ei}>0, \nu_{ei}=-\Re(\alpha)N_{ei}\}$.
\end{prop}
\begin{proof}We can use  Corollaries \ref{strequality}, \ref{namdivi}, Proposition \ref{conducte1}, Lemma \ref{sumstring} and the argument of \cite[Theorem 1.3.2]{DenefBour} to compute $\cZ_{f,\Phi_{L,C},\chi}(s)$ if $s$ is a positive rational number. In particular, there is a rational function $Q(s)$ in $q_L^{-s}$ whose denominator is a factor of $\gcd\{\prod_{1\leq i\leq n}(1-q_L^{-N_{ei}s-\nu_{ei}})\mid 1\leq e\leq l, d(\chi)\mid N_{ei} \forall \hspace{0.1cm} 1\leq i\leq n\}$ such that $Q(s)=\cZ_{f,\Phi_{L,C},\chi}(s)$ for all $s\in\QQ_{>0}$. As mentioned in Subsection \ref{Igex}, $\cZ_{f,\Phi_{L,C},\chi}(s)$ is a rational function in $q_L^{-s}$ when $\Re(s)>0$. Therefore, $Q(s)$ is the meromorphic continuation to $\CC$ of  $\cZ_{f,\Phi_{L,C},\chi}(s)$. 
\end{proof}
\begin{prop}\label{chilarge1}
Suppose that $c(\chi)>1$ and that $(Z,\pi)$ has $\QQ$-tame good reduction modulo $\varpi_L\cO_L$. Then, $\cZ_{f,\Phi_{L,C},\chi}(s)$ is a constant function of $s$. In particular, 
$\cZ_{f,\Phi_{L,C},\chi}(s)=0$
if the restriction of $\overline{f}$ on a Zariski neighbourhood of $C$ in  $\AA_{k_L}^n$ is smooth outside $\overline{f}^{-1}(0)$.
\end{prop}
\begin{proof}
We set $C_1=C\cap \overline{f}^{-1}(0)$ and $C_2=C\setminus C_1$. Then, we have $\Phi_{L,C}=\Phi_{L,C_1}+\Phi_{L,C_2}$ and $\cZ_{f,\Phi_{L,C},\chi}(s)=\cZ_{f,\Phi_{L,C_1},\chi}(s)+\cZ_{f,\Phi_{L,C_2},\chi}(s)$. Since $\cZ_{f,\Phi_{L,C},\chi}(s)$ is a rational function in $q_L^{-s}$ as mentioned in Subsection \ref{Igex}, we just need to prove the assertions for $s\in \QQ_{>0}$.
If $\overline{\pi}^{-1}(C_2)\cap \overline{Z}_{e}\neq\emptyset$ then Proposition \ref{stragood}  implies that   $G_e=\{1\},\nu_{ei}-1=N_{ei}=\mathfrak{n}_e=\mathfrak{N}_e=0$ for all $1\leq i\leq n$. Thus, we can use Corollary \ref{strequality} and Lemma \ref{sumstring} together with the proof of \cite[Theorem 2.1]{Denef91} to show that $\cZ_{f,\Phi_{L,C_2},\chi}(s)$ is a constant function on $\QQ_{>0}$. Moreover, $\cZ_{f,\Phi_{L,C_2},\chi}(s)=0$ for all $s\in\QQ_{>0}$ if the restriction of $\overline{f}$ on a Zariski neighbourhood of $C_2$ in  $\AA_{k_L}^n\setminus\overline{f}^{-1}(0)$ is smooth. Thus, the assertions are proved if we can show that $\cZ_{f,\Phi_{L,C_1},\chi}(s)=0$ for  all $s\in\QQ_{>0}$. Indeed, if $\overline{Z}_e\cap \overline{\pi}^{-1}(C_1)\neq \emptyset$ then the $\QQ$-tame good reduction of $(Z,\pi)$ implies that $N_{ei}\notin\varpi_L\cO_L$ for at least one index $i$. Hence, we can use Proposition \ref{larcond} and Lemma \ref{sumstring} to deduce $\cZ_{f,\Phi_{L,C_1},\chi}(s)=0$ for all $s\in\QQ_{>0}$. This completes our proof.  
\end{proof}
\begin{maintheorem}\label{mucase}Suppose that $(Z,\pi)$ is an embedded $\QQ$-resolution of $(\AA_{L}^n,D(f))$ with $\QQ$-good reduction modulo $\varpi_L\cO_L$. Let $(\overline{Z}_e,G_e,\nu_{ei},N_{ei},\mathfrak{u}_e)_{1\leq e\leq r, 1\leq i\leq n}$ be good data of $(Z,\pi)$,  $\chi$ be a multiplicative character of $\cO_L^{\times}$ and $C$ be a subscheme of $\AA_{k_L}^n$ as above. Moreover, we assume that $c(\chi)=1$. With the above setting, we have 
$$\cZ_{f,\Phi_{L,C},\chi}(s)=\sum_{1\leq e\leq r, d(\chi)\mid N_{ei} , d(\chi)\mid\eta_{ej} \textnormal{ for all } i,j}\Omega_{C,\chi,\pi,\overline{Z}_e}\Theta_{\overline{Z}_e,\chi}\prod_{i=1}^n\frac{(q_L-1)q_L^{-N_{ei}s-\nu_{ei}}}{q_L(1-q_L^{-N_{ei}s-\nu_{ei}})},$$
where 
$$\Theta_{\overline{Z}_e,\chi}=\sum_{\upsilon\in \prod_{j=1}^{\ell_e}\{0,...,d_{ej}-1\}}q_L^{\sum_{1\leq i\leq n}(\delta_{e,i,\upsilon}-\gamma_{e,i,\upsilon})(N_{ei}s+\nu_{ei})}$$
and 
$$\Omega_{C,\chi,\pi,\overline{Z}_e}=\sum_{z\in(\overline{\pi}^{-1}(C)\cap\overline{Z}_e)(k_L)}\overline{\chi}(\mathfrak{u}_e(z)).$$
\end{maintheorem}
\begin{proof}
Note that $\cZ_{f,\Phi_{L,C},\chi}(s)$ is a rational function in $q_L^{-s}$ that has a meromorphic continuity to the whole $\CC$ as mentioned in Subsection \ref{Igex}. Thus, it turns out that  we only need to prove the equality for all $s\in\QQ_{>0}$. In fact, the equality follows immediately from applying Propositions \ref{conducte1} and \ref{Gamsom} to our setting. 
\end{proof}
\begin{remark}\label{trivindep}
If $G_e=\{1\}$ for all $e$, it is easy to verify that Main Theorem \ref{mucase} agrees with \cite[Theorem 2.2]{Denef91}. Moreover, the formula in Main Theorem \ref{mucase} for $\chi=\chi_{\textnormal{triv}}$ is of a similar form to the one in \cite{LMVV}. 

On the other hand,  in Main Theorem \ref{mucase}, we observe that  $\cZ_{f,\Phi_{L,C},\chi}(s)$ is independent of the choice of the embedded $\QQ$-resolution  $(Z,\pi)$ and its good data. Thus, after fixing the data $(\overline{Z}_e,G_e,\nu_{ei},N_{ei})$,   for each  multiplicative character $\chi$ with $c(\chi)=1$ and each stratum $\overline{Z}_e$ such that $d(\chi)\mid \gcd(N_{ei},\eta_{ej})$ for all $i$ and $j$, the value $\overline{\chi}(\mathfrak{u}_e(z))$ must be independent of the choice of $\mathfrak{u}_e$ (see Remark \ref{choiceangu}).
\end{remark}

\section{The case of weighted homogeneous polynomials with isolated singularities}\label{wiso}
Let $w=(w_1,...,w_n)\in \NN_{>0}^n$. Suppose that $\gcd(w_1,...,w_n)=1$. We set $|w|=w_1+...+w_n$. A polynomial $g$ in $n$ variables over a ring is $w$-weighted homogeneous if there is a natural number $d$ such that $$g=\sum_{I=(i_1,...,i_n)\in \NN^n, \sum_{1\leq j\leq n}i_jw_j=d}a_Ix_1^{i_1}...x_n^{i_n}.$$ In this case, we say that $d$ is the $w$-degree of $g$. For an arbitrary polynomial $g$, we can write $g=g_{0w}+...+g_{dw}$, where $g_{iw}$ is the $w$-weighted homogeneous part of $w$-degree $i$ of $g$ and $g_{dw}\neq 0$. We call $d$ the $w$-degree 
of $g$. 

The following lemma  provides the technique condition for the calculation in this section. 
\begin{lem}\label{almostpp}
Let $K$ be a number field and $g\in\cO_K[x_1,...,x_n]$ be a non-constant polynomial.  Then, there is a natural number $M$ such that any singular point with coordinates in $k_L$ of  the hypersurface $\{g\mod \varpi_L\cO_L=0\}\subset \AA_{k_L}^n$  can be lifted to a singular point with coordinates in $\cO_L$   of the hypersurface $\{g_L=0\}$  for all local fields $L$ over $\cO_K$ of residue field characteristic at least $M$.
\end{lem}
\begin{proof}
There are several ways to prove the assertion. 

Note that if $L$ is a local field and $h\in \cO_L[x_1,...,x_n]$ is a non-constant polynomial such that $(\AA_L^n, \Spec L[x_1,...,x_n]/(h))$ admits an embedded resolution with good reduction modulo $\varpi_L\cO_L$ mentioned in Remark \ref{comparQgoo}, then any singular point with coordinates in $k_L$ of  the hypersurface $\{h(x)\mod \varpi_L\cO_L=0\}$ can be lifted to a singular point with coordinates in  $\cO_L$  of the hypersurface  $\{h(x)=0\}$. Thus, we can use \cite[Theorem 2.4]{Denefdegree} for $g$ to deduce the assertion. 

Alternatively, we observe that if $\mathsf{F}$ is  a Henselian value field of equi-characteristic zero endowed with a structure of $\cO_K$-algebra and $h$ is a non-constant polynomial in $n$ variables whose coefficients belong to $\cO_{K}$ then any $k_{\mathsf{F}}$-solution of the system $h(x)=\frac{\partial h}{\partial x_1}=...=\frac{\partial h}{\partial x_n}=0$ can be lifted to an $\cO_{\mathsf{F}}$-solution  of the system $h(x)=\frac{\partial h}{\partial x_1}=...=\frac{\partial h}{\partial x_n}=0$, where $k_{\mathsf{F}}$ is the residue field of $\mathsf{F}$. Therefore, we can use model theoretic compactness (see, e.g.,  \cite[Theorem 2.4]{Marker}) for the theory of Henselian valued fields of equi-characteristic zero endowed with a structure of $\cO_K$-algebra to prove the assertion. 
\end{proof}
Now, we fix a number field $K$, a tuple $w=(w_1,...,w_n)\in \NN_{>0}^n$ satisfying $\gcd(w_1,...,w_n)=1$ and a non constant polynomial $f\in \cO_K[x_1,...,x_n]$. To simpler, we assume that $f$ is primitive, i.e., the ideal of $\cO_K$ generated by coefficients of $f$ is $(1)$.  We write $f=f_{d_0w}+...+f_{dw}$ as above with $f_{d_0w}\neq 0$. We suppose that $d_0>1$ and the origin $O=(0,...,0)$ is the only common zero point of $\frac{\partial f_{d_0w}}{\partial x_1},...,\frac{\partial f_{d_0w}}{\partial x_n}$ in $\AA_\CC^n$. Then $O$ is the isolated singularity of $D_K(f_{d_0w})$. 

\subsection{$\QQ$-embedded resolution of singularities of weighted homogeneous polynomials with isolated singularities}

Let $\pi:Z\to\AA_{K}^n$ be the $w$-weighted blowing-up at the origin $O$ then  $\pi$ is an embedded $\QQ$-resolution of $(\AA_{K}^n,D_K(f))$ in a small enough Zariski neighbourhood of  $O$ in $\AA_K^n$ as shown in \cite[Section 2.2]{BBV}. 

\subsubsection{Description of $(Z,\pi)$}
It is well-known  that the $w$-weighted projective space of $\PP_{K}^{w}$, is the geometric quotient of $\AA_K^n\setminus \{O\}$ by the action of $\GG_{\textnormal{m},K}=\spec K[z,z^{-1}]$ given by $t(x_1,...,x_n)=(t^{w_1}x_1,...,t^{w_n}x_n)$, where $t\in \GG_{\textnormal{m},K}(R), (x_1,...,x_n)\in R^n$ and $R$ is any $K$-algebra. Let $\overline{K}$ be an algebraic closure of $K$ and $x=(x_1,...,x_n)\in \overline{K}^n\setminus \{O\}$, we write $[x]_w$ or $[x_1:...:x_n]_w$ for the orbit of $x$ under the above action of $\GG_{\textnormal{m},K}(\overline{K})=\overline{K}^{\times}$. A $K$-point of $\PP_{K}^{w}$ is an orbit $[x_1:...:x_n]_w$ containing a point in $\AA_K^n(K)\setminus \{O\}$. We can cover $\PP_{K}^{w}$ by open subset $U_i$ for $1\leq i\leq n$ given by $U_i:=\{[x_1:...:x_n]_w\mid x_i\neq 0\}$. On the other hand, $Z$ is the  Zariski closure in $\AA_{K}^n\times \PP_{K}^{w}$ of the set consisting of points $(a_1,...,a_n,[a_1:...:a_n]_w)$. Moreover, it is well-known that the projection $\pi:Z\to \AA_{K}^n$ is an isomorphism over $\AA_K^n\setminus \{O\}$ and $\pi^{-1}(O)=\PP_{K}^{w}$. We can cover $Z$ by open subsets $U'_i$ for $1\leq i\leq n$ given by 
$U'_i:=Z\cap (\AA_K^n\times U_i).$
Moreover, for each $a\in K^{\times}$ and each $1\leq i\leq n$, we can identify $U'_i$ with the quotient $\AA_K^n/\mu_{w_i,K}$ by the morphism $\sigma_{ia}:\AA_{K}^{n}\to U'_i$ sending $(x_1,...,x_{i-1},x_i,x_{i+1},...,x_n)$ to $$(x_i^{w_1}x_1,...,x_i^{w_{i-1}}x_{i-1},ax_i^{w_i},x_{i}^{w_{i+1}}x_{i+1},...,x_i^{w_n}x_n,[x_1:...:x_{i-1}:a:x_{i+1}:...:x_n]_w),$$
where the action $\phi_i$ of $\mu_{w_i,K}$ on $\AA_K^n$ is induced by the embedding $\mu_{w_i,K}\to \GG_{m,K}^n$ associated with the homomorphism of Hofp algebras $K[z_1,z_1^{-1},...,z_n,z_n^{-1}]\to K[\zeta]/(\zeta^{w_i}-1), z_j\mapsto \zeta^{w_j} \textnormal{ if } i\neq j, z_i\mapsto \zeta^{w_i-1}.$ In fact, we have 
\begin{align*}
\varphi_{i}: K[x_1,...,x_{i-1},x_i,x_{i+1},...,x_n]&\to K[x_1,...,x_{i-1},x_i,x_{i+1},...,x_n]\otimes_K K[\zeta]/(\zeta^{w_i}-1),\\
 x_j&\mapsto x_j\otimes \zeta^{w_j} \textnormal{ if } i\neq j,\\
 x_i&\mapsto x_i\otimes \zeta^{w_i-1}.
\end{align*}
The morphism $\sigma_{ia}$ has a property that if $a_1,...,a_n\in K$ with $a=a_i\neq 0$ then $\sigma_{ia}^{-1}(0,...,0,[a_1:...:a_n]_w)$ contains the $K$-point $(a_1,...,a_{i-1},0,a_{i+1},...,a_n)$ of $\AA_K^n$.  Let us consider the chart $U'_i$. Suppose that $i=1$ and $a\neq 0$ are fixed. We have the equality \begin{equation}\label{pulbf}f(\pi|_{U'_1}(\sigma_{1a}(x_1,x_2,...,x_n)))=x_1^{d_0}(f_{d_0w}(a,x_2,...,x_n)+x_1h(x_1,a,x_2,...,x_n))
\end{equation}
of the functions in $x_1,x_2,...,x_n$ for some polynomial $h$ and
\begin{equation}\label{pulcanon}(\pi|_{U'_1}\circ\sigma_{1a})^*(dx_1\wedge...\wedge dx_n)=w_1x_1^{|w|-1}dx_1\wedge dx_2\wedge...\wedge dx_n.
\end{equation}
We also observe that for any fixed $a\neq 0$,  $f_{d_0w}(a,x_2,...,x_n)$ is a smooth function  as shown in \cite[Section 2.2]{BBV}. 
\subsubsection{Almost good data of $(Z,\pi)$}
Let $L$ be a local field over $K$. 
Then $(Z,\pi)$ induces an embedded $\QQ$-resolution $(Z_L,\pi_L)$ of $(\AA_L^n,D_L(f))$ in a small neighbourhood of the origin $O\in \AA_L^n$.  By using Lemmas \ref{etalelocalquoti} and \ref{almostgood}, we will describe almost good data of $(Z_L,\pi_L)$. Let  $\tilde{Z}_{\cO_L}$ be the Zariski closure of $Z_L$ in $\AA_{\cO_L}^n\times \PP_{\cO_L}^w$ and $\tilde{\pi}_{\cO_L}$ be the projection from $\tilde{Z}_{\cO_L}$ to $\AA_{\cO_L}^n$. For each $\cO_L$-algebra $R$,  the set of $R$-points of $\tilde{Z}_{\cO_L}$ consists of elements  of the forms $(x,[x]_w)$ or $(0,...,0,[x]_w)$ for $x\in R^n\setminus\{(0,...,0)\}$. We also set $\overline{\pi}_{k_L}:=\tilde{\pi}_{\cO_L}\otimes_{\cO_L}k_L$. Then $\overline{\pi}_{k_L}: \overline{Z}_{k_L}=\tilde{Z}_{\cO_L}\otimes_{\cO_L}k_L\to \AA_{k_L}^n$ is the restriction to $\overline{Z}_{k_L}$ of the projection $\AA_{k_L}^n\times \PP_{k_L}^w\to\AA_{k_L}^n$. 

We can cover $\PP_{\cO_L}^{w}$ by open subset $\tilde{U}_i=\PP_{\cO_L}^{w}\setminus D_i$ for $1\leq i\leq n$, where $D_i$ is the divisor of $\PP_{\cO_L}^{w}=\Proj\cO_L[x_1,...,x_n;w]_{\ZZ}$ associated with $x_i$. We cover $\tilde{Z}_{\cO_L}$ by open subsets $\tilde{U}'_i$ for $1\leq i\leq n$ given by 
$\tilde{U}'_i:=\tilde{Z}_{\cO_L}\cap (\AA_{\cO_L}^n\times \tilde{U}_i).$
Moreover, for each $a\in \cO_L^{\times}$ and each $1\leq i\leq n$, we can identify $\tilde{U}'_i$ with the quotient $\AA_{\cO_L}^n/\mu_{w_i,\cO_L}$ by the morphism $\tilde{\sigma}_{ia}:\AA_{\cO_L}^{n}\to \tilde{U}'_i$ sending an $R$-point $(x_1,...,x_{i-1},x_i,x_{i+1},...,x_n)$ to $$(x_i^{w_1}x_1,...,x_i^{w_{i-1}}x_{i-1},ax_i^{w_i},x_{i}^{w_{i+1}}x_{i+1},...,x_i^{w_n}x_n,[x_1:...:x_{i-1}:a:x_{i+1}:...:x_n]_w),$$
for each $\cO_L$-algebra $R$, where  the action of $\mu_{w_i,\cO_L}$ on $\AA_{\cO_L}^n$ is given by 
\begin{align*}
\tilde{\varphi}_{i\cO_L}: \cO_L[x_1,...,x_{i-1},x_i,x_{i+1},...,x_n]&\to \cO_L[x_1,...,x_{i-1},x_i,x_{i+1},...,x_n]\otimes_{\cO_L} \cO_L[\zeta]/(\zeta^{w_i}-1),\\
 x_j&\mapsto x_j\otimes \zeta^{w_j} \textnormal{ if } i\neq j,\\
 x_i&\mapsto x_i\otimes \zeta^{w_i-1}.
\end{align*}

Now, we stratify $\overline{Z}_{k_L}$ as follows. Recall that $\PP_{k_L}^w=\Proj k_L[x_1,...,x_n;w]_{\ZZ}$. We set $\overline{Z}^{\circ}=\overline{Z}_{k_L}\setminus \{(0,...,0)\}\times \PP_{k_L}^w\subset \overline{Z}_{k_L}$. For each $1\leq i\leq n$, we set $\overline{D}_i:=D_i\otimes_{\cO_L}k_L$. We also denote by $\overline{H}$ the divisor of $\PP_{k_L}^w$ associated with $\overline{f}_{d_0w}=f_{d_0w}\mod \varpi_L\cO_L$.
For each non-empty subset $J$ of $\{1,2,...,n\}$, we set $d_J=\gcd(w_i)_{i\in J}$, $G_{J\cO_L}=\mu_{d_J,\cO_L}$, $$\overline{Z}_{J0L}=\{(0,...,0)\}\times (\overline{H}\cap(\cap_{i\notin J} \overline{D}_i)\setminus \cup_{i\in J}\overline{D}_i)\subset \overline{Z}_{k_L}$$ and $$\overline{Z}_{J1L}=\{(0,...,0)\}\times(\cap_{i\notin J} \overline{D}_i\setminus (\overline{H}\cup(\cup_{i\in J}\overline{D}_i))\subset \overline{Z}_{k_L}.$$
For each non-empty subset $J$ of $\{1,...,n\}$, we fix an index $i_J\in J$ and  consider the following action of $G_{J\cO_L}$ on $\AA_{\cO_K}^n$:
\begin{align*}
\varphi_{J\cO_L}: \cO_L[x_1,...,x_n]&\to \cO_L[x_1,...,x_n]\otimes_{\cO_L} \cO_L[\zeta]/(\zeta^{d_J}-1)\\
x_{i_J}&\mapsto x_{i_J}\otimes \zeta^{d_J-1}\\
x_i&\mapsto x_i\otimes 1 \textnormal{ if } i\in J\setminus\{i_J\}\\
x_i&\mapsto x_i\otimes \zeta^{w_i} \textnormal{ if } i\notin J,
\end{align*} 
\begin{lem} \label{desstraw}Suppose that $\overline{f}_{d_0w}$ defines a hypersurface with an isolated singularity in $\AA_{\overline{k}_L}^n$. Then, there is a Zariski open neighbourhood $W$ of $P=(\varpi_L)+(x_1,...,x_n)$ in $\AA_{\cO_L}^n$ such that  $(\pi_L^{-1}(W_L),\pi_L)$ is an embedded $\QQ$-resolution of $(W_L, D_L(f)|_{W_L})$ with a good $\cO_L$-model induced by the morphism $\tilde{\pi}_{\cO_L}:\tilde{Z}_{\cO_L}\to \AA_{\cO_L}^n$ and the morphisms $\tilde{\sigma}_{ia}$ mentioned above. Moreover, we have 
$\overline{Z}_{k_L}\cap \overline{\pi}_{k_L}^{-1}(W_{k_L})=(\overline{Z}^{\circ}\cap \overline{\pi}_{k_L}^{-1}(W_{k_L}))\sqcup(\sqcup_{\emptyset\neq J\subset \{1,...,n\}}(\overline{Z}_{J0L}\sqcup\overline{Z}_{J1L}))$ and this is a stratification of $\overline{Z}_{k_L}\cap \overline{\pi}_{k_L}^{-1}(W_{k_L})$ that satisfies the conditions of Lemma \ref{stragood}. Furthermore, we can choose almost good data of $(\pi_L^{-1}(W_L),\pi_L)$ such that 
\begin{itemize}
\item The numerical data of $\overline{Z}^{\circ}\cap \overline{\pi}_{k_L}^{-1}(W_{k_L})$ is $(\{1\},\nu_{i}^{\circ},N_i^{\circ},0,0,\overline{f}\circ\overline{\pi}_{k_L})_{1\leq i\leq n}$, where  $\nu_{i}^{\circ}=1$ and $N_i^{\circ}=0$ for all $1\leq i\leq n$.
\item The numerical data of $\overline{Z}_{J0L}$ is $(G_{J\cO_L},\nu_{J0i},N_{J0i},\mathfrak{n}_{J0L},0,1)_{1\leq i\leq n}$.
\item The numerical data of $\overline{Z}_{J1L}$ is $(G_{J,\cO_L},\nu_{J1i},N_{J1i},\mathfrak{n}_{J1L}, 0,\mathfrak{u}_{J1})_{1\leq i\leq n}$.
\item If $L'$ is a finite extension of $L$, $z\in \overline{Z}_{J1L}(k_{L'})$, and  $(0,...,0,[x]_w)$ is a fixed  representation of $z$ with $x\in k_{L'}^n\setminus\{(0,...,0)\}$, then we can take $\mathfrak{u}_{J1}(z)=\overline{f}_{d_0w}(x).$
\item $\mathfrak{n}_{J0L}, \mathfrak{n}_{J1L}$ are bounded above by $\ord_L(w_{i_J})+\ord_L(d_J)$.
\item $N_{J1i_J}=d_0, \nu_{J1i_J}=|w|, N_{J1i}=\nu_{J1i}-1=0$ if $i\neq i_J$.
\item $N_{J0i_J}=d_0, N_{J0i_J'}=1, \nu_{J0i_J}=|w|, \nu_{J0i_J'}=1$ and $N_{J0i}=\nu_{J0i}-1=0$ if $i\notin\{i_J,i_J'\}$ for an index  $i_J'$ satisfying $i_J'\neq i_J$ and $\varphi_{J\cO_L}(x_{i_J'})=x_{i_J'}\otimes \zeta^{d_0}$.
\end{itemize}
\end{lem}
\begin{proof}It is clear that $$\overline{Z}_{k_L}\cap \overline{\pi}_{k_L}^{-1}(W_{k_L})=(\overline{Z}^{\circ}\cap \overline{\pi}_{k_L}^{-1}(W_{k_L}))\sqcup(\sqcup_{\emptyset\neq J\subset \{1,...,n\}}(\overline{Z}_{J0L}\sqcup\overline{Z}_{J1L})).$$
Moreover, the assertions relating to  $\overline{Z}^{\circ}\cap \overline{\pi}_{k_L}^{-1}(W_{k_L})$ are obvious.

Let $J$ be a non-empty subset of $\{1,...,n\}$, $L'$ be a finite extension of $L$ and $z\in \overline{Z}_{J0L}(k_{L'})$. Suppose that $z=(0,...,0,[x_1:...:x_n]_w)$ for $x=(x_1,...,x_n)\in k_{L'}^n\setminus \{(0,...,0)\}$. Since $\overline{f}_{d_0w}$ defines a hypersurface with an isolated singularity in $\AA_{\overline{k}_L}^n$ and $\overline{f}_{d_0w}(x)=0$, we can use Hensel's lemma 
to obtain a point $\tilde{z}=(\tilde{z}_1,...,\tilde{z}_n)\in \cO_{L'}^n$ such that $\tilde{z}\mod \varpi_{L'}\cO_{L'}=x$  
and $f_{d_0w}(\tilde{z})=0$. Since $x_i=0$ iff $i\notin J$, we have $\tilde{z}_i\in\varpi_{L'}\cO_{L'}$ for all $i\notin J$.   We set $a=\tilde{z}_{i_J}$. Since $x_{i_J}\neq 0$, we have $a\in \cO_L^{\times}$.  We use the map $\tilde{\sigma}_{i_Ja}$ to identify $\AA_{\cO_{L'}}^n/\mu_{i_J,\cO_{L'}}$ with $\tilde{U}'_{i_J}\otimes_{\cO_L}\cO_{L'}\ni z$. In particular, it is easy to verify that  $z=\tilde{\sigma}_{i_J,a}(z')$ with $z'=(\tilde{z}_1,...,\tilde{z}_{i_J-1},0,\tilde{z}_{i_J+1},...,\tilde{z}_n) \mod \varpi_{L'}\cO_{L'}\in \AA_{k_{L'}}^n(k_{L'})=\AA_{\cO_{L'}}^n(k_{L'})$.  Similarly to (\ref{pulbf}), we have $$\tilde{\sigma}_{i_Ja}^*\tilde{\pi}_{\cO_L}^*(f)= x_{i_J}^{d_0}(f_{d_0w}(x_1,...,x_{i_J-1},a,x_{i_J+1},...,x_n)+x_{i_J}h_{a}(x_1,...,x_n))$$
for a polynomial $h_a\in \cO_{L'}[x_1,...,x_n]$. We set $\tilde{t}=(\tilde{t}_1,...,\tilde{t}_n)$ with $\tilde{t}_i=\tilde{z}_i$ if $i\in J$ and $\tilde{t}_i=0$ if $i\notin J$. 
 We use the morphism $\theta_{I,z}$ constructed in Lemma \ref{etalelocalquoti} with respect to $\tilde{t}$, where $I=\{1,...,n\}\setminus J$. More precisely, $\theta_{I,z}=\eta_J\circ T_{\tilde{t}}$, where $\eta_J:\AA_{\cO_{L'}}^n/\mu_{d_J,\cO_{L'}}\to \AA_{\cO_{L'}}^n/\mu_{w_{i_J},\cO_{L'}}$ is the quotient morphism and $T_{\tilde{t}}$ is the $\cO_{L'}$-automorphism of $\AA_{\cO_{L'}}^n/\mu_{w_{i_J},\cO_{L'}}$ induced by the $\mu_{d_J,\cO_{L'}}$-equivariant $\cO_L'$-isomorphism $\cO_{L'}[x_1,...,x_n]\to\cO_{L'}[x_1,...,x_n], x_i\mapsto x_i+\tilde{t}_i$.   Then $\theta_{I,z}$ is an $\cO_{L'}$-étale morphism  $\phi_z$ from a standard scheme over $\cO_L'$ witnessed by $\mu_{d_J,\cO_{L'}}$  to a neighbourhood of $z$ in $\tilde{U}_{i_J}'\otimes_{\cO_L}\cO_{L'}$ that sends $\rho_{\mu_{d_J,\cO_{L'}}}((\varpi_{L'})+(x_1,...,x_n))$ to $z$. Since $\overline{f}_{d_0w}$ defines a hypersurface with an isolated singularity in $\AA_{\overline{k}_L}^n$, the function $g_a:=f_{d_0w}(x_1,...,x_{i_J-1},a,x_{i_J+1},...,x_n)+x_{i_J}h_{a}(x_1,...,x_n)$ is smooth at $z'$. This and the construction of $\theta_{I,z}$ imply that $\rho_{\mu_{d_J,\cO_{L'}}}^*\theta_{I,z}^*\tilde{\pi}_{\cO_{L'}}^*(f)=x_{i_J}^{d_0}\tilde{g}_z$ for a polynomial $\tilde{g}_z\in \cO_{L'}[x_1,...,x_n]$ such that $\tilde{g}_z$ belongs to the ideal $(\varpi_{L'})+(x_1,...,x_n)$ of $\cO_L[x_1,...,x_n]$, $x_{i_J}\nmid \tilde{g}_z$ and $\tilde{g}_z$ is smooth at the origin $(\varpi_{L'})+(x_1,...,x_n)=(0,...,0)\in \AA_{k_{L'}}^n$. Therefore,  there is $i_J'\neq i_J$ with $\frac{\partial \tilde{g}_z}{\partial x_{i_J'}}(0,...,0)\in \cO_{L'}^{\times}$. Moreover, since $\varphi_{J\cO_{L'}}(x_{i_J}^{d_0}\tilde{g}_z)=x_{i_J}^{d_0}\tilde{g}_z\otimes 1$ and $\varphi_{J\cO_{L'}}(x_{i_J}^{d_0})=x_{i_J}^{d_0}\otimes \zeta^{d_0(d_J-1)}$, we have $\varphi_{J\cO_{L'}}(\tilde{g}_z)=\tilde{g}_z\otimes\zeta^{d_0}$. By looking at the linear part of $\tilde{g}_z$, we must have $\varphi_{J\cO_{L'}}(x_{i_J'})=x_{i_J'}\otimes \zeta^{d_0}$. Thus, the change of coordinates $(x_1,...,x_{i_J'-1},x_{i_J'},x_{i_J'+1},...,x_n)\mapsto (x_1,...,x_{i_J'-1},\tilde{g}_z,x_{i_J'+1},...,x_n)$ yields an $\cO_{L'}$-isomorphism $\varrho_{J,z}$ between two Zariski neighbourhoods of $P_J:=\rho_{\mu_{d_J,\cO_{L'}}}((\varpi_{L'})+(x_1,...,x_n))$ in $\AA_{\cO_{L'}}^n/\mu_{d_J,\cO_{L'}}$ sending $P_J$ to itself such that $\rho_{\mu_{d_J,\cO_{L'}}}^*\varrho_{J,z}^*\theta_{I,z}^*\tilde{\pi}_{\cO_{L'}}^*(f)=x_{i_J}^{d_0}x_{i_J'}.$ This implies  the assertions regarding $\overline{Z}_{J0L}$ except the one related to the upper bound of $\eta_{J1L}$.  To show the upper bound of $\eta_{J1L}$, we do similarly to (\ref{pulcanon}) to see that the coefficient of $\AA_{k_{L'}}^n$ in the relative canonical divisor $K_{\AA_{\cO_{L'}}^n/W_{\cO_{L'}}}=\tilde{\sigma}_{i_Ja}^*(K_{\tilde{U}'_{i_J}\otimes_{\cO_L}\cO_{L'}/W_{\cO_{L'}}})+K_{\AA_{\cO_{L'}}^n/\tilde{U}'_{i_J}\otimes_{\cO_{L}}\cO_{L'}}$ associated with the morphism $\tilde{\pi}_{\cO_L}\circ\tilde{\sigma}_{i_Ja}$ is $\ord_{L'}(w_{i_J})$. Thus, the coefficients  of $\AA_{k_{L'}}^n$ in both $K_{\AA_{\cO_{L'}}^n/\tilde{U}'_{i_J}\otimes_{\cO_{L}}\cO_{L'}}$ and $\tilde{\sigma}_{i_Ja}^*(K_{\tilde{U}'_{i_J}\otimes_{\cO_L}\cO_{L'}/W_{\cO_{L'}}})$ are at most  $\ord_{L'}(w_{i_J})$. On the other hand,  we observe that $\tilde{U}'_{i_J}\otimes_{\cO_{L}}\cO_{L'}$ is identified with $\AA_{\cO_{L'}}^n/\mu_{w_{i_J},\cO_{L'}}$ and the action $\varphi_{J\cO_{L'}}$  is of a form similar  to the action $\tilde{\varphi}_{i_J\cO_{L'}}$. Thus, the quotient $\AA_{\cO_{L'}}^n/\mu_{d_J,\cO_{L'}}$ can be described  similarly to $\AA_{\cO_{L'}}^n/\mu_{w_{i_J},\cO_{L'}}$. Therefore,  we see that the coefficient of $\AA_{k_{L'}}^n$ in $K_{\AA_{\cO_{L'}}^n/(\AA_{\cO_{L'}}^n/\mu_{d_J,\cO_{L'}})}$ is at most $\ord_{L'}(d_J)$. Thus, the coefficient of $\AA_{k_{L'}}^n$ in the relative canonical divisor $K_{\AA_{\cO_{L'}}^n/W_{\cO_{L'}}}=\varrho_{J,z}^*\theta_{I,z}^*(K_{\tilde{U}'_{i_J}\otimes_{\cO_L}\cO_{L'}/W_{\cO_{L'}}})+K_{\AA_{\cO_{L'}}^n/(\AA_{\cO_{L'}}^n/\mu_{d_J,\cO_{L'}})}$ associated with the morphism $\tilde{\pi}_{\cO_{L'}}\circ\theta_{I,z}\circ\varrho_{J,z}\circ\rho_{\mu_{d_J,\cO_{L'}}}$ is at most $\ord_{L'}(w_{i_J})+\ord_{L'}(d_J)$. Here, we need to use the fact that $\theta_{I,z}\circ\varrho_{J,z}$ is an $\cO_{L'}$-étale morphism.
 
The assertions relating to $\overline{Z}_{J1L}$ are proved similarly. We can take an arbitrary $\tilde{z}\in \cO_{L'}^n$ such that $\tilde{z}\mod \varpi_L\cO_{L'}= x$. After that, we can  repeat the above strategy, but skip the morphism $\varrho_{J,z}$ since it is not necessary. To compute $\mathfrak{u}_{J1}$, we can use Corollary \ref{angucomp}. 
\end{proof}

\subsection{Proof of Main Theorem \ref{mthm1}}

It follows from Lemma \ref{almostpp}  that Main Theorem \ref{mthm1} is a corollary of the following theorem. 
\begin{thm}\label{smallp}Let $K,w,f,d_0,f_{d_0w}$ be as above.  Let $L$ be a local field over $\cO_K$, $\chi$ be a multiplicative character of $\cO_L^{\times}$ and $Y$ be a subscheme of $\AA_{\cO_K}^n$. Suppose that $\overline{f}_{d_0w}=f_{d_0w} \mod \varpi_L\cO_L$ defines a hypersurface with an isolated singularity in $\AA_{\overline{k}_L}^n$. Let $A$ be the set of the real parts of poles of  $\cZ_{f,\textbf{1}_{\varpi_L\cO_L^n},\chi}(s)\cZ_{f_{d_0w},\Phi_{L,Y},\chi}(s)$. If $d(\chi)\neq 1$ and $d(\chi)\nmid d_0$ then $A=\emptyset$. If $d(\chi)=1$ then  
$A\subset\{-1,\frac{-(w_1+...+w_n)}{d_0}\}$. If $d(\chi)\neq 1$ and $d(\chi)\mid d_0$ then $A\subset\{\frac{-(w_1+...+w_n)}{d_0}\}$. Moreover, the multiplicity of any pole of $\cE_{L,\chi}(s)\cZ_{f,\textbf{1}_{\varpi_L\cO_L^n},\chi}(s)$  or $\cE_{L,\chi}(s)\cZ_{f_{d_0w},\Phi_{L,Y},\chi}(s)$ is at most $1$, where $\cE_{L,\chi}(s)=1$ if $d(\chi)>1$ and $\cE_{L,\chi}(s)=(q_L^{s+1}-1)$ if $d(\chi)=1$. 
\end{thm}
\begin{proof} Note that 
\[\cZ_{f_{d_0w},\Phi_{L,Y},\chi}(s)=\cZ_{f_{d_0w},\textbf{1}_{\varpi_L\cO_L^n},\chi}(s)+\cZ_{f_{d_0w},\Phi_{L,Y(k_L)\setminus\{(0,...,0)\}},\chi}(s).\]
Therefore, the assertions follow immediately from Propositions \ref{ngood}, \ref{chilarge1} and Lemma  \ref{desstraw}. 
\end{proof}

\subsection{Explicit formula for  Igusa zeta functions of weighted homogeneous polynomials with isolated singularities}
From now on until the end of this paper, we take a local field $L$ over $\cO_K$ that satisfies the condition of Lemma \ref{desstraw} and the condition $\ord_{L}(w_i)=0$ for all $i$. For such a local field $L$, Lemma \ref{desstraw} implies that the embedded $\QQ$-resolution $(Z_L,\pi_L)$ of $(\AA_L^n,D_L(f))$  in a small enough Zariski neighbourhood of the origin of $\AA_L^n$ has $\QQ$-good reduction modulo $\varpi_L\cO_L$. In particular, by Lemma \ref{almostpp}, there is an integer $M_0$ 
such that every local field $L$ over $\cO_K$ with residue field characteristic at least $M_0$ will satisfy the above conditions. 

If $J$ is a non-empty subset of $\{1,...,n\}$, recall that we have 
$$\overline{Z}_{J1L}=\{(0,...,0,[u_1,...,u_n]_w)\in \AA_{k_L}^n\times \PP_{k_L}^{w}\mid u_i=0 \textnormal{ iff } i\notin J, \overline{f}_{d_0 w}(u_1,...,u_n)\neq 0\}$$
and 
$$\overline{Z}_{J0L}=\{(0,...,0,[u_1,...,u_n]_w)\in \AA_{k_L}^n\times \PP_{k_L}^{w}\mid u_i=0 \textnormal{ iff } i\notin J, \overline{f}_{d_0 w}(u_1,...,u_n)=0\}.$$
We also set $F_J:=f_{d_0w}|_{\{x_i=0\forall i\notin J\}}$ for each non-empty subset $J$ of $\{1,...,n\}$ with the convention that $F_J=f_{d_0w}$ if $J=\{1,...,n\}$.

We will need to use  the following lemma in the calculation below related to  $\overline{Z}_{J1L}$.
\begin{lem}\label{etaJ}We use the above notations. If  $J$ is a non-empty subset of $\{1,...,n\}$ and $\overline{Z}_{J1L}\neq \emptyset$ then $F_J\neq 0$.  If $J$ is a non-empty subset of $\{1,...,n\}$ such that  $F_J\neq 0$ then $d_J\mid d_0$. This allows us to set $\eta_J:=\gcd(d_0,\frac{d_0(d_J-1)}{d_J})=\frac{d_0}{d_J}\geq 1$ if $F_J\neq 0$.
\end{lem}
\begin{proof}If  $F_J=0$ then it is clear that $\overline{Z}_{J1L}=\emptyset$. This implies the first assertion.

Now, we write $$f_{d_0w}(x_1,...,x_n)=\sum_{I=(i_1,...,i_n)\in \NN^n, w_1i_1+...+w_ni_n=d_0}a_Ix_1^{i_1}...x_n^{i_n}.$$ 

Since $d_J=\gcd(w_i)_{i\in J}$, if  there is a tuple  $I=(i_1,...,i_n)\in \NN^n$ with $a_I\neq 0$ and $i_j=0$ for all $j\notin J$ then it is clear that $d_J\mid d_0=\sum_{j\in J}w_ji_j$. 

Suppose that for each tuple $I=(i_1,...,i_n)\in \NN^n$ with $a_I\neq 0$ there is $j\notin J$ such that $i_j\neq 0$.  Then, for all $I=(i_1,...,i_n)\in\NN^n$ with $a_I\neq 0$, we  have 
 $a_Ix_1^{i_1}...x_n^{i_n}=0$ if  $x_j=0$ for all $j\notin J$. This implies that  $F_J=0$. Thus, we have a contradiction. 
\end{proof}
We also need the following results about the properties of $\overline{Z}_{J1L}$ and $\overline{Z}_{J0L}$.
\begin{lem}\label{L-W}We have 
$$\#(\overline{Z}_{J1L}(k_L))\leq (q_L-1)^{\#(J)-1}$$
and 
$$\#(\overline{Z}_{J0L}(k_L))\leq (q_L-1)^{\#(J)-1}$$
for
all $\emptyset\neq J\subset\{1,...,n\}$.
\end{lem}
\begin{proof}
By \cite[Lemma 7]{Perr}, we have 
$$\#(\{(0,...,0,[u_1,...,u_n]_w)\in \AA_{k_L}^n(k_L)\times \PP_{k_L}^{w}(k_L)\mid u_i=0 \textnormal{ iff } i\notin J\})=(q_L-1)^{\#(J)-1}.$$
Therefore,
the assertion follows from the definition of $\overline{Z}_{J1L}$ and $\overline{Z}_{J0L}$.
\end{proof}

\begin{lem}\label{weig}Let $J$ be a non-empty subset of $\{1,...,n\}$ and $x\in \PP_{k_L}^{w}(k_L)$. Then 
$$\#\{ u\in k_L^n\mid [u]_w=x\}=q_L-1.$$
Moreover, suppose that $x=[u]_w$ for $u=(u_1,...,u_n)\in k_L^n\setminus\{(0,...,0)\}$ such that $u_i=0$ if and only if $i\notin J$. Let $u'=(u'_1,...,u'_n)\in k_L^n\setminus\{(0,...,0)\}$.  Then $[u']_w=x$ if and only if there exists a unique $y\in k_L^{\times}$ such that $u'_i=y^{\frac{w_i}{d_J}}u_i$ for all $1\leq i\leq n$. In particular, if $\overline{Z}_{J1L}\neq\emptyset$  and $[u]_w=x=[u']_w$ then we have $\overline{f}_{d_0w}(u')=y^{\eta_J}\overline{f}_{d_0w}(u)$ for some $y\in k_L^{\times}$.
\end{lem}
\begin{proof}
The first and second assertions follow immediately from \cite[Lemma 7]{Perr}. The last assertion follows from the second assertion together with  Lemma \ref{etaJ} and the fact that $\overline{f}_{d_0w}$ is a $w$-weighted homogeneous polynomial of $w$-degree $d_0$.
\end{proof}

Now, we are ready to present explicit formulas for Igusa's local zeta functions of weighted homogeneous polynomials with isolated singularities. 
\begin{thm}\label{explicit}
 We keep the above notations and assumptions. Let $\chi$ be a multiplicative character of $\cO_L^{\times}$. The following assertions hold:
\begin{itemize}
\item[\textit{i,}] 
$\cZ_{f,\mathbf{1}_{\varpi_L\cO_L^n},\chi}(s)=0$ if  either $c(\chi)>1$ and $\ord_{L}(d_0)=0$ or $c(\chi)=1$ and $d(\chi)\nmid d_0$.
\item[\textit{ii,}] If $c(\chi)=1, d(\chi)\mid d_0$ and $d(\chi)>1$, we have $$\cZ_{f,\textbf{1}_{\varpi_L\cO_L^n},\chi}(s)=q_L^{-n}\left(\sum_{J\subset \{1,...,n\}, J\neq\emptyset,Z_{J1}\neq\emptyset, d(\chi)\mid \eta_{J}}\Omega_{J1,\chi}\sum_{\upsilon=0}^{d_J-1}q_L^{\epsilon_{J}(\upsilon)+\frac{\upsilon}{d_J}(d_0s+|w|)}\right)\frac{(q_L-1)q_L^{-d_0s-|w|}}{1-q_L^{-d_0s-|w|}},$$
where  
$\epsilon_{J}(\upsilon)=\sum_{i\notin J, d_J\nmid \upsilon w_i}(1-\frac{\upsilon w_i}{d_J}+\lfloor \frac{\upsilon w_i}{d_J}\rfloor)$ and

$$\Omega_{J1,\chi}=\frac{1}{(q_L-1)}\sum_{(u_1,...,u_n)\in k_L^n, u_i=0 \textnormal{ iff } i\notin J}\overline{\chi}(\overline{f}_{d_0w}(u_1,...,u_n)).$$
\end{itemize}
\end{thm}
\begin{thm}\label{explicit1}With the above notations and assumptions, we have  
\begin{align*}
&\cZ_{f,\textbf{1}_{\varpi_L\cO_L^n},\chi_{\textnormal{triv}}}(s)\\=&q_L^{-n}\bigg(\sum_{J\subset \{1,...,n\}, J\neq\emptyset,Z_{J1}\neq\emptyset}\#(\overline{Z}_{J1L}(k_L))\sum_{\upsilon=0}^{d_J-1}q_L^{\epsilon_{J}(\upsilon)+\frac{\upsilon}{d_J}(d_0s+|w|)}\bigg)\frac{(q_L-1)q_L^{-d_0s-|w|}}{1-q_L^{-d_0s-|w|}}\\
&+q_L^{-n}\bigg(\sum_{J\subset \{1,...,n\}, J\neq\emptyset,Z_{J0}\neq\emptyset}\#(\overline{Z}_{J0L}(k_L))\sum_{\upsilon=0}^{d_J-1}q_L^{\epsilon_{J}(\upsilon)+\left(\delta_{J\upsilon j_0}+\big\lfloor\frac{d_0\upsilon}{d_J}\big\rfloor-\frac{d_0\upsilon}{d_J}\right)s+\frac{\upsilon}{d_J}(d_0s+|w|)}\bigg)\\&\hspace{1cm}\times\frac{(q_L-1)^2q_L^{-d_0s-|w|}q_L^{-s-1}}{(1-q_L^{-d_0s-|w|})(1-q_L^{-s-1})},
\end{align*}
where $\delta_{J\upsilon}=1$ if $d_J\nmid d_0\upsilon$, $\delta_{J\upsilon}=0$ if $d_J\mid d_0\upsilon$.
\end{thm}
\begin{proof}[Proof of Theorems \ref{explicit} and \ref{explicit1}]
The assertions  follow from Proposition \ref{chilarge1}, Main Theorem \ref{mucase} and Lemma \ref{desstraw}.  Here, it remains to  calculate the constants appearing in Theorem \ref{mucase}. Let $J$ be a non-empty subset of $\{1,...,n\}$. 

First of all,  we calculate the numbers $\gamma_{J\upsilon i},\delta_{J\upsilon i}$ with respect to the group scheme $G_{J\cO_L}=\mu_{d_J,\cO_L}$ and its action $\varphi_{J\cO_L}$ on $\AA_{\cO_L}^n$.  
It is easy to have $\gamma_{J\upsilon i}=\delta_{J\upsilon i}=0$ for all $1\leq i\leq n$ if $\upsilon=0$. If $1\leq \upsilon\leq d_J-1$, then we have $\gamma_{J\upsilon i_J}=\frac{(d_J-1)\upsilon}{d_J}-\big\lfloor\frac{(d_J-1)\upsilon}{d_J}\big\rfloor=1-\frac{\upsilon}{d_J}$, $\delta_{J\upsilon i_J}=1$, 
$\gamma_{J\upsilon i}=\delta_{J\upsilon i}=0$ if $i\in J\setminus\{i_J\}$. Moreover, if $i\notin J$ we have $\gamma_{J\upsilon i}=\frac{w_i\upsilon}{d_J}-\big\lfloor\frac{w_i\upsilon}{d_J}\big\rfloor$, $\delta_{J\upsilon i}=1$ if $d_J\nmid w_i\upsilon$ and $\delta_{J\upsilon i}=0$ if $d_J\mid w_i\upsilon$. In addition, $\gamma_{J\upsilon i_J'}=\frac{d_0\upsilon}{d_J}-\big\lfloor\frac{d_0\upsilon}{d_J}\big\rfloor$, $\delta_{J\upsilon i_J'}=1$ if $d_J\nmid d_0\upsilon$ and $\delta_{J\upsilon i_J'}=0$ if $d_J\mid d_0\upsilon$. We will write $\delta_{J\upsilon}$ instead of $\delta_{J\upsilon i_J'}$. 

Now we can calculate the factors $\Theta_{J1,\chi},\Theta_{J0,\chi}$ with respect to the strata $\overline{Z}_{J0L},\overline{Z}_{J1L}$ over $k_L$ as follows. For the stratum $\overline{Z}_{J1L}$, we use the above calculation, Theorem \ref{mucase} and Lemma \ref{desstraw}
to obtain the factor $$\Lambda_{\overline{Z}_{J1L},\chi}=\sum_{\upsilon=0}^{d_J-1}q_L^{\epsilon_{J}(\upsilon)+\frac{\upsilon}{d_J}(d_0s+|w|)}$$
in our statement. Similarly, we also obtain the factor 
$$\Lambda_{\overline{Z}_{J0L},\chi}=\sum_{\upsilon=0}^{d_J-1}q_L^{\epsilon_{J}(\upsilon)+(\delta_{J\upsilon}+\big\lfloor\frac{d_0\upsilon}{d_J}\big\rfloor-\frac{d_0\upsilon}{d_J})s+\frac{\upsilon}{d_J}(d_0s+|w|)})$$
with respect to the stratum $\overline{Z}_{J0L}$. 

In the remaining part, we calculate the factor $\Omega_{C, \chi,\pi_L,J0}$ and $\Omega_{C, \chi,\pi_L,J1}$ in Theorem \ref{mucase}  with respect to the strata $\overline{Z}_{J0L}$ and $\overline{Z}_{J1L}$ over $k_L$, where $C=\Spec k_L[x_1,...,x_n]/(x_1,...,x_n)$ and $\chi$ is a multiplicative character of $\cO_L^{\times}$ of conductor $c(\chi)=1$. 
It is clear that $$\Omega_{C, \chi_\textnormal{triv},\pi,J0}=\#(\overline{Z}_{J0L}(k_L))$$ and $$\Omega_{C, \chi_\textnormal{triv},\pi,J1}=\#(\overline{Z}_{J1L}(k_L)).$$ If $c(\chi)=1$, $d(\chi)>1$ and $d(\chi)\mid d_0$ then $d(\chi)\nmid 1=\gcd(N_{J0i})_{1\leq i\leq n}$ thus $\overline{Z}_{J0L}$ does not contribute to $\cZ_{f,\textbf{1}_{\varpi_L\cO_L^n},\chi}(s)$ as shown in Theorem \ref{mucase}. As mentioned in the above, we have  $$\overline{Z}_{J1L}=\{(0,...,0,[u_1,...,u_n]_w)\in \AA_{k_L}^n\times \PP_{k_L}^{w}\mid u_i=0 \textnormal{ iff } i\notin J, \overline{f}_{d_0 w}(u_1,...,u_n)\neq 0\}.$$ 
Let $\eta_J$ be as in Lemma \ref{etaJ}. By Theorem \ref{mucase}, if $d(\chi)\nmid \eta_J$ then $\overline{Z}_{J1L}$ does not contribute to $\cZ_{f,\textbf{1}_{\varpi_L\cO_L^n},\chi}(s)$. Suppose that $d(\chi)\mid \eta_J$. For each $z\in \overline{Z}_{J1L}(k_L)$, we fix a representative $z=(0,...,0,[u_1(z),...,u_n(z)]_w)$ for $u_1(z),...,u_n(z)\in k_L$. By Lemma \ref{desstraw}, 
we can take $\mathfrak{u}_{J1}(z)=\overline{f}_{d_0w}(u_1(z),...,u_n(z))$   for every $z\in\overline{Z}_{J1L}(k_L)$. 
Note that Lemma \ref{weig} together with the conditions that $d(\chi)\mid \eta_J=\frac{d_0}{d_J}$  and $f_{d_0w}$ is a $w$-weighted homogeneous polynomial of $w$-degree $d_0$ assure that $\overline{\chi}(\overline{f}_{d_0w}(u_1(z),...,u_n(z)))$ does not depend on the representative $(u_1(z),...,u_n(z))$ of $z$. This agrees with Remark \ref{trivindep}.  On the other hand, each point $\overline{z}\in \overline{Z}_{J1L}(k_L)$  has exactly $q_L-1$ representatives with $k_L$-coefficients as shown in Lemma \ref{weig}.  Thus we obtain the factor $\Omega_{J1,\chi}$ as in our statement by using the fact that $\overline{\chi}(0)=0$. 
\end{proof}

\subsection{Some auxiliary results}\label{auxres}
In this subsection, we will give some preparations for the proof of Main Theorem \ref{mthm2}. In Theorems \ref{explicit} and \ref{explicit1}, we showed explicit formulas for Igusa's local zeta functions associated with $f$ and the Schwartz-Bruhat function $\textbf{1}_{\varpi_L\cO_L^n}$. Thus, we can use Proposition \ref{zetaex} to compute exponential sums associated with $f$ and $\textbf{1}_{\varpi_L\cO_L^n}$. This leads to estimating the coefficients of $t^{l}$ for $l\geq 1$ and $t=q_L^{-s}$ in the power series expansions of suitable rational functions in $q_L^{-s}$. Based on Theorems \ref{explicit} and \ref{explicit1}, we will study these estimates with respect to each non-empty subset $J$ of $\{1,...,n\}$ before summing them over all subsets. Recall that $L$ is a local field over $\cO_K$ that satisfies the condition of Lemma \ref{desstraw} and  $\ord_L(w_i)=0$ for all $i$. From now on, we also suppose that $\ord_L(d_0)=0$.

\begin{lem}\label{estm1}We have
\begin{align*}
&\bigg|q_L^{-n}\#(\overline{Z}_{J1L}(k_L))\sum_{\upsilon=0}^{d_J-1}q_L^{\epsilon_{J}(\upsilon)+\frac{\upsilon}{d_J}|w|}\frac{(q_L-1)q_L^{-|w|}}{1-q_L^{-|w|}}\\&-\textnormal{coeff}_{t^{m-1}}\left(q_L^{-n}\#(\overline{Z}_{J1L}(k_L))
\times\sum_{\upsilon=0}^{d_J-1}t^{-\frac{\upsilon d_0}{d_J}}q_L^{\epsilon_{J}(\upsilon)+\frac{\upsilon}{d_J}|w|}\frac{(q_L-1)t^{d_0}q_L^{-|w|}}{(1-t^{d_0}q_L^{-|w|})(1-t)}\right)\bigg|\\
\leq &\frac{d_Jq_L^{\frac{-m|w|}{d_0}}}{1-q_L^{-|w|}}
\end{align*}
for all non-empty subsets $J$ of $\{1,...,n\}$ 
and  all $m\geq 1$.
\end{lem}
\begin{proof}
We have 
\begin{align*}
&\bigg|q_L^{-n}\#(\overline{Z}_{J1L}(k_L))\sum_{\upsilon=0}^{d_J-1}q_L^{\epsilon_{J}(\upsilon)+\frac{\upsilon}{d_J}|w|}\frac{(q_L-1)q_L^{-|w|}}{1-q_L^{-|w|}}\\&-\textnormal{coeff}_{t^{m-1}}\left(q_L^{-n}\#(\overline{Z}_{J1L}(k_L))
\times\sum_{\upsilon=0}^{d_J-1}t^{-\frac{\upsilon d_0}{d_J}}q_L^{\epsilon_{J}(\upsilon)+\frac{\upsilon}{d_J}|w|}\frac{(q_L-1)t^{d_0}q_L^{-|w|}}{\left(1-t^{d_0}q_L^{-|w|}\right)(1-t)}\right)\bigg|\\
=&\#(\overline{Z}_{J1L}(k_L))\sum_{\upsilon=0}^{d_J-1}q_L^{-n-|w|+\epsilon_{J}(\upsilon)+\frac{\upsilon}{d_J}|w|}(q_L-1)\left|\frac{1}{1-q_L^{-|w|}}-\textnormal{coeff}_{t^{m-1}}\left(\frac{t^{d_0-\frac{\upsilon d_0}{d_J}}}{(1-t)\left(1-t^{d_0}q_L^{-|w|}\right)}\right)\right|\\
=&\#(\overline{Z}_{J1L}(k_L))\sum_{\upsilon=0}^{d_J-1}q_L^{-n-|w|+\epsilon_{J}(\upsilon)+\frac{\upsilon}{d_J}|w|}(q_L-1)\left|\left(\sum_{r\geq 0}q_L^{-r|w|}\right)-\textnormal{coeff}_{t^{m-1-d_0+\frac{\upsilon d_0}{d_J}}}\left(\frac{1}{(1-t)\left(1-t^{d_0}q_L^{-|w|}\right)}\right)\right|\\
=&\#(\overline{Z}_{J1L}(k_L))\sum_{\upsilon=0}^{d_J-1}q_L^{-n-|w|+\epsilon_{J}(\upsilon)+\frac{\upsilon}{d_J}|w|}(q_L-1)\left|\left(\sum_{r\geq 0}q_L^{-r|w|}\right)-\textnormal{coeff}_{t^{m-1-d_0+\frac{\upsilon d_0}{d_J}}}\left(\sum_{\ell\geq 0}\sum_{0\leq rd_0\leq \ell}q_L^{-r|w|}t^\ell\right)\right|\\
=&\#(\overline{Z}_{J1L}(k_L))\sum_{\upsilon=0}^{d_J-1}q_L^{-n-|w|+\epsilon_{J}(\upsilon)+\frac{\upsilon}{d_J}|w|}(q_L-1)\sum_{d_0r+d_0-\frac{\upsilon d_0}{d_J}\geq m}q_L^{-r|w|}\\
\leq& \#(\overline{Z}_{J1L}(k_L))\sum_{\upsilon=0}^{d_J-1}q_L^{-n-|w|+\epsilon_{J}(\upsilon)+\frac{\upsilon}{d_J}|w|}\frac{q_L-1}{1-q_L^{-|w|}}q_L^{-\frac{m-d_0+\frac{\upsilon d_0}{d_J}}{d_0}|w|}\\
\end{align*}
\begin{align*}
&\hspace{-8.5cm}\leq d_J\#(\overline{Z}_{J1L}(k_L))q_L^{-n+\epsilon_{J}(\upsilon)}(q_L-1)\frac{q_L^{\frac{-m|w|}{d_0}}}{1-q_L^{-|w|}}.
\end{align*}
Our assertion follows from the facts that $\#(\overline{Z}_{J1L}(k_L))\leq (q_L-1)^{\#(J)-1}$ 
as shown in Lemma \ref{L-W} and 
$$\epsilon_{J}(\upsilon)=\sum_{i\notin J, d_J\nmid \upsilon w_i}\left(1-\frac{\upsilon w_i}{d_J}+\bigg\lfloor \frac{\upsilon w_i}{d_J}\bigg\rfloor\right)\leq n-\#(J).$$

\end{proof}

\begin{lem-defn}\label{exis}
Let $J$ be a non-empty subset of $\{1,...,n\}$ and $m\in\ZZ$. Then we have $d_J\leq d_0$ and  there is at most one tuple $(\upsilon,r)$ with $0\leq\upsilon\leq d_J-1$ and $r\in\NN$ such that $d_0r+d_0+1-(\delta_{J\upsilon}+\big\lfloor\frac{d_0\upsilon}{d_J}\big\rfloor)=m$. If such a tuple $(\upsilon,r)$ exists, we will set $\upsilon_{Jm}=\upsilon$. If $\overline{Z}_{J1L}\neq\emptyset$ and $\upsilon_{Jm}$ exists then $d_0r+d_0+1-\frac{d_0\upsilon_{Jm}}{d_J}=m$ for a natural number $r$.
\end{lem-defn}
\begin{proof}
Since $f_{d_0w}$ is a $w$-weighted homogeneous polynomial of degree $d_0$ and defines a hypersurface with an isolated singularity in $\AA_\CC^n$, we have $d_J=\gcd(w_i)_{i\in J}\leq \min_{i\in J}w_i\leq d_0$. Suppose that we have two tuples $(\upsilon,r)$, $(\upsilon',r')$ with $0\leq \upsilon,\upsilon'\leq d_J-1$ and $r,r'\in\NN$ such that $$d_0r+d_0+1-\left(\delta_{J\upsilon}+\bigg\lfloor\frac{d_0\upsilon}{d_J}\bigg\rfloor\right)=d_0r'+d_0+1-\left(\delta_{J\upsilon'}+\bigg\lfloor\frac{d_0\upsilon'}{d_J}\bigg\rfloor\right).$$
This is equivalent with
\begin{equation}\label{equi}d_0(r-r')=(\delta_{J\upsilon}-\delta_{J\upsilon'})+\left(\bigg\lfloor\frac{d_0\upsilon}{d_J}\bigg\rfloor-\bigg\lfloor\frac{d_0\upsilon'}{d_J}\bigg\rfloor\right).
\end{equation}
If $d_0=d_J$ then $\delta_{J\upsilon}=\delta_{J\upsilon'}=0$. Thus we have $d_0(r-r')=\upsilon-\upsilon'$. By this and the fact that $0\leq \upsilon,\upsilon'\leq d_J-1=d_0-1$, we have $r=r',\upsilon=\upsilon'$.

Suppose that $d_0>d_J$. By the definition of $\delta_{J,\upsilon}$, we have  
\begin{equation}\label{inequi}
\delta_{J\upsilon}, \delta_{J\upsilon'}\in\{0,1\} \textnormal{ and } -1\leq \delta_{J\upsilon}-\delta_{J\upsilon'}\leq 1.
\end{equation} Since $0\leq\upsilon,\upsilon'\leq d_J-1$ and $d_0>d_J$ we have $$0\leq \bigg\lfloor\frac{d_0\upsilon}{d_J}\bigg\rfloor,\bigg\lfloor\frac{d_0\upsilon'}{d_J}\bigg\rfloor\leq \bigg\lfloor\frac{d_0(d_J-1)}{d_J} \bigg\rfloor\leq d_0-2.$$
Thus, $$-d_0+1\leq d_0(r-r')=(\delta_{J\upsilon}-\delta_{J\upsilon'})+\left(\bigg\lfloor\frac{d_0\upsilon}{d_J}\bigg\rfloor-\bigg\lfloor\frac{d_0\upsilon'}{d_J}\bigg\rfloor\right)\leq d_0-1.$$
So we must have $r=r'$. Suppose that $\upsilon\neq\upsilon'$, then the condition $d_0>d_J$ implies $$\left|\bigg\lfloor\frac{d_0\upsilon}{d_J}\bigg\rfloor-\bigg\lfloor\frac{d_0\upsilon'}{d_J}\bigg\rfloor\right|\geq 1$$
and if the equality holds then $|v-v'|=1$. This and (\ref{equi}), (\ref{inequi}) imply 
\begin{equation}\label{equ00}\left|\bigg\lfloor\frac{d_0\upsilon}{d_J}\bigg\rfloor-\bigg\lfloor\frac{d_0\upsilon'}{d_J}\bigg\rfloor\right|= 1 \textnormal{ and } |v-v'|=1.
\end{equation}
Without loss of generality, we can suppose that $v=v'+1$. By using (\ref{equi}) and  (\ref{inequi}) again we have $\delta_{J\upsilon}=0$, $\delta_{J\upsilon'}=1$. By the definition of $\delta_{J\upsilon}, \delta_{J\upsilon'}$ we have $d_J\mid \upsilon d_0$ and $d_J\nmid \upsilon' d_0$. This and the condition $d_0>d_J$ imply
$$\bigg\lfloor\frac{d_0\upsilon}{d_J}\bigg\rfloor-\bigg\lfloor\frac{d_0\upsilon'}{d_J}\bigg\rfloor=\frac{d_0\upsilon}{d_J}-\bigg\lfloor\frac{d_0\upsilon-d_0}{d_J}\bigg\rfloor\geq \frac{d_0\upsilon}{d_J}-\left(\frac{d_0\upsilon}{d_J}-2\right)\geq 2.$$
This contradicts with (\ref{equ00}). Thus $\upsilon=\upsilon'$ as desired.

Suppose that $\overline{Z}_{J1L}\neq\emptyset$. By Lemma \ref{etaJ}, we have $d_J\mid d_0$. Thus, $\delta_{J\upsilon}=0$ for all $1\leq \upsilon\leq d_J-1$. Moreover, if $\upsilon_{Jm}$ exists then $d_0r+d_0+1-\frac{d_0\upsilon_{Jm}}{d_J}=m$ for a natural number $r$.
\end{proof}

\begin{lem}\label{estm2}We have
\begin{align*}
D_{m,J,L,1}&:=-\textnormal{coeff}_{t^{m-1}}\left(\frac{q_L^{-n}}{q_L-1}\#(\overline{Z}_{J1L}(k_L))
\times\sum_{\upsilon=0}^{d_J-1}t^{-\frac{\upsilon d_0}{d_J}}q_L^{\epsilon_{J}(\upsilon)+\frac{\upsilon}{d_J}|w|}\frac{(q_L-1)t^{d_0}q_L^{-|w|}}{(1-t^{d_0}q_L^{-|w|})}\right)\\
&=\begin{cases}-q_L^{-n}\#(\overline{Z}_{J1L}(k_L))q_L^{\epsilon_{J}(\upsilon_{Jm})+\frac{|w|}{d_0}}q_L^{-\frac{|w|m}{d_0}} \textnormal{ if } \overline{Z}_{J1L}\neq\emptyset \textnormal{ and } \upsilon_{Jm} \textnormal{ exists } \\
0 \textnormal{ otherwise}
\end{cases}
\end{align*}
for all non-empty subsets $J$ of $\{1,...,n\}$ 
and all $m\geq 1$.
\end{lem}
\begin{proof}It follows from Lemmas \ref{etaJ} and \ref{exis} that 
$$\frac{\upsilon_{Jm}}{d_J}+\frac{1}{d_0}\left(1-\delta_{J\upsilon_{Jm}}-\bigg\lfloor\frac{d_0\upsilon_{Jm}}{d_J}\bigg\rfloor\right)=\frac{1}{d_0}$$
if $\overline{Z}_{J1L}\neq\emptyset$ and $\upsilon_{Jm}$  exists.
Therefore, we can proceed similarly to  Lemma \ref{estm1} to prove the equality. 
\end{proof}
\begin{lem}\label{estm3}
We have 
\begin{align*}
D_{m,J,L,2}&:=\textnormal{coeff}_{t^{m-1}}\bigg( q_L^{-n}\#(\overline{Z}_{J0L}(k_L))\sum_{\upsilon=0}^{d_J-1}q_L^{\epsilon_{J}(\upsilon)+\frac{\upsilon |w|}{d_J}}t^{-(\delta_{J\upsilon}+\big\lfloor\frac{d_0\upsilon}{d_J}\big\rfloor)}\times\frac{(q_L-1)^2t^{d_0}q_L^{-|w|}tq_L^{-1}}{\left(1-t^{d_0}q_L^{-|w|}\right)(1-tq_L^{-1})}\\&\hspace{9.5cm}\times \frac{t-q_L}{(q_L-1)(1-t)}\bigg)\\
&\hspace{0.6cm}+q_L^{-n}\#(\overline{Z}_{J0L}(k_L))\sum_{\upsilon=0}^{d_J-1}q_L^{\epsilon_{J}(\upsilon)+\frac{\upsilon |w|}{d_J}}\times\frac{(q_L-1)^2q_L^{-|w|}q_L^{-1}}{\left(1-q_L^{-|w|}\right)(1-q_L^{-1})}\\
&=\begin{cases}(q_L-1)q_L^{-n}\#(\overline{Z}_{J0L}(k_L))q_L^{a_{Jm}}q_L^{-\frac{|w|m}{d_0}}+ c(m,L)q_L^{-\frac{|w|m}{d_0}}  \textnormal{ if } \upsilon_{Jm} \textnormal{ exists } \\
 e(m,L)q_L^{-\frac{|w|m}{d_0}}\textnormal{ otherwise}
\end{cases}
\end{align*}
for all non-empty subsets $J$ of $\{1,...,n\}$ 
 and all $m\geq 1$, where $$a_{Jm}=\epsilon_{J}(\upsilon_{Jm})+\left(\frac{\upsilon_{Jm}}{d_J}+\frac{1}{d_0}\left(1-\delta_{J\upsilon_{Jm}}-\bigg\lfloor\frac{d_0\upsilon_{Jm}}{d_J}\bigg\rfloor\right)\right)|w|$$ and there is a constant $c_0$ depending only on $w$ such that $\max\{|c(m,L),|e(m,L)|\}|\leq c_0$.
\end{lem}
\begin{proof}By a similar calculation to Lemma \ref{estm1}, we have 
\begin{align}\label{mJL2}
D_{m,J,L,2}=&\textnormal{coeff}_{t^{m-1}}\bigg( q_L^{-n}\#(\overline{Z}_{J0L}(k_L))\sum_{\upsilon=0}^{d_J-1}q_L^{\epsilon_{J}(\upsilon)+\frac{\upsilon |w|}{d_J}}t^{-\left(\delta_{J\upsilon}+\big\lfloor\frac{d_0\upsilon}{d_J}\big\rfloor\right)}\times\frac{(q_L-1)^2t^{d_0}q_L^{-|w|}tq_L^{-1}}{\left(1-t^{d_0}q_L^{-|w|}\right)(1-tq_L^{-1})}\nonumber\\&\hspace{9.5cm}\times \frac{t-q_L}{(q_L-1)(1-t)}\bigg)\nonumber\\
&+q_L^{-n}\#(\overline{Z}_{J0L}(k_L))\sum_{\upsilon=0}^{d_J-1}q_L^{\epsilon_{J}(\upsilon)+\frac{\upsilon |w|}{d_J}}\times\frac{(q_L-1)^2q_L^{-|w|}q_L^{-1}}{\left(1-q_L^{-|w|}\right)(1-q_L^{-1})}\nonumber\\
=&(q_L-1)q_L^{-n-|w|}\#(\overline{Z}_{J0L}(k_L))\sum_{\upsilon=0}^{d_J-1}q_L^{\epsilon_{J}(\upsilon)+\frac{\upsilon |w|}{d_J}}\sum_{d_0r+d_0+1-\left(\delta_{J\upsilon}+\big\lfloor\frac{d_0\upsilon}{d_J}\big\rfloor\right)\geq m}q_L^{-r|w|}
\end{align}
for
all $m\geq 1$ and all $\emptyset\neq J\subset\{1,...,n\}$.
 By Lemma-Definition \ref{exis}, we have 
\begin{align}\label{series}
&\sum_{d_0r+d_0+1-\left(\delta_{J\upsilon}+\big\lfloor\frac{d_0\upsilon}{d_J}\big\rfloor\right)\geq m}q_L^{-r|w|}\nonumber\\=&\left(\sum_{d_0r+d_0+1-\left(\delta_{J\upsilon}+\big\lfloor\frac{d_0\upsilon}{d_J}\big\rfloor\right)\geq m+1}q_L^{-r|w|}\right)+\begin{cases}
q_L^{-\left(\frac{m-d_0-1+\delta_{J\upsilon_{Jm}}+\big\lfloor\frac{d_0\upsilon_{Jm}}{d_J}\big\rfloor}{d_0}\right)|w|} \textnormal{ if } \upsilon=\upsilon_{Jm},   \\
0 \textnormal{ otherwise}.
\end{cases}
\end{align}
Moreover, it is easy to have 
\begin{align}\label{d0r}
\sum_{d_0r+d_0+1-\left(\delta_{J\upsilon}+\big\lfloor\frac{d_0\upsilon}{d_J}\big\rfloor\right)\geq m+1}q_L^{-r|w|}&\leq \frac{1}{1-q_L^{-|w|}}q_L^{-\left(\frac{m-d_0+\delta_{J\upsilon}+\big\lfloor\frac{d_0\upsilon}{d_J}\big\rfloor}{d_0}\right)|w|} \nonumber\\&\leq \frac{1}{1-2^{-|w|}}q_L^{-\left(\frac{m-d_0+\delta_{J\upsilon}+\big\lfloor\frac{d_0\upsilon}{d_J}\big\rfloor}{d_0}\right)|w|}.
\end{align}
Thus, the assertion follows from Lemma \ref{L-W},  (\ref{mJL2}), (\ref{series}), (\ref{d0r}) and the facts that
$$\frac{\upsilon}{d_J}-\frac{1}{d_0}\left(\delta_{J\upsilon}+\bigg\lfloor\frac{d_0\upsilon}{d_J}\bigg\rfloor\right)\leq 0$$
 and 
$$\epsilon_{J}(\upsilon)=\sum_{i\notin J, d_J\nmid \upsilon w_i}\left(1-\frac{\upsilon w_i}{d_J}+\bigg\lfloor \frac{\upsilon w_i}{d_J}\bigg\rfloor\right)\leq n-\#(J)$$
for all $0\leq\upsilon\leq d_J-1$. 
\end{proof}

\begin{lem}\label{estm4}
Let 
$J$ be a non-empty subset of $\{1,...,n\}$ such that $\overline{Z}_{J1L}\neq\emptyset$. Let $\psi$ be an additive character of $L$ of conductor $0$ and $\Psi$  be the additive character of $k_L$ induced by $\psi_{\varpi_L^{-1}}$. Let $a\in \cO_L^{\times}$.  Then we have 
\begin{align*}
D_{m,J,L}(a)&:=\sum_{\chi\neq\chi_{\textnormal{triv}}, c(\chi)=1,\chi^{\eta_J}=1} g_{\chi^{-1}}\chi(a)\textnormal{coeff}_{t^{m-1}}\left(q_L^{-n}\Omega_{J1,\chi}
\sum_{\upsilon=0}^{d_J-1}t^{-\frac{\upsilon d_0}{d_J}}q_L^{\epsilon_{J}(\upsilon)+\frac{\upsilon}{d_J}|w|}\frac{(q_L-1)t^{d_0}q_L^{-|w|}}{\left(1-t^{d_0}q_L^{-|w|}\right)}\right)\\
&=\begin{cases}q_L^{-n+\epsilon_J(\upsilon_{Jm})-\frac{(m-1)|w|}{d_0}}\left(\#(\overline{Z}_{J1L}(k_L))-(q_L-1)\#(\overline{Z}_{J0L}(k_L))+ H_{J,L,\Psi,f_{d_0w},a}\right)\textnormal{ if } \upsilon_{Jm} \textnormal{ exists } \\
 0 \textnormal{ otherwise},
\end{cases}
\end{align*}
where
$$H_{J,L,\Psi,f_{d_0w},a}=\sum_{u=(u_1,...,u_n)\in k_L^n, u_i=0 \Leftrightarrow i\notin J}\Psi(\overline{a}\overline{f}_{d_0 w}(u))$$
and if $\chi$ is a multiplicative character of $\cO_L^{\times}$ then
$g_{\chi}$ is the Gaussian sum
\[
g_{\chi}=\frac{q_L^{1-c(\chi)}}{q_L-1}\sum_{\overline{v}\in (\cO_{L}/\varpi_L^{c(\chi)})^{\times}}\chi(v)\psi\bigl(v\bigr)
\] defined in Proposition \ref{zetaex}.
\end{lem}
\begin{proof}
By Lemma \ref{estm2}, we have
\begin{align}\label{coe}
&\textnormal{coeff}_{t^{m-1}}\left(\frac{q_L^{-n}}{q_L-1}
\times\sum_{\upsilon=0}^{d_J-1}t^{-\frac{\upsilon d_0}{d_J}}q_L^{\epsilon_{J}(\upsilon)+\frac{\upsilon}{d_J}|w|}\frac{(q_L-1)t^{d_0}q_L^{-|w|}}{(1-t^{d_0}q_L^{-|w|})}\right)\nonumber\\
=&\begin{cases}q_L^{-n}q_L^{\epsilon_{J}(\upsilon_{Jm})+\frac{|w|}{d_0}}q_L^{-\frac{|w|m}{d_0}} \textnormal{ if } \upsilon_{Jm} \textnormal{ exists } \\
0 \textnormal{ otherwise}
\end{cases}
\end{align}

We set $e_{J,L}=\gcd\{\eta_J,q_L-1\}$, $$\tilde{Z}_{JL}=\{u=(u_1,...,u_n)\in k_L^n\mid u_i=0 \textnormal{ iff } i\notin J \}$$ and $$\tilde{Z}_{J1L}=\{u=(u_1,...,u_n)\in k_L^n\mid u_i=0 \textnormal{ iff } i\notin J, \overline{f}_{d_0w}(u)\neq 0\}.$$ Let $\chi$ be a multiplicative character of $\cO_L^{\times}$. If  $c(\chi)=1$ then $\chi^{q_L-1}=1$. Thus, if $c(\chi)=1$ and $\chi^{\eta_J}=1$ then $\chi$ is of order $e_{J,L}$. This fact and the definition of $\Omega_{J1,\chi}$ yield 
\begin{align}
&\sum_{\chi\neq\chi_{\textnormal{triv}}, c(\chi)=1,\chi^{\eta_J}=1} g_{\chi^{-1}}\chi(a)(q_L-1)\Omega_{J1,\chi}\nonumber\\
=&\sum_{\chi\neq\chi_{\textnormal{triv}}, c(\chi)=1,\chi^{e_{J,L}}=1} g_{\chi^{-1}}\chi(a)(q_L-1)\Omega_{J1,\chi}\nonumber\\
=&\sum_{\chi\neq\chi_{\textnormal{triv}}, c(\chi)=1,\chi^{e_{J,L}}=1} \frac{1}{q_L-1}\sum_{\overline{v}\in k_L^{\times}}\chi\bigl(v^{-1}a\bigr)\psi\bigl(\frac{v}{\varpi_L}\bigr)(q_L-1)\Omega_{J1,\chi}\nonumber\\
=&\sum_{\chi\neq\chi_{\textnormal{triv}}, c(\chi)=1,\chi^{e_{J,L}}=1} \frac{1}{q_L-1}\sum_{\overline{v}\in k_L^{\times}}\overline{\chi}\bigl(\overline{v}^{-1}\overline{a}\bigr)\Psi\bigl(\overline{v}\bigr)\sum_{u\in\tilde{Z}_{J1L}}\overline{\chi}\bigl(\overline{f}_{d_0w}(u)\bigr)\nonumber\\
=&\frac{1}{q_L-1}\sum_{u\in\tilde{Z}_{J1L}}\sum_{\overline{v}\in k_L^{\times}}\Psi\bigl(\overline{v}\bigr)\bigg(-1+\sum_{c(\chi)=1,\chi^{e_{J,L}}=1} \overline{\chi}\bigl(\overline{v}^{-1}\overline{a}\overline{f}_{d_0w}(u)\bigr)\bigg).\label{chi1}
\end{align}
By orthogonality of characters, Lemma \ref{weig}  and the definition of $\overline{Z}_{J1L}$, we have
\begin{equation}\label{chi2}
\frac{1}{q_L-1}\sum_{u\in\tilde{Z}_{J1L}}\sum_{\overline{v}\in k_L^{\times}}\Psi\bigl(\overline{v}\bigr)=\frac{1}{q_L-1}\sum_{u\in\tilde{Z}_{J1L}}\bigg(-1+\sum_{\overline{v}\in k_L}\Psi\bigl(\overline{v}\bigr)\bigg)=-\#\left(\overline{Z}_{J1L}(k_L)\right).
\end{equation}
On the other hand, if $b\in k_L$ then we have 
$$\sum_{c(\chi)=1,\chi^{e_{J,L}}=1} \overline{\chi}(b)=\begin{cases} e_{J,L} \textnormal{ if } \exists y\in k_L^{\times} \textnormal{ such that } b=y^{e_{J,L}}\\
0 \textnormal{ otherwise}.
\end{cases}$$
Since $e_{J,L}=\gcd(\eta_J,q_L-1)\mid q_L-1$, we have $\#(\{y\in k_L^{\times}\mid y^{e_{J,L}}=1\})=\#(\{y\in k_L^{\times}\mid y^{\eta_J}=1\})=e_{J,L}$. Moreover $\{y^{\eta_J}\mid y\in k_L^{\times}\}=\{y^{e_{J,L}}\mid y\in k_L^{\times}\}$. These facts and Lemma \ref{weig} imply
\begin{align}\label{chi3}
&\frac{1}{q_L-1}\sum_{u\in\tilde{Z}_{J1L}}\sum_{\overline{v}\in k_L^{\times}}\Psi\bigl(\overline{v}\bigr)\sum_{c(\chi)=1,\chi^{e_{J,L}}=1} \overline{\chi}\bigl(\overline{v}^{-1}\overline{a}\overline{f}_{d_0w}(u)\bigr)\nonumber\\
=&\frac{1}{q_L-1}\sum_{u\in\tilde{Z}_{J1L}}\sum_{y\in k_L^{\times}}\Psi\bigl(y^{e_{J,L}}\overline{a}\overline{f}_{d_0w}(u)\bigr)\nonumber\\
=&\frac{1}{q_L-1}\sum_{u\in\tilde{Z}_{J1L}}\sum_{y\in k_L^{\times}}\Psi\bigl(y^{\eta_{J}}\overline{a}\overline{f}_{d_0w}(u)\bigr)\nonumber\\
=&\frac{1}{q_L-1}\sum_{u\in\tilde{Z}_{J1L}}\sum_{u'\in \tilde{Z}_{J1L}, [u]_w=[u']_w\in \PP_{k_L}^{w}(k_L)}\Psi\bigl(\overline{a}\overline{f}_{d_0w}(u')\bigr)\nonumber\\
=&\sum_{u\in\tilde{Z}_{J1L}}\Psi\bigl(\overline{a}\overline{f}_{d_0w}(u)\bigr)\nonumber\\
=&\left(\sum_{u\in\tilde{Z}_{JL}}\Psi\bigl(\overline{a}\overline{f}_{d_0w}(u)\bigr)\right)-\left(\sum_{u\in\tilde{Z}_{JL}, \overline{f}_{d_0w}(u)=0}1\right)\nonumber\\
=&-(q_L-1)\#(\overline{Z}_{J0L}(k_L))+\sum_{u\in\tilde{Z}_{JL}}\Psi\bigl(\overline{a}\overline{f}_{d_0w}(u)\bigr).
\end{align}

Lastly, we can combine (\ref{coe}), (\ref{chi1}), (\ref{chi2}) and (\ref{chi3}) to prove the assertion.
\end{proof}

\begin{cor}\label{estm5}With the notations in Lemma \ref{estm3} and the above assumptions, there is a positive constant $c_1$ depending only on $w$ such that 
$$\bigg|E_{f,\textbf{1}_{\varpi_L\cO_L^n},\psi_{a\varpi_L^{-m}}}-q_L^{-\frac{m|w|}{d_0}}\sum_{J\subset\{1,...,n\}, J\neq\emptyset, \upsilon_{Jm} \textnormal{ exists }}q_L^{\lambda(J,m)}H_{J,L,\Psi,f_{d_0w},a}\bigg|\leq c_1q_L^{-\frac{|w|m}{d_0}}$$
for 
all additive characters $\psi$ of $L$ of conductor $0$, all $a\in \cO_L^{\times}$, all $m\geq 1$, where $\Psi$ is the character of $k_L$ induced by $\psi_{\varpi^{-1}}$ and $$\lambda(J,m)=-n+\epsilon_{J}(\upsilon_{Jm})+\left(\frac{\upsilon_{Jm}}{d_J}+\frac{1}{d_0}\left(1-\delta_{J\upsilon_{Jm}}-\bigg\lfloor\frac{d_0\upsilon_{Jm}}{d_J}\bigg\rfloor\right)\right)|w|.$$

\end{cor}
\begin{proof}
If $\overline{Z}_{J1L}\neq\emptyset$ then $d_J\mid d_0$ so 
$$\epsilon_{J}(\upsilon)+\left(\frac{\upsilon}{d_J}+\frac{1}{d_0}\left(1-\delta_{J\upsilon}-\bigg\lfloor\frac{d_0\upsilon}{d_J}\bigg\rfloor\right)\right)|w|=\epsilon_{J}(\upsilon)+\frac{|w|}{d_0}$$
for all $0\leq\upsilon\leq d_J-1$.
By Lemma \ref{weig}, if $\overline{Z}_{J1L}=\emptyset$ then 
$$H_{J,L,\Psi,f_{d_0w},a}=(q_L-1)^{\#(J)}=(q_L-1)\#(\overline{Z}_{J0L}(k_L)).$$
By Proposition \ref{zetaex}, Theorems \ref{explicit}, \ref{explicit1} and the equality
\begin{equation}\label{simfor}\frac{t-q_L}{(q_L-1)(1-t)}=\frac{-1}{q_L-1}+\frac{-1}{1-t},
\end{equation}
we can write  $E_{f,\textbf{1}_{\varpi_L\cO_L^n},\psi_{a\varpi_L^{-m}}}$ as the sum over non-empty sets $J\subset\{1,...,n\}$ of $D_{m,J,L,1}$, $D_{m,J,L,2}$, $D_{m,J,L}(a)$ and the expression inside the absolute value bars in Lemma \ref{estm1}. Thus, the assertion follows from the combination of Lemmas \ref{estm1}, \ref{estm2}, \ref{estm3} and \ref{estm4}. 
\end{proof}

\begin{lem}\label{exfin}There is a natural number $M_1\geq \max\{M_0,d_0+1\}$ depending on $f_{d_0w}$ and a  constant $c_2>0$ depending on $n,w,d_0$ such that 

$$\left|H_{J,L,\Psi,f_{d_0w},a}\right|=\left|\sum_{u=(u_1,...,u_n)\in k_L^n, u_i=0 \Leftrightarrow i\notin J}\Psi(\overline{a}\overline{f}_{d_0 w}(u))\right|\leq c_2q_L^{\frac{\#(J)+s_J}{2}},$$
for all non-empty subsets $J$ of $\{1,...,n\}$ with $\overline{Z}_{J1L}\neq \emptyset$, all $\overline{a}\in k_L^{\times}$ and all non-trivial character $\Psi$ of $k_L$ provided that the residue field characteristic of $L$ is at least $M_1$, where $s_J$ is the dimension of the singular locus of $F_J=f_{d_0w}|_{\{x_i=0 \forall i\notin J\}}$ in $\AA_\CC^{\#(J)}$ and $M_0$ is defined in Lemma \ref{almostpp}.
\end{lem}
\begin{proof}The assertion follows immediately from \cite[Lemma 6.2]{MoiIde} and \cite[Theorem 7.4]{CDenSperlocal}.
\end{proof}

\begin{lem}\label{exinj}
Let $J$ be a non empty subset of $\{1,...,n\}$ and $s_J$ be as in Lemma \ref{exfin}. Then there is a subset $I$ of $J$ and an injection $\varsigma_J:I \to \{1,...,n\}\setminus J$ such that  $n-\#(J)\geq \#(I)\geq s_J$, $d_J\mid (d_0-w_{\varsigma_J(i)})$ and $d_0\geq w_i+w_{\varsigma_J(i)}$ for all $i\in I$.
\end{lem}

\begin{proof}We set $r=\#(J)>0$. Without loss of generality, we can suppose that $J=\{1,...,r\}$. Let $\cJ_{f_{d_0w}}=(\frac{\partial f_{d_0w}}{\partial x_1},...,\frac{\partial f_{d_0w}}{\partial x_n})$ be the Jacobian ideal of $f_{d_0w}$. Since $f_{d_0w}$ defines a hypersurface with an isolated singularity in $\AA_\CC^n$, the scheme $Z(\cJ_{f_{d_0w}})\subset \AA_{\CC}^n$ associated to $\cJ_{f_{d_0w}}$ has only one point. We set $\cI_J= \cJ_{f_{d_0w}}+ (x_{r+1},...,x_n)\subset \CC[x_1,...,x_n]$ and $\tilde{\cI}_J=(\frac{\partial f_{d_0w}}{\partial x_1},...,\frac{\partial f_{d_0w}}{\partial x_r}, x_{r+1},...,x_n)$. Let $Z(\cI_J)$ and $Z(\tilde{\cI}_J)$ be the subschemes of $\AA_\CC^n$ associated to $\cI_J, \tilde{\cI}_J$ respectively. Then $Z(\cI_J)$ has only one point and  $\dim(Z(\tilde{\cI}_J))=s_J$. For each $1\leq j\leq n$, there exists a unique $w$-weighted homogeneous polynomial $g_j\in \CC[x_1,...,x_r]$ such that $\frac{\partial f_{d_0w}}{\partial x_j}-g_j\in (x_{r+1},...,x_n)$. So we have
$$\tilde{\cI}_J=(g_1,...,g_r, x_{r+1},...,x_n), \cI_J=(g_1,...,g_r, x_{r+1},...,x_n, g_{r+1},...,g_n)=\tilde{\cI}_J+(g_{r+1},...,g_n).$$
We identify $\AA_\CC^r$  with $\spec(\CC[x_1,...,x_n]/(x_{r+1},...,x_n))\subset \AA_\CC^n$. Then $Z(\cI_J)$ and $ Z(\tilde{\cI}_J)$ are subschemes of $\AA_\CC^r$. Let $Z(g_{r+1},...,g_n)$ be the subscheme of $\AA_\CC^r$ associated to the ideal $(g_{r+1},...,g_n)$. Then we have $Z(\cI_J)=Z(\tilde{\cI}_J)\cap Z(g_{r+1},...,g_n)$. Since $\dim(Z(\cI_J))=0$ and $\dim(Z(\tilde{\cI}_J))=s_J$, we can use \cite[Proposition 7.1]{Harts} to deduce  $0\leq e=\dim(Z(g_{r+1},...,g_n))\leq r-s_J$. 

If $s_J=0$ then it is sufficient to take $I=\emptyset$. 
Suppose that $s_J>0$. Then $0\leq e=\dim(Z(g_{r+1},...,g_n))<r$ thus $g_{j_0}$ is not the zero polynomial for at least one index $r+1\leq j_0\leq n$. Let $1\leq i_0\leq r$ such that $x_{i_0}$ appears non-trivial in $g_{j_0}$. We set $I_0=\{i_0\}$ and $\varsigma_0:I_0\to \{1,...,n\}, i_0\mapsto j_0$. Then, we can apply Lemma \ref{inje} below for $Q=\{r+1,...,n\}, (g_i)_{r+1\leq j\leq n}, I_0, \varsigma_0$ to obtain a subset $I\supset I_0$ of $J=\{1,...,r\}$ and an injection $\varsigma_J:I\to\{r+1,...,n\}$ such that $\#(I)\geq r-e\geq s_J$ and $x_i$ appears non-trivial in $g_{\varsigma_J(i)}$ for all $i\in I$.  
It follows from the conditions $\frac{\partial f_{d_0w}}{\partial x_{\varsigma_J(i)}}-g_{\varsigma_J(i)}\in (x_{r+1},...,x_n)$ for all $i\in I$ and $g_j\in\CC[x_1,...,x_r]$ for all $j$ that there is a monomial $X^{(i)}=x_1^{i_1}...x_r^{i_r}$ such that the coefficient of $x_ix_{\varsigma_J(i)}X^{(i)}$ in $f_{d_0w}$ is non-zero for all $i\in J$. Thus, the $w$-weighted homogeneity of $f_{d_0w}$ implies that $d_0=w_i+w_{\varsigma_J(i)}+\sum_{j=1}^r w_ji_j$ for all $i\in J$. Therefore, $d_J\mid (d_0-w_{\varsigma_J(i)})$ and $d_0\geq w_i+w_{\varsigma_J(i)}$  for all $i\in J$. This completes the proof.
\end{proof}

\begin{lem}\label{inje}Let $Q$ be a non-empty set and $\cI$ be the ideal of $\CC[x_1,...,x_n]$ generated by non-constant polynomials $g_i$ for $i\in Q$. Suppose that the scheme $Z(\cI)$ associated to $\cI$ is non-empty and of codimension $\ell$. 
If $I_0$ is a subset of $\{1,...,n\}$ and $\varsigma_0:I_0\to Q$ be an injection such that $x_i$ appears non-trivial in $g_{\varsigma_0(i)}$ for all $i\in I_0$ then there is a set $I\supset I_0$ and an injection $\varsigma:I\to Q$ such that $\#(I)\geq \ell$,  $x_i$ appears non-trivial in $g_{\varsigma(i)}$ for all $i\in I$ and $\varsigma_0(I_0)\subset \varsigma(I)$. 
\end{lem}
\begin{proof}Suppose that the assertion holds for $\ell=n$ then we will show that it holds for all $0\leq \ell\leq n$. Indeed, let $Q, \cI, (g_i)_{i\in Q}, \ell, I_0, \varsigma_0$ be as in the statement such that $\ell<n$. It follows by induction and \cite[Propostion 7.1]{Harts} that there exist $n-\ell$ non-zero polynomials $l_1(x_1,...,x_n),...,l_{n-\ell}(x_1,...,x_n)$ of degree $1$ such that $\dim(Z(\cI)\cap \{l_1(x_1,...,x_n)=0\}\cap...\cap\{l_{n-\ell}(x_1,...,x_n)=0\})=0$. We set $\tilde{Q}=Q\sqcup\{1,...,n-\ell\}$ and $g_i=\ell_i$ if $i\in \{1,...,n-\ell\}$ then the ideal $\tilde{\cI}$ generated by $g_i$ for $i\in \tilde{Q}$ defines a subscheme of dimension $0$ of $\AA_\CC^n$. By our assumption, there is an injection $\varsigma:\{1,...,n\}\to\tilde{Q}$ such that $x_i$ appears non-trivial in $g_{\varsigma(i)}$ for all $1\leq i\leq n$ and $\varsigma_0(I_0)\subset \varsigma(\{1,...,n\})$. Let $I=\varsigma^{-
1}(Q\cap \varsigma(\{1,...,n\}))$ then $\#(I)\geq \ell$, the map $\varsigma|_I: I\to Q$ is injective  and $x_i$ appears non-trivial in $g_{\varsigma(i)}$ for all $i\in I$. Moreover, since $\gamma_0(I_0)\subset Q$,  we have $\varsigma_0(I_0)\subset \varsigma(\{1,...,n\})\cap Q=\varsigma(I)$. This shows that the assertion holds for $Q,\cI,(g_i)_{i\in Q}, \ell, I_0,\varsigma_0$  as desired.

Now, it remains to prove the assertion for $\ell=n$. Let $Q, \cI, (g_i)_{i\in Q}, \ell, I_0, \varsigma_0$ be as in the statement such that $\ell=n$. Since $\ell=n$, we must have $\#(Q)\geq n$. It is sufficient to show that if $\#(I_0)<n$ then there is a set $I\supsetneq I_0$ and an injection $\varsigma:I\to Q$ such that $x_i$ appears non-trivial in $g_{\varsigma(i)}$ for all $i\in I$ and $\varsigma_0(I_0)\subset\varsigma(I)$.  
Indeed, if there is $j\in Q\setminus \varsigma_0(I_0)$ and $i\notin I_0$ such that $x_i$ appears non-trivial in $g_j$, then we can take $\cI=\cI_0\sqcup \{i\}$ and $\varsigma:I\to Q, \varsigma|_{I_0}=\varsigma_0, \varsigma(i)=j$. Thus, we need to consider the case that $g_j\in\CC[x_i]_{i\in I_0}$ for all $j\in Q\setminus\varsigma_0(I_0)$. For this case,  let $I_1$ be the smallest subset of $I_0$ such that $g_j\in \CC[x_i]_{i\in I_1}$ for all $j\in Q\setminus\varsigma_0(I_0)$. We construct a sequence $(I_r)_{r\geq 1}$ of subsets of $\{1,...,n\}$ by letting $I_{r+1}$ to be the smallest set such that $g_j\in \CC[x_i]_{i\in I_{r+1}}$ for all $j\in \varsigma_0(I_r\cap I_0)$ and all $r\geq 1$. Since $x_i$ appears non-trivial in $g_{\varsigma(i)}$ for all $i\in I_0$,  $(I_r\cap I_0)_{r\geq 1}$ is an increasing sequence of subsets of $I_0$. This implies that $(I_r)_{r\geq 1}$ is an increasing sequence of subsets of $\{1,...,n\}$.  It cannot happen that $I_r\subset I_0$ for all $r$. Indeed, suppose otherwise that $I_r\subset I_0$ for all $r$. Since $(I_r)_{r\geq 1}$ is an increasing sequence of subsets of the finite set $I_0$,  there is $r_0\geq 1$ such that $I_r=I_{r+1}$ for all $r\geq r_0$. It follows from the definition of the sets $I_r$ for $r\geq 1$ that  $g_j\in \CC[x_i]_{i\in I_{r_0}}$ for all $j\in Q\setminus\varsigma_0(I_0\setminus I_{r_0})$. Let $\tilde{\cI}$ be the ideal of $\CC[x_1,...,x_n]$ generated by $g_j$ for $j\in Q\setminus\varsigma_0(I_0\setminus I_{r_0})$ and $Z(\tilde{\cI})$ be the closed subscheme of $\AA_\CC^n$ associated to $\tilde{\cI}$. Then the condition $g_j\in \CC[x_i]_{i\in I_{r_0}}$ for all $j\in Q\setminus\varsigma_0(I_0\setminus I_{r_0})$ implies that $\dim(Z(\tilde{\cI}))\geq n-\#(I_{r_0})$ if $Z(\tilde{\cI})\neq\emptyset$. On the other hand, we have $Z(\cI)=Z(\tilde{\cI})\cap (\cap_{j\in \varsigma_0(I_0\setminus I_{r_0})} \{g_{j}=0\})$. Since $\dim(Z(\cI))=0$, we  have $Z(\tilde{\cI}) \neq\emptyset$. Thus, we can use  \cite[Proposition 7.1]{Harts} again to have $$0=\dim(Z(\cI))\geq \dim(Z(\tilde{\cI}))-\#(I_0\setminus I_{r_0})\geq n-\#(I_{r_0})-\#(I_0\setminus I_{r_0}).$$
Thus $\#(I_0)\geq n$ and we have a contradiction. Thus, there is $r>1$ such that $I_r\nsubseteq I_0$. Let $r_0$ be the smallest index with this property. Then we have $I_{r_0-1}\subset I_0$ and $I_{r_0}\nsubseteq I_0$. Let $i_0\in I_{r_0}\setminus I_0$. By the construction of the sets $I_r$ for $2\leq r\leq r_0$, there are indexes $j_1,...,j_{r_0-1}$ such that $j_r\in I_r$ for all $1\leq r\leq r_0-1$, $x_{i_0}$ appears non-trivial in $g_{\varsigma_0(j_{r_0-1})}$ and  $x_{j_m}$ appears non-trivial in $g_{\varsigma_0(j_{m-1})}$ for all $1<m\leq r_0-1$. By the definition of $I_1$, there is $j'\in Q\setminus\varsigma_0(I_0)$ such that $x_{j_1}$ appears non-trivial in $g_{j'}$. We set $I=I_0\sqcup\{i_0\}$ and consider the map $\varsigma:I\to Q$ given by $\varsigma(i)=\varsigma_0(i)$ if $i\in I_0\setminus\{j_1,...,j_{r_0-1}\}$, $\varsigma(j_m)=\varsigma_0(j_{m-1})$ for $1<m\leq r_0-1$, $\varsigma(i_0)=\varsigma_0(j_{r_0-1})$ and $\varsigma(j_1)=j'$. It is easy to verify that $\varsigma$ is an injection, $x_i$ appears non-trivial in $g_{\varsigma(i)}$ for all $i\in I$ and $\varsigma_0(I_0)\subset \varsigma(I)$ as desired.
\end{proof}
\begin{lem}\label{ressmooth}Let $l$ be a positive integer. Let $J$ be the subset of $\{1,...,n\}$ consisting of indexes $i$ such that $\gcd(l,d_0)\mid w_i$. Suppose that $J\neq \emptyset$. Recall that $F_J=f_{d_0w}|_{\{x_i=0\forall i\notin J\}}$. 
Then $F_J$ defines a hypersurface with an isolated singularity in the affine subspace $H=\{x_i=0\hspace{0.1cm} \forall i\notin J\}$ of $\AA_\CC^n$.  
\end{lem}
\begin{proof} By rearranging the coordinates, we can suppose that $J=\{1,...,r\}$. By the definition of $J$,  we can write $$f_{d_0w}(x_1,...,x_n)=g(x_1,...,x_r)+h(x_1,...,x_n)$$ for a polynomial $h\in (x_{r+1},...,x_n)^2$. Let $x^{r,(0)}=(x_1^{(0)},...,x_r^{(0)})$ such that $\frac{\partial g}{\partial x_i}(x^{(0)})=0$ for all $1\leq i\leq r$. We set $x^{(0)}=(x^{r,(0)},0,...,0)\in\AA_\CC^n$. Since $h\in (x_{r+1},...,x_n)^2$, it is easy to verify that $x^{(0)}$ is a singularity of $\{f_{d_0w}=0\}$. Thus $x_i^{(0)}=0$ for all $1\leq i\leq r$. This implies the assertion.
\end{proof}

\begin{lem}\label{1/2}Let $w'=(w'_1,...,w'_r)\in \NN_{>0}^r$ and  $g$ be a $w'$-weighted homogeneous polynomial of $w'$-degree $d_0$ such that $g$ defines a hypersurface with an isolated singularity in $\AA_\CC^n$. Then we have 
$\frac{d_0r}{2}\geq |w'|=w'_1+...+w'_r.$
\end{lem}
\begin{proof}If $r=1$ then $f=x^\ell$ for some $\ell\geq 2$ and $d_0=\ell w'_1$ if the weight of $x$ is $w'_1$, thus $\frac{1}{2}\geq \frac{1}{\ell}=\frac{w'_1}{d_0}$. If $r\geq 2$,  we will show that $\frac{r}{2}\geq \frac{\sum_{i=1}^r w'_i}{d_0}$. Let $J=\{1\leq i\leq r\mid w'_i>\frac{d_0}{2}\}$. If $J=\emptyset$ then it is clear that $\frac{r}{2}\geq \frac{\sum_{i=1}^r w'_i}{d_0}$. If $J\neq\emptyset$, by rearranging the order of coordinates, we can suppose that $J=\{1,...,r_0\}$ and $g=h(x_{r_0+1},...,x_r)+\sum_{i=1}^{r_0} x_ig_i(x_{r_0+1},...,x_r)$. Let $A=(\frac{\partial g_i}{\partial x_j}(0,...,0))_{r_0+1\leq j\leq r, 1\leq i\leq r_0}$ be an $(r-r_0)\times r_0$ matrix. If $\rank(A)<r_0$ then there exists $x^{r_0,(0)}=(x_1^{(0)},...,x_{r_0}^{(0)})\in \CC^{r_0}\setminus\{0,...,0\}$ such that $x^{r_0,(0)}A^T=0$. It is easy to verify that $x^{(0)}:=(x^{r_0,(0)},0,...,0)\neq (0,...,0)$ is a critical point of $g$ in $\AA_\CC^r$. We get a contradiction. Thus $\rank(A)=r_0$. Then there are distinct indexes $r_0+1\leq i_1,...,i_{r_0}\leq r$ such that $x_{i_j}$ appears as a non-trivial linear term in $g_{j}(x_{r_0+1},...,x_{r})$ for all $1\leq j\leq r_0$. Then $w'_j+w'_{i_j}=d_0$ for all $1\leq j\leq r_0$. Thus, we have 
\begin{align*}
\frac{rd_0}{2}&=\left(\sum_{1\leq j\leq r_0}d_0\right)+ \left(\sum_{i>r_0, i\neq i_j\forall 1\leq j\leq r_0}\frac{d_0}{2}\right)\\&=\left(\sum_{1\leq j\leq r_0}(w'_j+w'_{i_j})\right)+\left(\sum_{i>r_0, i\neq i_j\forall 1\leq j\leq r_0}\frac{d_0}{2}\right)\\&\geq \left(\sum_{1\leq j\leq r_0}(w'_j+w'_{i_j})\right)+\left(\sum_{i>r_0, i\neq i_j\forall 1\leq j\leq r_0}w'_i\right)\\&=\sum_{1\leq i\leq r}w'_i.
\end{align*}
\end{proof}

\begin{lem}\label{ubound}Let $J$ be a non empty subset of $\{1,...,n\}$ and $s_J$ be as in Lemma \ref{exfin}. Suppose that $d_0\geq\kappa(w)$. 

If $F_J\neq \emptyset$ then 
$$\epsilon_{J}(\upsilon)+\frac{|w|}{d_0}-n+\frac{\#(J)+s_{J}}{2}\leq 0$$
for all $0\leq\upsilon\leq d_J-1$.

If $F_J=0$ then 
$$\epsilon_{J}(\upsilon)+\left(\frac{\upsilon}{d_J}+\frac{1}{d_0}\left(1-\delta_{J\upsilon}-\bigg\lfloor\frac{d_0\upsilon}{d_J}\bigg\rfloor\right)\right)|w|-n+\#(J)\leq 0$$
for all $0\leq\upsilon\leq d_J-1$.
\end{lem}
\begin{proof}
Suppose that $F_J\neq 0$. By Lemma \ref{etaJ}, we have $d_J\mid d_0$. Let $J'$ be the subset of $\{1,...,n\}$ consisting of indexes $i$ such that $d_J\mid w_i$. It follows from Lemma \ref{ressmooth} that $f_{d_0w}|_{\{x_i=0\hspace{0.1cm} \forall i\notin J'\}}$ defines a hypersurface with an isolated singularity in the affine subspace $\{x_i=0\hspace{0.1cm} \forall i\notin J'\}$. By Lemma \ref{1/2}, we have 
\begin{equation}\label{inew1}
\frac{d_0\#(J')}{2}\geq \sum_{i\in J'} w_i.
\end{equation}
By Lemma \ref{exinj}, we must have $\#(J')\geq \#(J)+s_J$ thus 
\begin{equation}\label{inew2}
\#(J')\leq 2\#(J')-\#(J)-s_J
\end{equation}
For each $0\leq \upsilon\leq d_J-1$, we set $J'_\upsilon:=\{1\leq i\leq n\mid \frac{\upsilon w_i}{d_J}\in\ZZ\}$. Then we have $J'\subset J'_\upsilon$ for all $0\leq\upsilon\leq d_J-1$. Let $i\notin J'_\upsilon$ then $i\notin J'$. Since $d_0\geq\kappa(w)$ we have $d_0\geq\lcm(d_J,w_i)$. Thus, there are $a,b\in \NN_{\geq 1}$ such that $ad_J=bw_i\leq d_0$. Since $d_J\nmid \upsilon w_i$ we have 
$$0<\frac{\upsilon w_i}{d_J}-\bigg\lfloor\frac{\upsilon w_i}{d_J}\bigg\rfloor= \frac{\upsilon a}{b}-\bigg\lfloor\frac{\upsilon a}{b}\bigg\rfloor\geq \frac{1}{b}\geq \frac{w_i}{d_0}.$$
Therefore, we have
\begin{equation}\label{inew3}
d_0\left(\frac{\upsilon w_i}{d_J}-\bigg\lfloor\frac{\upsilon w_i}{d_J}\bigg\rfloor\right)\geq w_i
\end{equation}
if $i\notin J'_\upsilon$ and $0\leq \upsilon\leq d_J-1$. This implies
\begin{equation}\label{key1}
\sum_{i\notin J_\upsilon'}\frac{w_i}{d_0}\leq \sum_{i\notin J_\upsilon'}\left(\frac{\upsilon w_i}{d_J}-\bigg\lfloor\frac{\upsilon w_i}{d_J}\bigg\rfloor\right)
\end{equation}
for all $0\leq \upsilon\leq d_J-1$.
By combining (\ref{inew1}), (\ref{inew2}, (\ref{key1}), for each $0\leq \upsilon\leq d_J-1$, we have
\begin{align*}
&\epsilon_{J}(\upsilon)+\frac{|w|}{d_0}-n+\frac{\#(J)+s_{J}}{2}\\=&\left(\sum_{i\notin J'_\upsilon}\left(1-\frac{\upsilon w_i}{d_J}+\bigg\lfloor \frac{\upsilon w_i}{d_J}\bigg\rfloor\right)\right)-\left( n-\frac{\#(J)+s_{J}}{2}\right)+\frac{1}{d_0}\sum_{1\leq i\leq n} w_i\\
=&\left(\sum_{i\notin J'_\upsilon}\left(-\frac{\upsilon w_i}{d_J}+\bigg\lfloor \frac{\upsilon w_i}{d_J}\bigg\rfloor\right)\right)-\left(\#(J'_\upsilon)-\frac{\#(J)+s_{J}}{2}\right)+\frac{1}{d_0}\sum_{1\leq i\leq n} w_i\\
=& \left(\sum_{i\notin J'_\upsilon}\left(-\frac{\upsilon w_i}{d_J}+\bigg\lfloor \frac{\upsilon w_i}{d_J}\bigg\rfloor\right)\right)-\left(\#(J')-\frac{\#(J)+s_{J}}{2}\right)-\left(\#(J'_\upsilon)-\#(J')\right)+\frac{1}{d_0}\sum_{1\leq i\leq n} w_i\\
\leq&\left(\sum_{i\notin J'_\upsilon}\left(-\frac{\upsilon w_i}{d_J}+\bigg\lfloor \frac{\upsilon w_i}{d_J}\bigg\rfloor\right)\right)-\frac{\#(J')}{2}-\left(\#(J'_\upsilon)-\#(J')\right)+\frac{1}{d_0}\sum_{1\leq i\leq n} w_i\\
\leq &\frac{1}{d_0}\left(-\left(\sum_{i\notin J'_\upsilon}w_i\right)-\left(\sum_{i\in J'} w_i\right)-\left(\sum_{i\in J'_\upsilon\setminus J'}d_0\right)+\left(\sum_{1\leq i\leq n}w_i\right)\right)\\
=& \frac{1}{d_0}\sum_{i\in J'_\upsilon\setminus J'}(w_i-d_0)\\
\leq& 0.
\end{align*}

Now, we suppose that $F_J=0$. Then $s_J=\#(J)$. By using Lemma \ref{exinj}, there is an injection $\varsigma_J:J\to \{1,...,n\}\setminus J$ such that $d_J\mid (d_0-w_{\varsigma_J(i)})$ and $d_0\geq w_i+w_{\varsigma_J(i)}$ for all $i\in J$. Let $0\leq \upsilon\leq d_J-1$.  We set $J'_\upsilon=\{1\leq i\leq n\mid \frac{\upsilon w_i}{d_J}\in \ZZ\}$ and $J'=J\cup \varsigma_J(J)$. Then $\#(J')=2\#(J)$ and $J\subset J'_\upsilon$. Suppose that $d_J\mid d_0\upsilon$. Then we have $J'\subset J'_{\upsilon}$. Note that if $i\notin J'_\upsilon$ then (\ref{inew3}) still holds for $J,i,\upsilon$. Therefore, (\ref{key1}) is still valid for the current case. Thus, we have
\begin{align*}
&\epsilon_{J}(\upsilon)+\left(\frac{\upsilon}{d_J}+\frac{1}{d_0}\left(1-\delta_{J\upsilon}-\bigg\lfloor\frac{d_0\upsilon}{d_J}\bigg\rfloor\right)\right)|w|-n+\#(J)\\
=&\left(\sum_{i\notin J'_\upsilon}\left(1-\frac{\upsilon w_i}{d_J}+\bigg\lfloor \frac{\upsilon w_i}{d_J}\bigg\rfloor\right)\right)+\frac{|w|}{d_0}-\left(n-\#(J)\right)\\
=&\frac{1}{d_0}d_0\left(\left(\sum_{i\notin J'_\upsilon}\left(-\frac{\upsilon w_i}{d_J}+\bigg\lfloor \frac{\upsilon w_i}{d_J}\bigg\rfloor\right)\right)-\left(\#(J'_\upsilon)-\#(J')\right)-\left(\#(J')-\#(J)\right)+\frac{1}{d_0}\sum_{1\leq i\leq n} w_i\right)\\
\end{align*}
\begin{align*}
=&\frac{1}{d_0}\left(\left(\sum_{i\notin J'_\upsilon}d_0\left(-\frac{\upsilon w_i}{d_J}+\bigg\lfloor \frac{\upsilon w_i}{d_J}\bigg\rfloor\right)\right)-\left(\sum_{i\in J'_\upsilon\setminus J'}d_0\right)-\left(\sum_{i\in J}d_0\right)+\left(\sum_{1\leq i\leq n} w_i\right)\right)\\
\leq &\frac{1}{d_0}\left(-\left(\sum_{i\notin J'_\upsilon}w_i\right)-\left(\sum_{i\in J'_\upsilon\setminus J'}d_0\right)-\left(\sum_{i\in J} \left(w_i+w_{\varsigma_J(i)}\right)\right)+\left(\sum_{1\leq i\leq n}w_i\right)\right)\\
=& \frac{1}{d_0}\sum_{i\in J'_\upsilon\setminus J'}(w_i-d_0)\\
\leq& 0.
\end{align*}
Suppose that $d_J\nmid d_0\upsilon$ then we have $\varsigma_J(J)\cap J'_\upsilon=\emptyset$ and
\begin{equation}\label{inew4}
0<\frac{\upsilon}{d_J}+\frac{1}{d_0}\left(1-\delta_{J\upsilon}-\bigg\lfloor\frac{d_0\upsilon}{d_J}\bigg\rfloor\right)=\frac{1}{d_0}\left(\frac{d_0\upsilon}{d_J}-\bigg\lfloor \frac{d_0\upsilon}{d_J}\bigg\rfloor\right)< \frac{1}{d_0}.
\end{equation}
Let $i\notin J'_\upsilon$ then $d_J\nmid \upsilon w_i$. We set $\gcd(w_i,d_J)=a_i, \gcd(d_J,d_0)=b_i$. Then we can write $w_i=a_ic_i, d_J=a_im_i=b_ie_i, d_0=b_iu_i$ for positive integers $c_i,m_i,e_i,u_i$ such that $\gcd(c_i,m_i)=\gcd(e_i,u_i)=1$. Then we have $\frac{\upsilon w_i}{d_J}=\frac{\upsilon c_i}{m_i}$, $\frac{w_i}{d_0}=\frac{a_ic_i}{b_iu_i}$, $\frac{\upsilon d_0}{d_J}=\frac{\upsilon u_i}{e_i}$. Since $d_J\nmid \upsilon d_0$ and $d_J\nmid \upsilon w_i$, we have
$$\frac{\upsilon w_i}{d_J}-\bigg\lfloor\frac{\upsilon w_i}{d_J}\bigg\rfloor= \frac{\upsilon c_i}{m_i}-\bigg\lfloor\frac{\upsilon c_i}{m_i}\bigg\rfloor\geq \frac{1}{m_i}$$
and 
$$\frac{\upsilon d_0}{d_J}-\bigg\lfloor\frac{\upsilon d_0}{d_J}\bigg\rfloor=\frac{\upsilon u_i}{e_i}-\bigg\lfloor\frac{\upsilon u_i}{e_i}\bigg\rfloor\leq \frac{e_i-1}{e_i}.$$
On the other hand, we have $$\lcm(w_i,d_J)=\frac{w_id_J}{\gcd(w_i,d_J)}=\frac{a_ic_id_J}{a_i}=c_id_J.$$ 
By the condition $d_0\geq \kappa(w)$, we have $d_0\geq\lcm(w_i,d_J)=c_id_J$. Therefore, we have 
\begin{equation}\label{inew5}
\frac{w_i}{d_0}\left(\frac{\upsilon d_0}{d_J}-\bigg\lfloor\frac{\upsilon d_0}{d_J}\bigg\rfloor\right)\leq \frac{a_ic_i(e_i-1)}{b_iu_ie_i}=\frac{a_ic_i(e_i-1)}{u_ia_im_i}=\frac{c_i(e_i-1)b_i}{u_ib_im_i}<\frac{c_id_J}{d_0m_i}\leq\frac{1}{m_i}\leq \frac{\upsilon w_i}{d_J}-\bigg\lfloor\frac{\upsilon w_i}{d_J}\bigg\rfloor
\end{equation}
for all $i\notin J'_\upsilon$. Hence,
\begin{equation}\label{key2}\sum_{i\notin J'_\upsilon\cup\varsigma_J(J)}^{d_J\nmid \upsilon(d_0-w_i)}\left(\frac{w_i}{d_0}\left(\frac{d_0\upsilon}{d_J}-\bigg\lfloor \frac{d_0\upsilon}{d_J}\bigg\rfloor\right)-\frac{\upsilon w_i}{d_J}+\bigg\lfloor \frac{\upsilon w_i}{d_J}\bigg\rfloor\right)\leq 0.
\end{equation}
If  $d_J\mid \left(\upsilon d_0-\upsilon w_{i}\right)$, we have 
\begin{equation}\label{inew6}
\frac{d_0\upsilon}{d_J}-\bigg\lfloor \frac{d_0\upsilon}{d_J}\bigg\rfloor=\frac{w_{i}\upsilon}{d_J}-\bigg\lfloor \frac{w_{i}\upsilon}{d_J}\bigg\rfloor.
\end{equation}
By (\ref{inew4}), (\ref{key2}), (\ref{inew6}) and the conditions $d_0\geq w_i+w_{\varsigma_J(i)}$, $d_J\mid (d_0-w_{\varsigma_J(i)})$ for all $i\in J$, we have 
\begin{align*}
&\epsilon_{J}(\upsilon)+\left(\frac{\upsilon}{d_J}+\frac{1}{d_0}\left(1-\delta_{J\upsilon}-\bigg\lfloor\frac{d_0\upsilon}{d_J}\bigg\rfloor\right)\right)|w|-n+\#(J)\\
=&\left(\sum_{i\notin J'_\upsilon}\left(1-\frac{\upsilon w_i}{d_J}+\bigg\lfloor \frac{\upsilon w_i}{d_J}\bigg\rfloor\right)\right)+ \frac{1}{d_0}\left(\frac{d_0\upsilon}{d_J}-\bigg\lfloor \frac{d_0\upsilon}{d_J}\bigg\rfloor\right)|w|-n+\#(J)\\
=&\left(\sum_{i\notin J'_\upsilon}\left(\frac{w_i}{d_0}\left(\frac{d_0\upsilon}{d_J}-\bigg\lfloor \frac{d_0\upsilon}{d_J}\bigg\rfloor\right)-\frac{\upsilon w_i}{d_J}+\bigg\lfloor \frac{\upsilon w_i}{d_J}\bigg\rfloor\right)\right)+ \left(\sum_{i\in J'_\upsilon}\frac{w_i}{d_0}\left(\frac{d_0\upsilon}{d_J}-\bigg\lfloor \frac{d_0\upsilon}{d_J}\bigg\rfloor\right)\right)-\#(J'_\upsilon)+\#(J)\\
=&\left(\sum_{i\notin J'_\upsilon\cup\varsigma_J(J)}\left(\frac{w_i}{d_0}\left(\frac{d_0\upsilon}{d_J}-\bigg\lfloor \frac{d_0\upsilon}{d_J}\bigg\rfloor\right)-\frac{\upsilon w_i}{d_J}+\bigg\lfloor \frac{\upsilon w_i}{d_J}\bigg\rfloor\right)\right)+\left(\sum_{i\in J'_\upsilon\setminus J}\frac{w_i}{d_0}\left(\frac{d_0\upsilon}{d_J}-\bigg\lfloor \frac{d_0\upsilon}{d_J}\bigg\rfloor\right)\right)\\
\end{align*}
\begin{align*}
&-\#(J'_\upsilon)+\#(J)+\left(\sum_{i\in J}\left(\frac{w_i+w_{\varsigma(i)}}{d_0}\left(\frac{d_0\upsilon}{d_J}-\bigg\lfloor \frac{d_0\upsilon}{d_J}\bigg\rfloor\right)-\frac{\upsilon w_{\varsigma(i)}}{d_J}+\bigg\lfloor \frac{\upsilon w_{\varsigma(i)}}{d_J}\bigg\rfloor\right)\right)\\
=&\left(\sum_{i\notin J'_\upsilon\cup\varsigma_J(J)}\left(\frac{w_i}{d_0}\left(\frac{d_0\upsilon}{d_J}-\bigg\lfloor \frac{d_0\upsilon}{d_J}\bigg\rfloor\right)-\frac{\upsilon w_i}{d_J}+\bigg\lfloor \frac{\upsilon w_i}{d_J}\bigg\rfloor\right)\right)+\left(\sum_{i\in J'_\upsilon\setminus J}\left(\frac{w_i}{d_0}\left(\frac{d_0\upsilon}{d_J}-\bigg\lfloor \frac{d_0\upsilon}{d_J}\bigg\rfloor\right)-1\right)\right)\\&+\left(\frac{d_0\upsilon}{d_J}-\bigg\lfloor \frac{d_0\upsilon}{d_J}\bigg\rfloor\right)\sum_{i\in J}\left(\frac{w_i+w_{\varsigma(i)}}{d_0}-1\right)\\
=&\left(\sum_{i\notin J'_\upsilon\cup\varsigma_J(J)}^{d_J\nmid \upsilon(d_0-w_i)}\left(\frac{w_i}{d_0}\left(\frac{d_0\upsilon}{d_J}-\bigg\lfloor \frac{d_0\upsilon}{d_J}\bigg\rfloor\right)-\frac{\upsilon w_i}{d_J}+\bigg\lfloor \frac{\upsilon w_i}{d_J}\bigg\rfloor\right)\right)+\left(\sum_{i\in J'_\upsilon\setminus J}\left(\frac{w_i}{d_0}\left(\frac{d_0\upsilon}{d_J}-\bigg\lfloor \frac{d_0\upsilon}{d_J}\bigg\rfloor\right)-1\right)\right)\\&+\left(\sum_{i\notin J'_\upsilon\cup\varsigma_J(J)}^{d_J\mid \upsilon(d_0-w_i)}\left(\frac{w_i}{d_0}\left(\frac{d_0\upsilon}{d_J}-\bigg\lfloor \frac{d_0\upsilon}{d_J}\bigg\rfloor\right)-\frac{\upsilon w_i}{d_J}+\bigg\lfloor \frac{\upsilon w_i}{d_J}\bigg\rfloor\right)\right)+\left(\frac{d_0\upsilon}{d_J}-\bigg\lfloor \frac{d_0\upsilon}{d_J}\bigg\rfloor\right)\sum_{i\in J}\left(\frac{w_i+w_{\varsigma(i)}}{d_0}-1\right)\\
<& \left(\sum_{i\in J'_\upsilon\setminus J}\left(\frac{w_i}{d_0}-1\right)\right)+ \left(\frac{d_0\upsilon}{d_J}-\bigg\lfloor \frac{d_0\upsilon}{d_J}\bigg\rfloor\right)\left(\left(\sum_{i\notin J'_\upsilon\cup\varsigma_J(J)}^{d_J\mid \upsilon(d_0-w_i)}\left(\frac{w_i}{d_0}-1\right)\right)+ \left( \sum_{i\in J}\left(\frac{w_i+w_{\varsigma(i)}}{d_0}-1\right)\right)\right)\\
\leq& 0.
\end{align*}
This completes the proof.
\end{proof}

\subsection{Proof of Main Theorem 2}
We will prove the following  theorem that generalize and strengthen  Main Theorem \ref{mthm2}.
\begin{thm}\label{expthm} Let $K,w,d_0,d,f,\kappa(w)$ be as above.
Suppose that $d_0\geq \kappa(w)$. Let $M_1$ be as in Lemma \ref{exfin}. Let $L$ be a a local field over $\cO_K$ of residue field characteristic at least $M_1$. Then, there is a constant $c$ depending only on $n,w,d_0$ such that
\begin{enumerate}
\item The inequality 
$$|E_{f,\textbf{1}_{\varpi_L\cO_L^n},\psi}|\leq cq_L^{-\frac{|w|m_\psi}{d_0}}$$
 holds for all non-trivial additive characters $\psi$ of $L$.
\item The inequality $$|E_{f_{d_0w},\Phi_{L,Y},\psi}|\leq cq_L^{-\frac{|w|m_\psi}{d_0}}$$
 holds for all subschemes $Y\subset \AA_{\cO_K}^n$ and all  additive characters  $\psi$ of $L$ of conductor at least $2$.
Moreover, when $Y=\AA_{\cO_K}^n$, this inequality holds for all non-trivial additive characters $\psi$ of $L$. 
\end{enumerate}
\end{thm}
\begin{proof} 
Let $M_1$ be as in Lemma \ref{exfin} and $L$ be a local field over $\cO_K$ of residue field characteristic at least $M_1$. By Lemma \ref{1/2}, we have 
\begin{equation}\label{1/22}
nd_0>\frac{nd_0}{2}\geq \sum_{1\leq i\leq n}w_i=|w|.
\end{equation}
Thus, for all non-trivial additive characters $\psi$ of $L$ of conductor $m_\psi\leq 1$, we have 
$$|E_{f,\textbf{1}_{\varpi_L\cO_L^n},\psi}|\leq q_L^{-n}\leq q_L^{-m_\psi\frac{|w|}{d_0}}.$$
Moreover, it follows from Proposition \ref{zetaex} and Theorems \ref{explicit}, \ref{explicit1} that
$$E_{f,\textbf{1}_{\varpi_L\cO_L^n},\psi}=E_{f_{d_0w},\textbf{1}_{\varpi_L\cO_L^n},\psi}$$
for all additive characters $\psi$ of $L$ of conductor at least $1$. Therefore, 
it remains to work with the case that $f=f_{d_0w}$.

Now, suppose that $f=f_{d_0w}$ then the origin $(0,...,0)$ is the only critical point of $f$ in $\AA_{\CC}^n$. Let $Y$ be a subscheme of $\AA_{\cO_K}^n$.  It follows from \cite[Remark 4.5.3]{DenefBour} and the definition of $M_1$ that if  $\varpi_L\cO_L^n\nsubseteq\Supp(\Phi_{L,Y})$ then $E_{f,\Phi_{L,Y},\psi}=0$ for all additive characters $\psi$ of $L$ of conductor at least $2$. Moreover, if $\varpi_L\cO_L^n\subset\Supp(\Phi_{L,Y})$ then
\begin{equation}\label{reduto1}E_{f,\Phi_{L,\AA_{\cO_K}^n},\psi}=E_{f,\textbf{1}_{\varpi_L\cO_L^n},\psi}=E_{f,\Phi_{L,Y},\psi}
\end{equation}
for all additive characters $\psi$ of $L$ of conductors $m_\psi\geq 2$. Thus, we only need to prove the theorem for $Y=\AA_{\cO_K}^n$.

Suppose that $Y=\AA_{\cO_K}^n$ then $\Phi_{L,Y}=\textbf{1}_{\cO_L^n}$. It is obvious that 
\begin{equation}\label{m=01}|E_{f,\textbf{1}_{\cO_L^n},\psi}|\leq 1\leq q_L^{-\frac{m_\psi|w|}{d_0}}
\end{equation}
for all non-trivial additive characters $\psi$ of $L$ of conductor $m_{\psi}\leq 0$. By the definition of $M_1$, the origin $(0,...,0)$ is the only critical point of $\overline{f}:=f\mod \varpi_L\cO_L\in k_L[x_1,...,x_n]$ in $\AA_{\overline{k}_L}^n$. If  $\psi$ is a character of $L$ of conductors $m_\psi=1$ then $\Psi( x \mod \varpi_L\cO_L):=\psi(x)$ is a non-trivial character of $k_L$ and 
$$E_{f,\textbf{1}_{\cO_L^n},\psi}=q_L^{-n}\sum_{x\in k_L^n}\Psi(\overline{f}(x)).$$
Let $c_2$ be as in Lemma \ref{exfin},  then $c_2$ depends only on $n,w,d_0$ and 
$$q_L^{-n}\bigg|\sum_{x\in k_L^n}\Psi(\overline{f}(x))\bigg|\leq c_2q_L^{\frac{-n}{2}}.$$
This and (\ref{1/22}) imply 
\begin{equation}\label{m=11}\bigg|E_{f,\textbf{1}_{\cO_L^n},\psi}\bigg|\leq c_2q_L^{-\frac{|w|}{d_0}}
\end{equation}
for all additive characters $\psi$ of $L$ of conductor $1$. By  (\ref{m=01}) and (\ref{m=11}), to prove the theorem for $f=f_{d_0w}$ and $Y=\AA_{\cO_K}^n$, it remains to prove Conjecture \ref{conjex} for $f_{d_0w}$ and $Y=\AA_{\cO_K}^n$. Therefore, we can combine (\ref{reduto1}), Corollary \ref{estm5}, Lemmas \ref{exfin}  and \ref{ubound} to complete our proof. 
\end{proof}

\begin{conj}\label{conjw}Let $w=(w_1,...,w_n)\in \NN_{\geq 1}^n$ and $g\in\CC[x_1,...,x_n]$ be a $w$-weighted homogeneous polynomial of $w$-degree $d_0\geq 1$. Suppose that $g$ defines a hypersurface with an isolated singularity in $\AA_{\CC}^n$. Then we have 
$$\sum_{1\leq i\leq n, e\nmid w_i, e\nmid d_0-w_i}\left(\frac{w_i}{e}-\bigg\lfloor\frac{w_i}{e}\bigg\rfloor\right)\geq \sum_{1\leq i\leq n, e\nmid w_i, e\nmid d_0-w_i}\left(\frac{d_0}{e}+\bigg\lfloor 1-\frac{d_0}{e}\bigg\rfloor\right)\frac{w_i}{d_0}$$
for all positive integers $e$.
\end{conj}

\begin{prop}Let $w=(w_1,...,w_n)\in \NN_{\geq 1}^n$ such that $\gcd(w_1,...,w_n)=1$ and $g\in\cO_K[x_1,...,x_n]$ be a $w$-weighted homogeneous polynomial of degree $d_0\geq 1$. Suppose that $g$ defines a hypersurface with an isolated singularity in $\AA_\CC^n$. If Conjecture \ref{conjw} holds for $g$ then Conjecture \ref{conjex} holds for $g$. 
\end{prop}

\begin{proof}
As seen in the proof of Theorem \ref{expthm},  Conjecture \ref{conjex} for $g$ is proved  if we can show an unconditional version of Lemma \ref{ubound}. Moreover, the condition $d_0\geq \kappa(w)$ in Lemma \ref{ubound} is only used to show Inequality (\ref{key1}) (when $d_J\mid d_0\upsilon$) and Inequality (\ref{key2}) (when $d_J\nmid d_0\upsilon$). In fact,  Inequalities (\ref{key1}) and (\ref{key2}) are consequences of Conjecture \ref{conjw}. To see this,  let $J$ be a non-empty subset of $\{1,...,n\}$ and $0\leq \upsilon \leq d_J-1$. If $d_J\mid \upsilon d_0$, we use Conjecture \ref{conjw} for $g$, the weight $\upsilon w=(\upsilon w_1,...,\upsilon w_n)$, $e=d_J$ to have 
$$\sum_{i, d_J\nmid \upsilon w_i}\left(\frac{\upsilon w_i}{d_J}-\bigg\lfloor\frac{\upsilon w_i}{d_J}\bigg\rfloor\right)\geq \sum_{i, d_J\nmid \upsilon w_i}\frac{w_i}{d_0}.$$
If $d_J\nmid \upsilon d_0$, we use Conjecture \ref{conjw} for $g$, the weight $\upsilon w=(\upsilon w_1,...,\upsilon w_n)$, $e=d_J$ to have 
$$\sum_{i, d_J\nmid w_i, d_J\nmid d-w_i}\left(\frac{\upsilon w_i}{d_J}-\bigg\lfloor\frac{\upsilon w_i}{d_J}\bigg\rfloor\right)\geq\sum_{i, d_J\nmid w_i, d_J\nmid d_0-w_i} \left(\frac{\upsilon d_0}{d_J}-\bigg\lfloor\frac{\upsilon d_0}{d_J}\bigg\rfloor\right)\frac{w_i}{d_0}.$$
On the other hand, with the notations in the proof of Lemma \ref{ubound}, we have 
$$\{i\in \{1,...,n\}\mid d_J \nmid \upsilon w_i, d_J\nmid \upsilon (d_0-w_i)\}= \{ i\notin J'_{\upsilon}\cup \varsigma_J(J)\mid  d_J\nmid \upsilon(d_0-w_i)\}.$$ Thus, if Conjecture \ref{conjw} holds for $g$ then Inequalities (\ref{key1}) and (\ref{key2}) are valid. Therefore, Conjecture \ref{conjex} for $g$ is a consequence of Conjecture \ref{conjex} for $g$ as explained above. 
\end{proof}

\bibliographystyle{amsplain}
\bibliography{anbib}

\def\cprime{$'$}
\providecommand{\bysame}{\leavevmode\hbox to3em{\hrulefill}\thinspace}
\providecommand{\MR}{\relax\ifhmode\unskip\space\fi MR }
\providecommand{\MRhref}[2]{%
  \href{http://www.ams.org/mathscinet-getitem?mr=#1}{#2}
}
\providecommand{\href}[2]{#2}
\begin{thebibliography}{10}

\bibitem{ATW}
D.~Abramovich, M.~Temkin, and J.~W\l~odarczyk, \emph{Functorial embedded
  resolution via weighted blowings up}, Algebra Number Theory \textbf{18}
  (2024), no.~8, 1557--1587. \MR{4856603}

\bibitem{ArMO}
E.~Artal~Bartolo, J.~Mart\'in-Morales, and J.~Ortigas-Galindo,
  \emph{Intersection theory on abelian-quotient {$V$}-surfaces and {$\bf
  Q$}-resolutions}, J. Singul. \textbf{8} (2014), 11--30. \MR{3193225}

\bibitem{BBV}
G.~Blanco, N.~Budur, and R.~van~der Veer, \emph{Monodromy conjecture for
  semi-quasihomogeneous hypersurfaces}, Math. Nachr. \textbf{296} (2023),
  no.~4, 1394--1403. \MR{4604543}

\bibitem{Bosha}
A.~I. Borevich and I.~R. Shafarevich, \emph{Number theory}, Pure and Applied
  Mathematics, Vol. 20, Academic Press, New York-London, 1966, Translated from
  the Russian by Newcomb Greenleaf. \MR{0195803}

\bibitem{CAN}
W.~Castryck and K.~H. Nguyen, \emph{New bounds for exponential sums with a
  non-degenerate phase polynomial}, Journal de math\'ematiques pures et
  appliqu\'ees (2019), https://doi.org/10.1016/j.matpur.2019.01.015.

\bibitem{CDenSper}
R.~Cluckers, \emph{Igusa and {D}enef-{S}perber conjectures on nondegenerate
  {$p$}-adic exponential sums}, Duke Math. J. \textbf{141} (2008), no.~1,
  205--216.

\bibitem{CDenSperlocal}
\bysame, \emph{Exponential sums: questions by {D}enef, {S}perber, and {I}gusa},
  Trans. Amer. Math. Soc. \textbf{362} (2010), no.~7, 3745--3756.

\bibitem{CMN}
R.~Cluckers, M.~Musta\c{t}\u{a}, and K.~H. Nguyen, \emph{I{gusa}'{s}
  {conjecture} {for} {exponential} {sums}: {optimal} {estimates} {for}
  {nonrational} {singularities}}, Forum Math. Pi \textbf{7} (2019), e3.
  \MR{3987303}

\bibitem{CluckerNguyen}
R.~Cluckers and K.~H. Nguyen, \emph{Combining {I}gusa's conjectures on
  exponential sums and monodromy with semicontinuity of the minimal exponent},
  Algebra Number Theory \textbf{18} (2024), no.~7, 1275--1303. \MR{4757306}

\bibitem{Denefdegree}
J.~Denef, \emph{On the degree of {I}gusa's local zeta function}, American
  Journal of Mathematics \textbf{109} (1987), 991--1008.

\bibitem{Denef91}
\bysame, \emph{Local zeta functions and euler characteristics}, Duke Math. J.
  \textbf{63} (1991), no.~3, 713--721.

\bibitem{DenefBour}
\bysame, \emph{Report on {I}gusa's local zeta function}, S\'eminaire Bourbaki
  \textbf{Vol. 1990/91, Exp. No.730-744} (1991), 359--386, Ast\'erisque
  201-203, http://wis.kuleuven.be/algebra/denef.html$\sharp$D2.

\bibitem{DLJAG}
J.~Denef and F.~Loeser, \emph{Motivic {I}gusa zeta functions}, J. Algebraic
  Geom. \textbf{7} (1998), no.~3, 505--537.

\bibitem{DL02}
\bysame, \emph{Lefschetz numbers of iterates of the monodromy and truncated
  arcs}, Topology \textbf{41} (2002), no.~5, 1031--1040. \MR{1923998}

\bibitem{SGA3I}
Philippe Gille and Patrick Polo (eds.), \emph{Sch\'emas en groupes ({SGA} 3).
  {T}ome {I}. {P}ropri\'et\'es g\'en\'erales des sch\'emas en groupes},
  annotated ed., Documents Math\'ematiques (Paris) [Mathematical Documents
  (Paris)], vol.~7, Soci\'et\'e{} Math\'ematique de France, Paris, 2011,
  S\'eminaire de G\'eom\'etrie Alg\'ebrique du Bois Marie 1962--64. [Algebraic
  Geometry Seminar of Bois Marie 1962--64], A seminar directed by M. Demazure
  and A. Grothendieck with the collaboration of M. Artin, J.-E. Bertin, P.
  Gabriel, M. Raynaud and J-P. Serre. \MR{2867621}

\bibitem{GWZ}
M.~Groechenig, D.~Wyss, and P.~Ziegler, \emph{Mirror symmetry for moduli spaces
  of higgs bundles via p-adic integration}, Invent. Math. \textbf{221} (2020),
  505--596.

\bibitem{Haris}
S.~J. Haris, \emph{Number theoretical developments arising from the {S}iegel
  formula}, Bull. Amer. Math. Soc. (N.S.) \textbf{2} (1980), no.~3, 417--433.
  \MR{561526}

\bibitem{Harts}
R.~Hartshorne, \emph{Algebraic geometry}, Graduate Texts in Mathematics, No.
  52, Springer-Verlag, New York-Heidelberg, 1977. \MR{0463157}

\bibitem{Hir:Res}
H.~Hironaka, \emph{Resolution of singularities of an algebraic variety over a
  field of characteristic zero. {I}}, Ann. of Math. (2) \textbf{79} (1964),
  no.~1, 109--203.

\bibitem{IGUSA88}
Jun ichi Igusa, \emph{b-functions and p-adic integrals}, Algebraic Analysis
  (Masaki Kashiwara and Takahiro Kawai, eds.), Academic Press, 1988,
  pp.~231--241.

\bibitem{Igusa1}
J.~Igusa, \emph{Complex powers and asymptotic expansions {I}}, Journal f{\"u}r
  die reine und angewandte Mathematik \textbf{268/269} (1974), 110--130.

\bibitem{Igusa2}
\bysame, \emph{Complex powers and asymptotic expansions {II}}, Journal f{\"u}r
  die reine und angewandte Mathematik \textbf{278/279} (1975), 307--321.

\bibitem{Igusa3}
\bysame, \emph{Lectures on forms of higher degree (notes by {S}. {R}aghavan)},
  Lectures on mathematics and physics, Tata institute of fundamental research,
  vol.~59, Springer-Verlag, 1978.

\bibitem{LMVV}
E.~Le\'on-Cardenal, J.~Mart\'in-Morales, W.~Veys, and J.~Viu-Sos, \emph{Motivic
  zeta functions on {$\Bbb{Q}$}-{G}orenstein varieties}, Adv. Math.
  \textbf{370} (2020), 107192, 34. \MR{4100755}

\bibitem{Lichtin2}
B.~Lichtin, \emph{On a conjecture of {I}gusa {II}}, Amer. J. Math. \textbf{138}
  (2016), no.~1, 201--249. \MR{3462885}

\bibitem{Loeasym}
F.~Loeser, \emph{Une estimation asymptotique du nombre de solutions
  approch\'ees d'une \'equation {$p$}-adique}, Invent. Math. \textbf{85}
  (1986), no.~1, 31--38. \MR{842046}

\bibitem{Loemono}
\bysame, \emph{Fonctions d'{I}gusa {$p$}-adiques et polyn\^omes de
  {B}ernstein}, Amer. J. Math. \textbf{110} (1988), no.~1, 1--21. \MR{926736}

\bibitem{Marker}
D.~Marker, \emph{Model theory: an introduction}, Graduate Texts in Mathematics,
  vol. 217, Springer-Verlag, 2002.

\bibitem{Nguyennsd}
K.~H. Nguyen, \emph{On a uniform bound for exponential sums modulo $p^m$ for
  deligne polynomials}, arXiv:2111.11898 (2021).

\bibitem{MoiIde}
\bysame, \emph{Exponential sums and motivic oscillation index of arbitrary
  ideals and their applications}, Journal de Mathématiques Pures et
  Appliquées (2026), 103846.

\bibitem{NguyenLW}
\bysame, \emph{Uniform estimate of lang-weil type for counting points of
  schemes over finite rings}, http://arxiv.org/abs/2610.01748 (2026).

\bibitem{NguyenVeys}
K.~H. Nguyen and W.~Veys, \emph{On the motivic oscillation index and bound of
  exponential sums modulo {$p^m$} via analytic isomorphisms}, J. Math. Pures
  Appl. (9) \textbf{157} (2022), 211--242. \MR{4351077}

\bibitem{Perr}
M.~Perret, \emph{On the number of points of some varieties over finite fields},
  Bull. London Math. Soc. \textbf{35} (2003), no.~3, 309--320. \MR{1960941}

\bibitem{Sebagiso}
J.~Sebag, \emph{Jet schemes of quasi-homogeneous hypersurfaces and motivic
  monodromy conjecture for isolated quasi-homogeneous hypersurface
  singularities}, Bull. Belg. Math. Soc. Simon Stevin \textbf{28} (2021),
  no.~5, 689--708. \MR{4489194}

\bibitem{VeysZ}
W.~Veys and W.~A. Z{\'u}{\~n}iga-Galindo, \emph{Zeta functions for analytic
  mappings, log-principalization of ideals, and {N}ewton polyhedra}, Trans.
  Amer. Math. Soc. \textbf{360} (2008), no.~4, 2205--2227.

\bibitem{Wood-Yasuda-I}
Melanie~Machett Wood and Takehiko Yasuda, \emph{Mass formulas for local
  {G}alois representations and quotient singularities. {I}: a comparison of
  counting functions}, Int. Math. Res. Not. IMRN (2015), no.~23, 12590--12619.
  \MR{3431631}

\bibitem{JWright}
James Wright, \emph{On a conjecture of {I}gusa in two dimensions}, Amer. J.
  Math. \textbf{142} (2020), no.~4, 1193--1238. \MR{4124118}

\bibitem{Yas}
T.~Yasuda, \emph{The wild {M}c{K}ay correspondence and {$p$}-adic measures}, J.
  Eur. Math. Soc. (JEMS) \textbf{19} (2017), no.~12, 3709--3734. \MR{3730512}

\end{thebibliography}
\end{document}